\documentclass[11pt]{amsart}
\usepackage[foot]{amsaddr}
\usepackage{fullpage}
\usepackage[utf8]{inputenc}
\usepackage{mathtools}
\usepackage{amsmath,amsthm}
\usepackage{amssymb}
\usepackage{color}
\usepackage{bbm}
\usepackage{hyperref}
\usepackage{enumitem}
\numberwithin{equation}{section}
\allowdisplaybreaks

\newcommand\numberthis{\addtocounter{equation}{1}\tag{\theequation}}

\newtheorem{theorem}{Theorem}[section]
\newtheorem{corollary}[theorem]{Corollary}
\newtheorem{proposition}[theorem]{Proposition}
\newtheorem{lemma}[theorem]{Lemma}

\theoremstyle{definition}

\theoremstyle{remark}
\newtheorem{remark}[theorem]{Remark}

\title{\textbf{Unstable shock formation of the Burgers-Hilbert equation}}
\author{Ruoxuan Yang}
\address{Department of Mathematics, Massachusetts Institute of Technology, Cambridge, MA 02139} \email{rxyang@mit.edu}
\date{}

\begin{document}
\maketitle
\begin{abstract}
    This paper proves the existence of unstable shock solutions of the Burgers-Hilbert equation conjectured in \cite{Yang}. More precisely, we construct a smooth initial datum with finite $H^9$-norm such that the solution in self-similar coordinates is asymptotic to the first unstable solution to the self-similar inviscid Burgers equation. The blowup profile is a cusp with H\"older 1/5 continuity with explicit blowup time and location. Unlike the previously established stable shocks, the initial data cannot be taken in an open set; instead, we control the two unstable directions by Newton's iteration.
    % use an iterative method to find the exact values of the second and third derivative at $X=0$ of the initial data in the self-similar coordinates.
\end{abstract}

\section{Introduction}
The Burgers-Hilbert (BH) equation consists of an inviscid Burgers equation with a source term given by the Hilbert transform 
\begin{align}
    \partial_t u+u\partial_xu=H[u],\label{eq:BH}
\end{align}
where the Hilbert transform is defined for $f:\mathbb{R}\to\mathbb{R}$ by
\begin{align*}
    H[f](x):=\frac{1}{\pi}\mathrm{p.v.}\int_\mathbb{R}\frac{f(y)}{x-y}\,dy,\qquad \widehat{H[f]}(\xi)=-i\mathrm{sgn}(\xi)\hat{f}(\xi),
\end{align*}
where p.v.\ stands for principal value. We refer the readers to our previous paper \cite{Yang} for an overview of previous studies on Burgers-Hilbert equation, and the references \cite{Biello-Hunter, Bressan-Nguyen, Bressan-Zhang, C-C-G, Hunter, Hunter-Ifrim, Hunter-I-T-Wong, Hunter-andmore, Kenig, Krupa-Vasseur, Saut-Wang}.

In our previous paper \cite{Yang}, we proved the existence of stable shocks of the Burgers-Hilbert equation by constructing a smooth
initial datum with a finite $H^5$-norm such that the  spatial derivative of the solution
tends to infinity at a single point in finite time. The blowup profile was a cusp with H\"older 1/3 continuity, with explicit blowup time and location. To this end, we applied the modulated self-similar transformation, then constructed a solution asymptotically converge to the stable self-similar inviscid Burgers solution. The shock is stable in a sense that the initial data can be taken from an open set in $H^5(\mathbb{R})$.

During the preparation of the current paper, there have been two independent self-similar gradient blowup results for perturbations of the Burgers equation by Chickering-Moreno-Vasquez-Pandy \cite{Chickering} and Oh-Pasqualotto \cite{Oh}. In \cite{Chickering}, Chickering-Moreno-Vasquez-Pandy constructed asymptotic self-similar shock solutions to the fractal Burgers equation
\begin{align*}
    \partial_tu+u\partial_xu+(-\Delta)^\alpha u=0
\end{align*}
for $0<\alpha<\frac{1}{3}$, starting from smooth initial data. The shock is an $H^6$ perturbation of the stable self-similar Burgers profile, and the time, location, and regularity of this
shock can also be precisely computed. In \cite{Oh}, Oh-Pasqualotto established gradient blowup for dispersive and dissipative perturbations of the Burgers equation. More precisely, for the fractional KdV equation of order $\alpha$
\begin{align}
    \partial_tu+u\partial_xu+|D_x|^{\alpha-1}\partial_xu=0,\qquad \text{for any }\alpha\in[0,1),\label{eq:KdV}
\end{align}
the fractal Burgers equation of order $\beta$
\begin{align}
    \partial_t +u\partial_xu+|D_x|^\beta u = 0,\qquad\text{for any }\beta\in[0,1),\label{eq:fractal}
\end{align}
where $|D_x|^\beta=(-\Delta)^\frac{\beta}{2}$, and Whitham's equation
\begin{align}
\partial_tu+u\partial_xu+\Gamma(D_x)\partial_xu=0,\qquad\text{where }\Gamma(\xi)=\sqrt{\frac{\tanh{\xi}}{\xi}},\label{eq:Whitham}
\end{align}
the authors constructed solutions with a ``shock-like" singularity, i.e.\ the amplitude is bounded but the gradient blows up at a single point. Moreover, they provided an asymptotic description of the blowup. Again the smooth self-similar solutions to the inviscid Burgers equation were used as the asymptotic profiles, but perturbations could also be around unstable self-similar Burgers profiles (decribed below). The case when $0<\beta<\frac{2}{3}$ covered the result in \cite{Chickering}, but the techniques were different. Both \cite{Chickering} and \cite{Oh} used the idea of modulated self-similar transformation in \cite{Yang}, which was in turn inspired by \cite{B-S-V1},\,\cite{B-S-V2}, and originally came from \cite{Merle},\,\cite{Merle-Zaag} applied to non-linear dispersive equations. In the appropriate self-similar coordinates, in order to deal with the dissipative or dispersive terms, just as what we did in \cite{Yang}, both works separated the space variable into two regions, middle field and far field, and used bootstrap arguments to show that the space derivatives of the solutions have spatial decay in the middle field and temporal decay in the far field.

It is well known that the invisicid Burgers equation 
\begin{align}
    \partial_tu+u\partial_xu=0\label{eq:invisicd}
\end{align}
has a family of self-similar solutions 
\begin{align}
    u(x,t)=(-t+T_*)^\frac{1}{2i}U_i\bigg(\frac{x-x_*}{(-t+T_*)^{\frac{1}{2i}+1}}\bigg)\label{eq:family}
\end{align}
for $i=1,2,3,...$, where $T_*$ is the blowup time, $x_*$ is the blowup location, and each $U_i\in\mathcal{C}^\infty(\mathbb{R})$ solves 
\begin{align*}
    -\frac{1}{2i}U_i+\bigg(\frac{2i+1}{2i}X+U_i\bigg)U_i'=0,
\end{align*}
where $X\in\mathbb{R}$ is the variable for $U_i$. Furthermore, if the initial datum $u_0$ to \eqref{eq:invisicd} has minimum slope at $x_*$, and satisfies 
\begin{align*}
   u_0(0)=0,\ u_0'(0)=-1,\ u_0^{(j)}(0) = 0\text{ for }j=2,...,2i,\ u_0^{(2i+1)}(0)=\nu >0,
\end{align*}
then the solution $u$ blows up at $T_*=1$, and as $t\to T_*$, 
\begin{align}
    u(x,t)\to \bigg(\frac{\nu}{(2i+1)!}\bigg)^{-\frac{1}{2i}}(T_*-t)^\frac{1}{2i}U_i\bigg(\Big(\frac{\nu}{(2i+1)!}\Big)^\frac{1}{2i}\frac{x}{(T_*-t)^\frac{2i+1}{2i}}\bigg).\label{eq:Burgersunified}
\end{align}
In some suitable function space, only $U_1$ is stable as a fixed point in some dynamical system, all the $U_i$ for $i\geq 2$ are unstable fixed points (see \cite{Eggers2}). Moreover, the set of initial conditions that yield solutions approaching $U_i$ for $i\geq 2$ is located at the boundary of the set of initial conditions leading
to solutions approaching $U_1$, and admits $2i-2$ instability directions yielding to shocks formed by $U_j$ for $j<i$. The blowup of the solution given by \eqref{eq:family} occurs in the H\"older space $\mathcal{C}^\delta(\mathbb{R})$ for all $\delta>\frac{1}{2i+1}$. For more detailed discussions, see \cite{Collot}, \cite{Eggers2} and \cite{Eggers}, Chapter 11.

In \cite{Yang} and \cite{Chickering}, solutions were constructed as perturbations of $U_1$ only. In \cite{Oh}, multiple unstable $U_i$'s were also used, in fact, for given $\alpha,\beta\in[0,1)$, for each $i\in\mathbb{N}$ such that $\alpha,\beta<\frac{2i}{2i+1}$, there is a gradient blowup solution associated with $U_i$\footnote{It is expected that in appropriate self-similar coordinates the solutions converge to $U_i$ on compact sets of $X$, but the authors did no carry out the details.}, hence the more unstable with respect to initial data perturbations (the bigger the $i$), the more stable with respect to perturbations of the Burgers equation (the bigger $\alpha$ and $\beta$ are allowed). In the stable case $i=1$, the initial data can be taken from an open subset of $H^5$ (which agrees with \cite{Yang}), and in the unstable case $i\geq 2$, the initial data form a ``codimension $2i-2$ subset" of $H^{2k+3}$. In the self-similar coordinates $(U,X,s)$\footnote{In \cite{Oh}, $y$ is used in place of $X$, and the orders of space and time is switched.}, modulation constraints are 
\begin{align*}
    U(0,s)=\partial_XU(0,s)+1=\partial_X^{2i}U(0,s)=0\qquad \text{for all }s,
\end{align*}
and the $(2i-2)$ unstable directions were controlled by selecting the values of the $j$-th derivatives of the initial data at $X=0$ for $j=2,...,2i-1$. The values were found by a topological argument well known in the dispersive community, namely, a trapping condition combined with  Brouwer's fixed point theorem. Such a topological argument relies on contradiction, so the values of the $(2i-2)$ derivatives of the initial data were not explicit. To achieve the sharp spatial growth of solutions in suitable self-similar coordinates (and thus the H\"older continuity of solutions in physical coordinates), instead of solely relying on the Lagrangian trajectory-based method in \cite{Yang}, Oh-Pasqualotto also applied a streamlined, sharp weighted $L^2$-based method to space derivatives of order $2i+2$ and $2i+3$ (``top order derivatives"). 
Such a weighted $L^2$-based approach, in addition to the Lagrangian trajectories for the spatial behavior of solutions in self-similar coordinates, together with the topological argument, works for all suitable $i$, so only one proof is carried out simultaneously for all suitable $i$'s.

Another gradient blowup result that used all $U_i$'s as blowup profiles via weighted $L^2$ estimates and a topological argument is \cite{Collot}, for 2D Burgers equation with transverse viscosity
\begin{align*}
\partial_tu+u\partial_xu-\partial_{yy}^2u&=0,\qquad (x,y)\in\mathbb{R}^2,\\
    u\big|_{t=0}&=0.
\end{align*}
In \cite{Collot}, for each $i\in\mathbb{N}$, if the initial datum satisfies $\partial_x^ju_0(0,y)=0$ for all $y\in\mathbb{R}$ and $j=2,...,2i$, then this remains true for later times, and the
traces of the derivatives 
\begin{align*}
    \xi(t,y):=-\partial_xu(t,0,y),\quad\text{and}\quad\zeta(t,y):=\partial_x^{2i+1}u(t,0,y)
\end{align*}
solve the parabolic system
\begin{equation*}
    \begin{cases}
    \partial_t\xi-\xi^2-\partial_{yy}^2\xi=0,\\
    \partial_t\zeta -(2i+2)\xi\zeta-\partial_{yy}^2\zeta=0.
    \end{cases}
\end{equation*}
Since $i$ enters the reduced system only as a coefficient, the proof can be done for all $i\in\mathbb{N}$ at once.                                            

On the other hand, in the context of 2D isentropic compressible Euler equation \cite{B-I}, Buckmaster-Iyer constructed unstable shocks by using the first unstable self-similar Burgers solution $U_2$ as the blowup profile and using only Lagrangian trajectories, and the values of 2nd and 3rd derivative of the initial data at $X=0$ were found via Newton's iteration, i.e.\ one obtain a sequence from the iterative steps, and the targeting value is the limit of the sequence. This method is more explicit at the cost of generality, i.e.\ one has to introduce different sets of bootstrap assumptions for different $i$ and re-prove the theorem; moreover, the bigger the $i$, the more bootstrap assumptions are required. The amount of calculations
becomes very large even with moderate values of $i$, which is the limitation of
this method.

In \cite{Yang}, we conjectured unstable shock formation of the Burgers-Hilbert equation, and this paper presents the proof. Here we only give a rough statement of our main result. The precise statement can be found in Theorem \ref{thm:main} below.
\begin{theorem}[rough statement of the main result]
There exists a smooth initial datum with finite $H^9$-norm and minimum initial slope $-1/\epsilon$ for $0<\epsilon\ll 1$, such that the corresponding solution to the Burgers-Hilbert equation \eqref{eq:BH} forms a shock in $O(\epsilon)$ time. The blowup profile is a cusp with (sharp) H\"older 1/5 continuity, and the blowup time and location are explicitly computable. The shock formation is unstable in the sense that the set of initial data leading to such blowup has codimension 2 in $H^9(\mathbb{R})$.
\end{theorem}

\begin{remark}\label{rem:adhoc}
Note that the fractional KdV equation of order $\alpha=0$ is exactly the Burgers-Hilbert equation. But here we largely follows our previous work \cite{Yang}, the arguments are solely Lagrangian trajectory-based, $L^2$-estimates are only for pointwise estimates of the Hilbert transform in the forcing terms, so our approach is different from \cite{Oh}. Here we only use $U_2$ as the blowup profile. However, we are able to prove that after a suitable modulated self-similar transformation \eqref{eq:selftransform}, the solution $U$ in self-similar coordinates $(X,s)$ converges to $U_2$ in $s$-dependent compact subsets of $X$ (such details were not carried out in \cite{Oh}, because their weighted $L^2$-estimates and Lagrangian trajectory estimates were carried out on the solution directly rather than the difference with $U_2$). We adopt the iteration method illustrated in \cite{B-I} to explicitly find the initial values of $\partial_X^2U(0,s)$ and $\partial_X^3U(0,s)$. So our method is more explicit but also more ad hoc.
\end{remark}

\begin{remark}
Here the initial data are taken in $H^9(\mathbb{R})$, i.e.\ two orders of derivatives higher than those in \cite{Oh}. In \cite{Oh}, a bound for $\|\partial_X^7U\|_{L^2}$ is needed so that $\partial_X^5U(0,s)$ can be bounded by Sobolev embedding (and weighted $L^2$-estimates have no loss of derivatives). Here we need to go to $H^9$ due to the iteration method. For convenience, let the parameters $\alpha,\,\beta$ be the initial values of $\partial_X^2U(0,s)$ and $\partial_X^3U(0,s)$ we are searching for, then the solution $U$ a priori varies smoothly with respect to $\alpha,\,\beta$. Hence, we consider the equations for $\partial_\alpha\partial_X^nU$ and $\partial_{\alpha\beta}^2\partial_X^nU$. We need up to $\partial_{\alpha\beta}^2\partial_X^5U$ to bound the second order parameter derivatives of the modulation variables (which are necessary to close the bootstrap), but to bound $\|\partial_{\alpha\beta}^2\partial_X^5U\|_{L^\infty}$ we need an $L^2$-bound for $\partial_{\alpha\beta}^2\partial_X^6U$ to control $\|H[\partial_{\alpha\beta}^2\partial_X^6U]\|_{L^\infty}$. To control the forcing term in the $L^2$-inner product equation, we need $\|\partial_\alpha\partial_X^7U\|_{L^\infty}$, which in turns requires $\|\partial_X^8U\|_{L^\infty}$, and hence $\|\partial_X^9U\|_{L^2}$ due to the loss of derivative on the Hilbert transform.
\end{remark}

\begin{remark}[future directions]
As we discussed in Remark \ref{rem:adhoc}, the bootstrapping plus iteration method cannot find shock solutions to the Burgers-Hilbert equation in a unified fashion. In contrast, the method in \cite{Oh} is more general but is not as explicit. One future direction is to combine or modify the two methods in order to achieve something similar to the conclusion \eqref{eq:Burgersunified} for \eqref{eq:BH}, and for dispersive and dissipative modified Burgers equations~\eqref{eq:KdV}-\eqref{eq:Whitham}. We anticipate that the tools we use here and in \cite{Yang} to overcome the nonlocality and lack of pointwise boundedness of the Hilbert transform will also be useful for the resolution of this potential problem.

Here, and in \cite{B-I, B-S-V1, B-S-V2, Chickering, Oh, Saut-Wang, Yang}, the magnitude of the slope or gradient of the initial data has to be large at a single point. For the Burgers-Hilbert equation, this is not very surprising, since the convection $uu_x$ is already too strong to be depleted by the non-smoothing Hilbert transform term $H[u]$. It is more interesting and of course more difficult, to construct shock solutions starting with a moderate slope, or even better, a small slope. Or one can try to prove that small-slope initial data like the one in \cite{Biello-Hunter} indeed form a singularity, shock or non-shock, with an explicit description, and prove that the enhanced lifespan obtained in \cite{Hunter-Ifrim, Hunter-I-T-Wong} is sharp. The small slope initial data regime is physically more realistic, since the Burgers-Hilbert equation is a good approximation to the motion of a vorticity front when the slope is small (see \cite{Biello-Hunter, Hunter}); in fact, the large initial slope regime is not covered by this approximation. 
% Even if we can indeed construct a small slope smooth initial datum such that the solution to the Burgers-Hilbert equation forms a finite time singularity, as we said before in Section~\ref{ch1:shock}, we need to relate the analytic result to the physical system.

Another question is the continuation and propagation of shock solutions. Buckmaster-Shkoller-Vicol \cite{B-S-V1} use the so-called Lax-Olienik formula to give a unique continuation of the 2-dimensional compressible Euler's equation after the shock. The explicit calculations are performed for adiabatic  constant $\gamma=3$ when the shock solution only has velocity in the angular component but not in the radial component in the polar coordinate. The speed of movement of the shock and the line of jump discontinuity of the solution are given. One may wish to obtain such a continuation for the Burgers-Hilbert equation, but the Hilbert transform in \eqref{eq:BH} makes explicit calculations very difficult.
\end{remark}

\subsection{Strategy of the proof}
Since the Burgers-Hilbert equation is translation invariant in time, we take the initial time $t_0=-\epsilon$ for convenience. Since the asymptotic blowup profile is $U_2$, we perform the self-similar transformation in \eqref{eq:selftransform} to change the physical coordinates $(x,t)$ to the self-similar coordinates $(X,s)$, and map $u$
to $U$ by \eqref{eq:uUrelation} accordingly. We then rewrite the Burgers-Hilbert equation into \eqref{eq:ansatz} for $U$ in self-similar coordinates.

The $\tau,\,\xi,\,\kappa$ are dynamic modulation variables: the function $\tau(t)$ tracks the blowup
time, $\xi(t)$ tracks the movement of the shock location, and $\kappa(t)=u\big(\xi(t),t\big)$ controls the amplitude at the shock location. We impose the constraints \eqref{eq:constraint} on $U(0,s),\,\partial_XU(0,s)$ and $\partial_X^4U(0,s)$, then we get the equations \eqref{eq:eqkappadot}-\eqref{eq:kappa-xi} for $\tau,\,\xi,\,\kappa$. Together with the initial conditions \eqref{eq:modulationinitial}, we can solve for $\tau,\,\xi,\,\kappa$ explicitly.

In self-similar coordinates, the blowup time $T_*$ corresponds to $s=+\infty$ and the
shock location $x_*:=\xi(T_*)$ corresponds to $X = 0$. Shock formation at $x_*$ corresponds
to the constraint $\partial_XU(0,s)=-1$, and the regularity at other points corresponds to the spatial and
temporal decays of $\partial_XU(X,s)$. The instability of the self-similar shock and the $\mathcal{C}^{\frac{1}{5}}$ regularity of the blowup profile corresponds to $U(X,s)\to U_2$ on $\mathcal{C}^1\big([-\frac{1}{2}e^{\frac{5}{4}s},\frac{1}{2}e^{\frac{5}{4}s}]\big)$ and $\partial_X^2U(0,s)\to 0,\,\partial_X^3U(0,s)\to 0$, $\partial_X^5U(0,s)\to U_2^{(5)}(0)=120$ as $s\to+\infty$.

To ensure $\partial_X^2U(0,s)\to 0,\,\partial_X^3U(0,s)\to 0$ as $s\to+\infty$, we must find the exact values of\footnote{For simplicity of this outline, the definition of $\alpha,\,\beta$ here is different from the actual definition in \eqref{eq:initialdata}, hence the difference of $\alpha_0,\,\beta_0$, but the two only differ by a constant, hence the parameter derivatives are essentially the same. The actual definition \eqref{eq:initialdata} is to ensure that the solution varies smoothly with respect to $\alpha,\,\beta$.}
\begin{align*}
    \alpha:=\partial_X^2U(0,-\log\epsilon),\qquad\beta:=\partial_X^3U(0,-\log\epsilon),
\end{align*}
and hence the set of initial data has codimention $2$ in $H^9(\mathbb{R})$. This is done by Newton's iteration method: let $s_n=-\log\epsilon+n$ for $n=0,1,2,...$. We start with $\alpha_0=0$, $\beta_0=0$, so that
\begin{align*}
    \partial_X^2U(0,s_0)=0,\qquad \partial_X^3U(0,s_0)=0.
\end{align*}
Suppose we can find $(\alpha_n,\beta_n)$ such that 
\begin{align*}
\partial_X^2U_{\alpha_n,\beta_n}(0,s_n)=\partial_X^3U_{\alpha_n,\beta_n}(0,s_n)=0,
\end{align*}
where $U_{\alpha,\beta}$ denotes the solution with initial conditions
\begin{align*}
    \partial_X^2U(0,-\log\epsilon)=\alpha,\qquad\partial_X^3U(0,-\log\epsilon)=\beta,
\end{align*}
then we prove that we can find $(\alpha_{n+1},\beta_{n+1})\in B_n(\alpha_n,\beta_n)$ defined in \eqref{eq:sizealphabetan} such that at time $s_{n+1}=s_n+1$,
\begin{align*}
\partial_X^2U_{\alpha_{n+1},\beta_{n+1}}(0,s_{n+1})=\partial_X^3U_{\alpha_{n+1},\beta_{n+1}}(0,s_{n+1})=0.
\end{align*}
The existence of $\alpha_{n+1},\,\beta_{n+1}$ is guaranteed by a non-singular total derivative matrix and a bounded Hessian matrix. The sequences $\{\alpha_n\}$ and $\{\beta_n\}$ are Cauchy sequences. As $n\to\infty$, $\alpha_n\to \alpha_*$, $\beta_n\to\beta_*$, and
\begin{align*}
    \lim_{s\to+\infty} \partial_X^2U_{\alpha_*,\beta_*}(0,s)=\lim_{s\to+\infty}\partial_X^3U_{\alpha_*,\beta_*}(0,s)=0,
\end{align*}
so we obtain our targeting values $\alpha_*,\,\beta_*$.

Thus Theorem \ref{thm:main} follows directly from Theorem \ref{thm:selfsimilar}.
The proof of Theorem \ref{thm:main} utilizes a bootstrap argument applied to $\partial_X^nU,\,\partial_\alpha\partial_X^nU$ and $\partial_{\alpha\beta}^2\partial_X^nU$. The pointwise convergence is not part of the bootstrap since we will prove it a posteriori. We isolate the steps regarding estimates at $X=0$ at the end. For other estimates, we only describe the steps briefly, and the readers can refer to \cite{Yang} for more details. 

\emph{Without parameter derivatives}:
\begin{enumerate}
    \item We control $\|\partial_XU\|_{L^2}$ and $\|\partial_X^9U\|_{L^2}$ uniformly in $s$ by taking $L^2$-inner products of the equations.
    \item We control the growth of $\|U(\cdot,s)+e^{\frac{1}{4}s}\kappa\|_{L^\infty}$ by a transport estimate on the equation of $e^{-\frac{1}{4}s}U(X,s)+\kappa$.
    \item We prove that $U(X,s)$ and $\partial_XU(X,s)$ are close to $U_2$ for $|X|\leq \frac{1}{2}e^{\frac{5}{4}s}$, and the spatial decays of $\partial_X^nU(X,s)$, $n=2,...,8$. We consider equations on $(1+X^4)^{-\frac{1}{20}}\widetilde{U}$, $(1+X^4)^\frac{1}{5}\partial_X\widetilde{U}$ where $\widetilde{U}:=U-U_2$, and $(1+X^4)^\frac{1}{5}\partial_X^nU$, $n=2,...,8$. We use the repelling property of the Lagrangian trajectories to do transport estimates in the middle field $l\leq|X|\leq \frac{1}{2}e^{\frac{5}{4}s}$, and we deal with the near field $|X|\leq l$ differently because the Lagrangian trajectories are not necessarily repelling near $X=0$.
    \item We prove the temporal decay of $\partial_X^nU(X,s)$ in the far field $|X|\geq \frac{1}{2}e^{\frac{5}{4}s}$ for $n=1,...,8$ by considering the equations of $e^s\partial_X^nU$ composed with Lagrangian trajectories.
    \item We close the bootstrap assumptions on modulation variables $\tau,\,\xi,\,\kappa$ using their explicit ODEs. 
\end{enumerate}

\emph{First order parameter derivatives}: Here the estimates hold for all $\alpha,\,\beta$ in a relevant parameter ball $B$ defined in \eqref{eq:sizealphabeta}, and we don't distinguish between $\alpha$ and $\beta$ so we only write $\partial_\alpha$. These estimates themselves are not crucial, but they are necessary for establishing the non-singular total derivative matrix in the iteration step at $X=0$ below.
\begin{enumerate}
    \item We control $\|\partial_\alpha U(\cdot,s)+e^{\frac{1}{4}s}\partial_\alpha\kappa\|_{L^2}$ by taking $L^2$-inner product with the equation of $e^{-\frac{1}{4}s}\partial_\alpha U+\partial_\alpha\kappa$ (because we lose the $L^2$-conservation in physical coordinates). And similarly, we control $\|\partial_\alpha\partial_XU(\cdot,s)\|_{L^2}$ and $\|\partial_\alpha\partial_X^8U(\cdot,s)\|_{L^2}$. 
    \item We give a pointwise bound of $\partial_\alpha U(X,s)+e^{\frac{1}{4}s}\partial_\alpha\kappa$ by considering the equation of $e^{-\frac{1}{4}s}\partial_\alpha U+\partial_\alpha\kappa$. We separate the regions $|X|\leq \frac{1}{2}e^{\frac{5}{4}s}$ and $|X|\geq \frac{1}{2}e^{\frac{5}{4}s}$ due to the different decays of $\partial_XU$ in the two regions.
    \item We give spatial decay and temporal growth of $\partial_\alpha\partial_X^n U(X,s)$, $n=0,...,7$ for $|X|\leq\frac{1}{2}e^{\frac{5}{4}s}$. Similarly as without parameter derivatives, we carry out weighted transport estimates and deal with the near field separately.
    \item We give temporal decay of $\partial_\alpha\partial_X^nU(X,s)$, $n=1,...,7$. 
    \item We close the bootstrap assumptions on the first order parameter derivatives of the modulation variables.
\end{enumerate}

\emph{Second order parameter derivatives}:
Again the estimates hold for all $\alpha,\,\beta$ in the parameter ball \eqref{eq:sizealphabeta}, and we don't distinguish among $\partial^2_{\alpha\alpha},\,\partial^2_{\alpha\beta}$ and $\partial_{\beta\beta}^2$. The purpose of these estimates is to give an upper bound of the Hessian matrix in the iteration step. The argument becomes simpler since we only need $L^\infty$-bounds which are less sharp.
\begin{enumerate}
    \item We control $\|\partial_{\alpha\beta}^2U(\cdot,s)+e^{\frac{1}{4}s}\partial_{\alpha\beta}^2\kappa\|_{L^2}$, $\|\partial_{\alpha\beta}^2\partial_XU(\cdot,s)\|_{L^2}$ and $\|\partial_{\alpha\beta}^2\partial_X^7U(\cdot,s)\|_{L^2}$.
    \item We establish the temporal growths of $\|\partial_{\alpha\beta}^2U(\cdot,s)+e^{\frac{1}{4}s}\partial_{\alpha\beta}^2\kappa\|_{L^\infty}$ and $\|\partial_{\alpha\beta}^2\partial_X^nU(\cdot,s)\|_{L^\infty}$, $n=1,...,6$ by transport estimates on all $X$. We are able to do this because the damping terms, even though not always strictly positive, are all $\geq -\frac{5}{4}$, and $\frac{5}{4}\leq \frac{3}{2}$, hence the $e^{\frac{3}{2}s}$ growth of the forcing terms dominate. 
    \item We close the bootstrap assumptions on the second order parameter derivatives of the modulation variables.
\end{enumerate}

\emph{Estimates at $X=0$, and finding $\alpha,\,\beta$}: 
\begin{enumerate}
    \item \emph{Without parameter derivatives}: We first establish the intermediate step, the decays of $\partial_X^2U_{\alpha_n,\beta_n}(0,s),\,\partial_X^3U_{\alpha_n,\beta_n}(0,s)$ for $s_n\leq s\leq s_{n+1}$, in order to establish the (slightly slower) decays of $\partial_X^2U_{\alpha,\beta}(0,s),\,\partial_X^3U_{\alpha,\beta}(0,s)$ for all $s\geq -\log\epsilon$ and all $(\alpha,\beta)\in B$. The intermediate step is done by integrating the ODEs of $\partial_X^2U(0,s)$ and $\partial_X^3U(0,s)$, and then we use the mean value theorem to get the decay for all $s,\,\alpha,\,\beta$. Finally we show $\partial_X^5\widetilde{U}(0,s)$ is small by integrating its ODE. 
    \item \emph{First parameter derivatives}: Here we distinguish between $\alpha$ and $\beta$. We establish upper and lower bounds of $\partial_\alpha\partial_X^2U(0,s),\,\partial_\beta\partial_X^3U(0,s)$, i.e.\ the diagonal entries of the total derivative matrix in the iteration, bounds on $\partial_\alpha\partial_X^3U(0,s),\,\partial_\beta\partial_X^2U(0,s)$, i.e.\ the anti-diagonal entries. Therefore, the total derivative matrix has strictly positive determinant. We achieve these bounds by integrating the corresponding ODEs. Lastly we bound the growth of $\partial_\alpha\partial_X^5U(0,s)$, which is needed for the modulation variables.
    \item    \emph{Finding $\alpha,\,\beta$ by Newton's iteration}: Given $\alpha_n,\,\beta_n$, we treat $(\alpha_{n+1},\beta_{n+1})$ as the root to the system defined by the Taylor expansions \eqref{eq:inductionTayloralpha},\,\eqref{eq:inductionTaylorbeta}. From the previous steps, in a small neighbourhood $B_n$ around $(\alpha_n,\beta_n)$, the functions are bounded, the determinant of the derivatives is strictly positive, and the Hessian matrix is bounded, so by Newton's iteration method we can find the root $(\alpha_{n+1},\beta_{n+1})\in B_n$. Then we close the bootstrap assumptions on the size of $B$ and $B_n$ for $n=0,1,2,...$.
\end{enumerate}

\emph{Pointwise convergence of $U$ to a rescaled $U_2$}: After closing the bootstrap, we prove the pointwise convergence of $U(X,s)$
to $U_2^\nu$, a rescaled version of $U_2$ defined in Theorem \ref{thm:selfsimilar}. The rescaling is for matching
the fifth order spatial derivative at $X = 0$. To do so, we consider the equation for the difference $\widetilde{U}^\nu:=U-U_2^\nu$, then use the Taylor expansion of $\widetilde{U}^\nu$ at $X = 0$ and the (new) Lagrangian
trajectory which is repelling for all $X_0\neq 0$ to propagate to all $X\neq 0$.

\emph{Codimension 2 subset of $H^9(\mathbb{R})$}: Finally we show that the initial data $u_0$ can be taken from a small codimension 2 subset in $H^9(\mathbb{R})$: for each such $u_0$, by coordinate translation and rescaling if necessary, there exist a $\widehat{U}_0$ and unique $\alpha,\,\beta$ as ``second and third order derivative corrections at $X=0$".

\subsection{Paper outline}
In Section \ref{sec:preli} we introduce the unstable Burgers profile $U_2$, carry out the self-similar transformation, impose the constraints on the solution, and write down the equations we need later. In Section \ref{sec:main} we describe the initial datum and the solution by listing the bootstrap assumptions, and we provide a precise statement of our main result, Theorem \ref{thm:main} and Corollary \ref{cor:codimension}, both of which follow from the result in self-similar coordinates, Theorem \ref{thm:selfsimilar}. In Section \ref{sec:forcing} and \ref{sec:paraforce}, we calculate some estimates for forcing terms. In Section \ref{sec:closure}-\ref{sec:2ndparaclosure}, we close the bootstrap assumptions except for those at $X=0$, which are closed in Section \ref{sec:close0}. In Section \ref{sec:findpara} we find the parameters $\alpha,\,\beta$. In Section \ref{sec:proof} we prove Theorem \ref{thm:selfsimilar}, Theorem \ref{thm:main} and Corollary \ref{cor:codimension}.

\subsection*{Acknowledgments}
Since this is a follow up paper from \cite{Yang}, the author would like to thank Tristan Buckmaster again for suggesting the problem and explaining the differences between stable and unstable shocks. The author would also like to thank Gigliola Staffilani for helpful ideas and writing suggestions. The author was partially supported by the National
Science Foundation under Grant No. DMS-1764403, and was later supported by MIT Department of Mathematics under Graduate Student Appreciation Fellowship. 

\section{Self similarity}\label{sec:preli}
\subsection{The first unstable self-similar Burgers profile}
Recall the first unstable self-similar Burgers solution $U_2$ which solves
\begin{align}
-\frac{1}{4}U_2+\big(U_2+\frac{5}{4}X\big)U_2'=0,\quad U_2(0)=0\label{eq:U2},
\end{align}
so that the function
\begin{align*}
u(x,t):=(T_*-t)^\frac{1}{4}U_2\bigg(\frac{x-x_*}{(T_*-t)^\frac{5}{4}}\bigg)
\end{align*}
solves the inviscid Burgers equation and forms a self-similar shock at $(x_*,T_*)$. 

Unlike the stable solution $U_1$, the unstable solution $U_2$ does not have a closed-form expression. However, by \cite{Collot}, Chapter 11 of \cite{Eggers} and elementary calculus, we are still able to obtain the quantitative properties we need, similar to those in \cite{Yang}.

\begin{lemma}[quantitative properties of $U_2$]\label{lem:U2}
The unstable self-similar Burgers solution $U_2\in\mathcal{C}^\infty(\mathbb{R})$ is odd, decreasing, and concave on $(-\infty,0]$, and it satisfies the implicit equation 
\begin{align*}
    X=-U_2(X)-U_2(X)^5.
\end{align*}
At $X=0$, $U_2^{(n)}\neq 0$ only for $n=1,\,5,\,9,...$, in particular,
\begin{align}
    U_2'(0)=-1,\qquad U_2^{(5)}(0)=120.\label{eq:U2(0)}
\end{align}
As $X\to\pm\infty$, $U_2$ has the asymptotic expansion
\begin{align*}
    U_2(X)=-\mathrm{sgn}(X)|X|^\frac{1}{5}+\frac{\mathrm{sgn}(X)}{5}|X|^{-\frac{3}{5}}+O(|X|^{-\frac{7}{5}}).
\end{align*}
The Taylor expansions for $U_2'$ are
\begin{align*}
    \text{for }X\approx 0:\qquad U_2'(X)&=-1+5X^4+O(X^8),\\
    \text{for }X\to\pm\infty:\qquad U_2'(X)&=-\frac{1}{5}|X|^{-\frac{4}{5}}-\frac{4}{25}|X|^{-\frac{8}{5}}+O(|X|^{-\frac{12}{5}}).
\end{align*}
Moreover, we have the following bounds
\begin{align}
    |U_2(X)|\leq (1+X^4)^\frac{1}{20},\qquad |U_2'(X)|\leq (1+X^4)^{-\frac{1}{5}}.\label{eq:U2bounds}
\end{align}
For all $0<l<0.2$,
\begin{align}
     0>U_2'(X)\geq -(1-2l^4)(1+X^4)^{-\frac{1}{5}}\qquad\text{for }|X|\geq l.\label{eq:U2away0}
\end{align}
And for $|X|\geq 100$,
\begin{align}
    -\frac{7}{40}(1+X^4)^{-\frac{1}{5}}>U_2'(X)>-\frac{9}{40}(1+X^4)^{-\frac{1}{5}}.\label{eq:U2far}
\end{align}
\end{lemma}

\subsection{Modulated self-similar transformation}
Define
\begin{align}
X:=\frac{x-\xi(t)}{(\tau(t)-t)^\frac{5}{4}},\qquad s:=-\log{\big(\tau(t)-t\big)},\label{eq:selftransform}
\end{align}
and define the self-similar variable $U$ by
\begin{equation}
\begin{aligned}
u(x,t)&=\big(\tau(t)-t\big)^\frac{1}{4}U\Big(\frac{x-\xi(t)}{(\tau(t)-t)^\frac{5}{4}},-\log\big(\tau(t)-t\big)\Big)+\kappa(t)\\
&=e^{-\frac{1}{4}s}U(X,s)+\kappa(t),
\end{aligned}\label{eq:uUrelation}
\end{equation}
where the dynamic modulation variables
$\xi,\,\tau,\kappa:[-\epsilon,T_*]\to\mathbb{R}$ control the shock
location, blowup time, and wave amplitude, respectively. We require that at time $t_0=-\epsilon$,
\begin{align}
    \tau_0:=\tau(-\epsilon)=0,\qquad\xi_0:=\xi(-\epsilon)=0,\qquad\kappa_0:=\kappa(-\epsilon)=u_0(0).\label{eq:modulationinitial}
\end{align}
And we let $T_*$ to be the solution to 
$\tau(T_*)=T_*$. We will design $u_0$ so that $T_*\sim o(\epsilon)$ and $\tau(t)>t$ for all $t\in[-\epsilon,T)$, so that $T_*$ uniquely exists. In the self-similar time $s$, the blowup as $t\to T_*$ corresponds to $s\to+\infty$.

From the self-similar transformation \eqref{eq:selftransform}, we have the identities
\begin{align*}
\tau(t)-t=e^{-s},\quad x=\big(\tau(t)-t\big)^\frac{5}{4}X+\xi(t)=e^{-\frac{5}{4}s}X+\xi(t).
\end{align*}
From \eqref{eq:uUrelation}, we get the Burgers-Hilbert equation in self-similar form
\begin{align}
\big(\partial_s-\frac{1}{4}\big)U+\Big(\frac{U+e^{\frac{1}{4}s}(\kappa-\dot{\xi})}{1-\dot{\tau}}+\frac{5}{4}X\Big)\partial_XU=-\frac{e^{-\frac{3}{4}s}\dot{\kappa}}{1-\dot{\tau}}+\frac{e^{-s}}{1-\dot{\tau}}H[U+e^{\frac{1}{4}s}\kappa].\label{eq:ansatz}
\end{align}
As a consequence of \eqref{eq:uUrelation}, we have
\begin{align*}
\partial_x^ju(x,t)&=e^{(\frac{5}{4}j-\frac{1}{4})s}\partial_X^jU(X,s),\\
\|\partial_x^ju(\cdot,t)\|_{L^2}^2&=e^{(\frac{5}{2}j-\frac{7}{4})s}\|\partial_X^jU(\cdot,s)\|_{L^2}^2.
\end{align*}
In particular, since $\|u(\cdot,t)\|_{L^2}=\|u_0\|_{L^2}$, $\|U(\cdot,s)+e^{\frac{1}{4}s}\kappa\|_{L^2}=e^{\frac{7}{8}s}\|u_0\|_{L^2}$.

\subsection{Equations of $\partial^j_XU$ }
Here we record the equations of $\partial^j_XU$ for $j=1,\dots,9$ by differentiating \eqref{eq:ansatz} repeatedly. For convenience, we introduce the transport speed
\begin{align}
V:=\frac{U+e^{\frac{1}{4}s}(\kappa-\dot{\xi})}{1-\dot{\tau}}+\frac{5}{4}X.\label{eq:speed}
\end{align}
Then we have
\begin{align}
\Big(\partial_s+\frac{5}{4}n-\frac{1}{4}+\frac{(\mathbbm{1}_{n\geq 2} +n)\partial_XU}{1-\dot{\tau}}\Big)\partial_X^nU+V\partial_X^{n+1}U=F^{(n)}_U,\label{eq:eqUderivative}
\end{align}
where 
\begin{align}
F^{(n)}_U=\frac{e^{-s}}{1-\dot{\tau}}H[\partial^n_XU]-\frac{\mathbbm{1}_{\{n\text{ odd},n\geq 2\}}{n \choose \frac{n-1}{2}}}{1-\dot{\tau}}(\partial_X^\frac{n+1}{2}U)^2-\frac{1}{1-\dot{\tau}}\sum_{j=2}^{\lfloor \frac{n}{2}\rfloor}{n+1\choose j}\partial_X^jU\partial_X^{n-j+1}U\label{eq:forcingUderivative},
\end{align}
the last sum is non-zero when $n\geq 4$.

\subsection{Constraints on $U$ and equations of modulation variables}
We impose the following constraints on $U$ which are consistent with values of $U_2$:
\begin{align}
U(0,s)=0,\quad\partial_XU(0,s)=-1,\quad\partial_X^4U(0,s)=0.\label{eq:constraint}
\end{align}
Plugging these constraints in the equations of $U,\,\partial_XU$ and $\partial_X^4U$ and evaluating at $X=0$, we get the equations for the modulation variables
\begin{align}
-e^{\frac{1}{4}s}(\kappa-\dot{\xi})&=-e^{-\frac{3}{4}s}\dot{\kappa}+e^{-s}H[U+e^{\frac{1}{4}s}\kappa](0,s),\label{eq:eqkappadot}\\
\dot{\tau}+e^{\frac{1}{4}s}(\kappa-\dot{\xi})\partial_X^2U(0,s)&=e^{-s}H[\partial_XU](0,s),\label{eq:dottaukappa-xi}\\
e^{\frac{1}{4}s}(\kappa-\dot{\xi})\partial_X^5U(0,s)&=e^{-s}H[\partial_X^4U](0,s)-10\partial_X^2U(0,s)\partial_X^3U(0,s).\label{eq:kappa-xi}
\end{align}

\subsection{Equations of $\partial_X^j\widetilde{U}$}
Let $\widetilde{U}:=U-U_2$. Then
\begin{gather}
    \Big(\partial_s-\frac{1}{4}+\frac{\partial_XU_2}{1-\dot{\tau}}\Big)\widetilde{U}+V\partial_X\widetilde{U}=-\frac{e^{-\frac{3}{4}s}\dot{\kappa}}{1-\dot{\tau}}+\frac{e^{-s}}{1-\dot{\tau}}H[U+e^{\frac{1}{4}s}\kappa]-\partial_XU_2\frac{\dot{\tau}U_2+e^{\frac{1}{4}s}(\kappa-\dot{\xi})}{1-\dot{\tau}}.\label{eq:eqtilde}
\end{gather}    
For equations of the derivatives $\partial_X^n\widetilde{U}$, we split into the case when $n=1$ and higher derivatives. When $n=1$,
\begin{gather}
    \Big(\partial_s+1+\frac{2\partial_X U_2+\partial_X\widetilde{U}}{1-\dot{\tau}}\Big)\partial_X\widetilde{U}+V\partial^2_X\widetilde{U}=F^{(1)}_{\tilde{U}},\label{eq:eq1xtildeU}
\end{gather}
where the forcing term is
\begin{gather}
F^{(1)}_{\tilde{U}}=\frac{e^{-s}}{1-\dot{\tau}}H[\partial_XU]-\frac{\dot{\tau}}{1-\dot{\tau}}\Big[(\partial_XU_2)^2+U_2\partial_X^2U_2\Big]-\frac{1}{1-\dot{\tau}}\Big[e^{\frac{1}{4}s}(\kappa-\dot{\xi})\partial_X^2U_2+\widetilde{U}\partial_X^2U_2\Big]\label{eq:forcing1xtildeU}.
\end{gather}
When $n\geq 2$,
\begin{gather}
    \Big(\partial_s+\frac{5}{4}n-\frac{1}{4}+\frac{n+1}{1-\dot{\tau}}\partial_X U\Big)\partial^n_X\widetilde{U}+V\partial^{n+1}_X\widetilde{U}=F^{(n)}_{\tilde{U}},\label{eq:eqnxtildeU}
 \end{gather}   
where the forcing terms are 
\begin{equation}
    \begin{aligned}
    F_{\tilde{U}}^{(n)}&=\frac{e^{-s}}{1-\dot{\tau}}H[\partial_X^nU]-\frac{\dot{\tau}}{1-\dot{\tau}}\bigg[\sum_{j=0}^{\lfloor \frac{n}{2}\rfloor}  {n+1\choose j}\partial_X^jU_2\partial_X^{n-j+1}U_2+{n\choose \frac{n+1}{2}}\mathbbm{1}_{\{n\text{ odd}\}}\big(\partial_X^\frac{n+1}{2}U_2\big)^2\bigg]\\
&\qquad-\frac{1}{1-\dot{\tau}}\bigg[e^{\frac{1}{4}s}(\kappa-\dot{\xi})\partial_X^{n+1}U_2+\sum_{j=0}^{n-1}{n+1\choose j}\partial_X^j\widetilde{U}\partial_X^{n+1-j}U_2\\
&\qquad+\sum_{j=2}^{\lfloor\frac{n}{2}\rfloor}{n+1 \choose j}\partial_X^j\widetilde{U}\partial_X^{n+1-j}\widetilde{U}+\mathbbm{1}_{\{n\text{ odd}\}}{n\choose \frac{n+1}{2}}\big(\partial_X^\frac{n+1}{2}\widetilde{U}\big)^2\bigg]
    \end{aligned}.\label{eq:forcingnxtildeU}
\end{equation}

\subsection{Equations containing first order parameter derivatives}
In the following, $\partial_\alpha$ can be replaced with $\partial_\beta$.
\begin{align}
\Big(\partial_s-\frac{1}{4}+\frac{\partial_XU}{1-\dot{\tau}}\Big)\partial_\alpha U+V\partial_X\partial_\alpha U=F_{U,\alpha},\label{eq:eqUalpha}
\end{align} 
where 
\begin{equation}
    \begin{aligned}
F_{U,\alpha}&=\partial_\alpha F_U-\frac{e^{\frac{1}{4}s}}{1-\dot{\tau}}\partial_XU\partial_\alpha(\kappa-\dot{\xi})-\frac{\partial_\alpha\dot{\tau}}{(1-\dot{\tau})^2}\partial_XU\big(U+e^{\frac{1}{4}s}(\kappa-\dot{\xi})\big)\\
&=-\frac{e^{-\frac{3}{4}s}}{1-\dot{\tau}}\partial_\alpha\dot{\kappa}-\frac{e^{-\frac{3}{4}s}}{(1-\dot{\tau})^2}\partial_\alpha\dot{\tau}\,\dot{\kappa}+\frac{e^{-s}}{1-\dot{\tau}}H[\partial_\alpha U+e^{\frac{1}{4}s}\partial_\alpha\kappa]+\frac{e^{-s}}{(1-\dot{\tau})^2}\partial_\alpha\dot{\tau}H[U+e^{\frac{1}{4}s}\kappa]\\
&\qquad-\frac{e^{\frac{1}{4}s}}{1-\dot{\tau}}\partial_XU\partial_\alpha(\kappa-\dot{\xi})-\frac{\partial_\alpha\dot{\tau}}{(1-\dot{\tau})^2}\partial_XU\big(U+e^{\frac{1}{4}s}(\kappa-\dot{\xi})\big)
\end{aligned}.\label{eq:forcingUalpha}
\end{equation}
% The following might also be helpful
% \begin{align*}
% \big(\partial_s-\frac{1}{4}+\frac{\partial_XU}{1-\dot{\tau}}\big)(\partial_\alpha U+e^{\frac{1}{4}s}\partial_\alpha\kappa)+V\partial_\alpha\partial_XU=-\frac{e^{-\frac{3}{4}s}}{(1-\dot{\tau})^2}\partial_\alpha\dot{\tau}\,\dot{\kappa}+\frac{e^{-s}}{1-\dot{\tau}}H[\partial_\alpha U+e^{\frac{1}{4}s}\partial_\alpha\kappa]\\
% +\frac{e^{-s}}{(1-\dot{\tau})^2}\partial_\alpha\dot{\tau}H[U+e^{\frac{1}{4}s}\kappa]+\frac{e^{\frac{1}{4}s}}{1-\dot{\tau}}\partial_XU\partial_\alpha\dot{\xi}-\frac{\partial_\alpha\dot{\tau}}{(1-\dot{\tau})^2}\partial_XU\big(U+e^{\frac{1}{4}s}(\kappa-\dot{\xi})\big).
% \end{align*}
For $n=1,...,8$, we have
\begin{align}
\Big(\partial_s+\frac{5}{4}n-\frac{1}{4}+\frac{n+1}{1-\dot{\tau}}\partial_XU\Big)\partial_\alpha\partial_X^nU+V\partial_\alpha\partial_X^{n+1}U=F_{U,\alpha}^{(n)},\label{eq:eqnxUalpha}
\end{align}
where 
\begin{equation}
    \begin{aligned}
    F_{U,\alpha}^{(n)}&=e^{-s}\partial_\alpha\Big(\frac{1}{1-\dot{\tau}}H[\partial_X^nU]\Big)-\frac{1}{1-\dot{\tau}}\Big\{\partial_\alpha U\partial_X^{n+1}U+\mathbbm{1}_{n\geq 2}G^{(n)}_{1,\alpha}\Big\}\\
&\qquad-\frac{1}{(1-\dot{\tau})^2}\Big\{\partial_\alpha\dot{\tau}U\partial_X^{n+1}U+(n+\mathbbm{1}_{n\geq 2})\partial_\alpha\dot{\tau}\partial_XU\partial_X^nU+\mathbbm{1}_{n\geq 2}G^{(n)}_{2,\alpha}\Big\}
\\
&\qquad-\frac{e^{\frac{1}{4}s}}{1-\dot{\tau}}(\partial_\alpha\kappa-\partial_\alpha\dot{\xi})\partial_X^{n+1}U-\frac{e^{\frac{1}{4}s}}{(1-\dot{\tau})^2}\partial_\alpha\dot{\tau}(\kappa-\dot{\xi})\partial_X^{n+1}U,\label{eq:forcingnxUalpha}
    \end{aligned}
\end{equation}
and
\begin{align*}
\partial_\alpha\Big(\frac{1}{1-\dot{\tau}}H[\partial_X^nU]\Big)&=\frac{1}{1-\dot{\tau}}H[\partial_\alpha\partial_X^nU]+\frac{\partial_\alpha\dot{\tau}}{(1-\dot{\tau})^2}H[\partial_X^nU],\\
G_{1,\alpha}^{(n)}&=\sum_{j=1}^{n-1}{n+1\choose j}\partial_\alpha\partial_X^jU\partial_X^{n-j+1}U,\\
G_{2,\alpha}^{(n)}&=\partial_\alpha\dot{\tau}\bigg[\mathbbm{1}_{n\text{ odd}}{n\choose \frac{n-1}{2}}(\partial_X^\frac{n+1}{2}U)^2+\sum_{j=2}^{\lfloor \frac{n}{2}\rfloor}{n+1\choose j}\partial_X^jU\partial_X^{n-j+1}U\bigg].
\end{align*}

\subsection{Equations containing second order parameter derivatives}
We use mixed derivative $\partial^2_{\alpha\beta}$ to demonstrate, but the same is true for $\partial^2_{\alpha\alpha}$ and $\partial^2_{\beta\beta}$. For $U$,
\begin{align}
\big(\partial_s-\frac{1}{4}+\frac{\partial_XU}{1-\dot{\tau}}\big)\partial_{\alpha\beta}^2U+V\partial_X\partial_{\alpha\beta}^2U=F_{\alpha,\beta},\label{eq:eqUalphabeta}
\end{align}
where 
\begin{equation}
    \begin{aligned}
F_{\alpha,\beta}&=\frac{e^{-s}}{1-\dot{\tau}}H[\partial_{\alpha\beta}^2U+e^{\frac{1}{4}s}\partial_{\alpha\beta}^2\kappa]+\frac{e^{-s}}{(1-\dot{\tau})^2}\partial_\beta\dot{\tau}H[\partial_\alpha U+e^{\frac{1}{4}s}\partial_\alpha\kappa]\\
&\qquad+\frac{e^{-s}}{(1-\dot{\tau})^2}\partial_\alpha\dot{\tau}H[\partial_\beta U+e^{\frac{1}{4}s}\partial_\beta\kappa]
+e^{-s}H[U+e^{\frac{1}{4}s}\kappa]\bigg(\frac{\partial_{\alpha\beta}^2\dot{\tau}}{(1-\dot{\tau})^2}+\frac{2\partial_\alpha\dot{\tau}\partial_\beta\dot{\tau}}{(1-\dot{\tau})^3}\bigg)\\
&\qquad-e^{-\frac{3}{4}s}\bigg[\frac{\partial_{\alpha\beta}^2\kappa}{1-\dot{\tau}}+\frac{\partial_\beta\dot{\tau}\partial_\alpha\dot{\kappa}+\partial_\alpha\dot{\tau}\partial_\beta\dot{\kappa}+\dot{\kappa}\partial_{\alpha\beta}^2\dot{\tau}}{(1-\dot{\tau})^2}+\frac{2\dot{\kappa}\partial_\alpha\dot{\tau}\partial_\beta\dot{\tau}}{(1-\dot{\tau})^3}\bigg]\\
&\qquad-\frac{1}{1-\dot{\tau}}\Big[\partial_\beta\partial_XU\big(\partial_\alpha U+e^{\frac{1}{4}s}\partial_\alpha(\kappa-\dot{\xi})\big)+\partial_\alpha\partial_X U\big(\partial_\beta U+e^{\frac{1}{4}s}\partial_\beta(\kappa-\dot{\xi})\big)\Big]\\
&\qquad-\frac{e^{\frac{1}{4}s}}{1-\dot{\tau}}\partial_{\alpha\beta}^2(\kappa-\dot{\xi})\partial_XU-\frac{2}{(1-\dot{\tau})^3}\partial_\beta\dot{\tau}\partial_\alpha\dot{\tau}\big(U+e^{\frac{1}{4}s}(\kappa-\dot{\xi})\big)\partial_XU\\
&\qquad-\frac{1}{(1-\dot{\tau})^2}\Big[\partial_\beta\dot{\tau}\partial_XU\big(\partial_\alpha U+e^{\frac{1}{4}s}\partial_\alpha(\kappa-\dot{\xi})\big)+\partial_\alpha\dot{\tau}\partial_\beta\partial_XU\big(U+e^{\frac{1}{4}s}(\kappa-\dot{\xi})\big)\\
&\qquad\qquad\qquad\quad+\partial_\beta\dot{\tau}\partial_\alpha\partial_XU\big(U+e^{\frac{1}{4}s}(\kappa-\dot{\xi})\big)+\partial_\alpha\dot{\tau}\partial_XU\big(\partial_\beta U+e^{\frac{1}{4}s}\partial_\beta(\kappa-\dot{\xi})\big)\\
&\qquad\qquad\qquad\quad+\partial^2_{\alpha\beta}\dot{\tau}\partial_XU\big(U+e^{\frac{1}{4}s}(\kappa-\dot{\xi})\big)\Big].
\end{aligned}\label{eq:forcingUalphabeta}
\end{equation}

For $n=1,...,7$, 
\begin{align}
\Big(\partial_s+\frac{5}{4}n-\frac{1}{4}+\frac{n+1}{1-\dot{\tau}}\partial_XU\Big)\partial^2_{\alpha\beta}\partial_X^nU+V\partial_{\alpha\beta}^2\partial_X^{n+1}U=F_{\alpha,\beta}^{(n)},\label{eq:eqnxUalphabeta}
\end{align}
where
\begin{equation}
    \begin{aligned}
    F_{\alpha,\beta}^{(n)}&=e^{-s}\partial_{\alpha\beta}^2\Big(\frac{1}{1-\dot{\tau}}H[\partial_X^nU]\Big)\\
    &\qquad-\frac{1}{1-\dot{\tau}}\Big[\partial_\beta\partial_X^{n+1}U\big(\partial_\alpha U+e^{\frac{1}{4}s}\partial_\alpha(\kappa-\dot{\xi})\big)+\partial_\alpha\partial_X^{n+1}U\big(\partial_\beta U+e^{\frac{1}{4}s}\partial_\beta(\kappa-\dot{\xi})\big)\\
    &\qquad\qquad\qquad+(n+1)\partial_\beta\partial_XU\partial_\alpha\partial_X^nU+\partial_X^{n+1}U\big(\partial_{\alpha\beta}^2U+e^{\frac{1}{4}s}\partial^2_{\alpha\beta}(\kappa-\dot{\xi})\big)+\mathbbm{1}_{n\geq 2}\partial_\beta G_{1,\alpha}^{(n)}\Big]\\
&\qquad-\frac{1}{(1-\dot{\tau})^2}\Big[\partial_\beta\dot{\tau}G_{1,\alpha}^{(n)}+\partial_\beta\dot{\tau}\partial_\alpha\partial_X^{n+1}U\big(U+e^{\frac{1}{4}s}(\kappa-\dot{\xi})\big)+\partial_\alpha\dot{\tau}\partial_\beta\partial_X^{n+1}U\big(U+e^{\frac{1}{4}s}(\kappa-\dot{\xi})\big)\\
&\qquad\qquad\qquad\quad+\partial_{\alpha\beta}^2\dot{\tau}\partial_X^{n+1}U\big(U+e^{\frac{1}{4}s}(\kappa-\dot{\xi})\big)+\partial_\beta\dot{\tau}\partial_X^{n+1}U\big(\partial_\alpha U+e^{\frac{1}{4}s}\partial_\alpha(\kappa-\dot{\xi})\big)\\
&\qquad\qquad\qquad\quad+\partial_\alpha\dot{\tau}\partial_X^{n+1}U\big(\partial_\beta U+e^{\frac{1}{4}s}(\kappa-\dot{\xi})\big)+(n+1)\partial_\beta\dot{\tau}\partial_XU\partial_\alpha\partial_X^nU+\mathbbm{1}_{n\geq 2}\partial_\beta G_{2,\alpha}^{(n)}\\
&\qquad\qquad\qquad\quad+(n+\mathbbm{1}_{n\geq 2})\big(\partial_{\alpha\beta}^2\dot{\tau}\partial_XU\partial_X^nU+\partial_\alpha\dot{\tau}\partial_\beta\partial_XU\partial_X^nU+\partial_\alpha\dot{\tau}\partial_XU\partial_\beta\partial_X^nU\big)\Big]
\\
&\qquad-\frac{2}{(1-\dot{\tau})^3}\Big[\partial_\beta\dot{\tau}\mathbbm{1}_{n\geq 2}G_{2,\alpha}^{(n)}+\partial_\alpha\dot{\tau}\partial_\beta\dot{\tau}\partial_X^{n+1}U\big(U+e^{\frac{1}{4}s}(\kappa-\dot{\xi})\big)\\
&\qquad\qquad\qquad\quad+(n+\mathbbm{1}_{n\geq 2})\partial_\alpha\dot{\tau}\partial_\beta\dot{\tau}\partial_XU\partial_X^nU\Big],
    \end{aligned}\label{eq:forcingnxUalphabeta}
\end{equation}
and
\begin{align*}
\partial_{\alpha\beta}^2\Big(\frac{1}{1-\dot{\tau}}H[\partial_X^nU]\Big)&=\frac{1}{1-\dot{\tau}}H[\partial^2_{\alpha\beta}\partial_X^nU]+\frac{1}{(1-\dot{\tau})^2}\big(\partial_\beta\dot{\tau}H[\partial_\alpha\partial_X^nU]+\partial_\alpha\dot{\tau}H[\partial_\beta\partial_X^nU]\big)\\
&\qquad+\Big(\frac{\partial_{\alpha\beta}^2\dot{\tau}}{(1-\dot{\tau})^2}+\frac{2\partial_\alpha\dot{\tau}\partial_\beta\dot{\tau}}{(1-\dot{\tau})^3}\Big)H[\partial_X^nU],\\
\partial_\beta G_{1,\alpha}^{(n)}&=\sum_{j=1}^{n-1}{n+1\choose j}\big(\partial_{\alpha\beta}^2\partial_X^jU\partial_X^{n-j+1}U+\partial_\alpha\partial_X^jU\partial_\beta\partial_X^{n-j+1}U\big),\\
\partial_\beta G_{2,\alpha}^{(n)}&=\partial_{\alpha\beta}^2\dot{\tau}\Big[\mathbbm{1}_{n\text{ odd}}{n\choose \frac{n-1}{2}}(\partial_X^\frac{n+1}{2}U)^2+\sum_{j=2}^{\lfloor \frac{n}{2}\rfloor}{n+1\choose j}\partial_X^jU\partial_X^{n-j+1}U\Big]\\
&\qquad+\partial_\alpha\dot{\tau}\sum_{j=2}^{n-1}{n+1\choose j}\partial_\beta\partial_X^jU\partial_X^{n-j+1}U.
\end{align*}

\section{Main result}\label{sec:main}
\subsection{Initial data}\label{sec:initialdata}
We take $M>0$ very large to be determined from the proof, and $\epsilon>0$ very small according to $M$.

Let $\chi:\mathbb{R}\to\mathbb{R}$ be an even, smooth bump such that 
\begin{align*}
    \chi(X)=\begin{cases}
    1\quad&\text{if }-1\leq X\leq 1,\\
    0\quad&\text{if }|X|\geq 2,
    \end{cases}
\end{align*}
and almost linear on the intervals $[-2,-1]$ and $[1,2]$. The latter requirement is to make the 2nd and 3rd derivatives of $\chi$ not too big (this is only for convenience and not essential). Let $\widehat{U}_0:\mathbb{R}\to\mathbb{R}$ such that
\begin{align}
   \widehat{U}_0^{(n)}(0)&=0\qquad\text{for }n=0,1,4,5,\label{eq:hatU0vanish0}\\
    |\widehat{U}_0''(0)|&\leq \epsilon,\qquad|\widehat{U}_0^{(3)}(0)|\leq \epsilon, \label{eq:hatU023small}\\
    |\widehat{U}_0(X)|&\leq \epsilon(1+X^4)^\frac{1}{20}\qquad \text{for }|X|\leq \frac{1}{2}\epsilon^{-\frac{5}{4}},\label{eq:hatUmiddle}\\
    |\widehat{U}^{(n)}_0(X)|&\leq \epsilon(1+X^4)^{-\frac{1}{5}}\qquad\text{for }|X|\leq \frac{1}{2}\epsilon^{-\frac{5}{4}}\text{ and }n=1,\dots,8,\label{eq:hatU0nxmiddle}\\
    |\widehat{U}_0'(X)|&\leq \frac{1}{2}\epsilon\qquad\text{for } |X|\geq \frac{1}{2}\epsilon^{-\frac{5}{4}},\label{eq:hatU01xfar}\\
    |\widehat{U}^{(n)}_0(X)|&\leq \frac{1}{2}M^{n^2}\epsilon\qquad\text{for } |X|\geq \frac{1}{2}\epsilon^{-\frac{5}{4}}\text{ and }n=2,\dots,8.\label{eq:hatU0nxfar}
\end{align}
We set the initial data to be
\begin{gather}
    U_{\alpha,\beta}(X,-\log\epsilon)=U_2(X)\chi(2\epsilon^{\frac{5}{4}}X)+\widehat{U}_0(X)+\chi(X)(\alpha X^2+\beta X^3),\label{eq:initialdata}
    \end{gather}
converting back to the physical coordinates,
   \begin{gather} u_0(x)=\epsilon^\frac{1}{4}U_2(\epsilon^{-\frac{5}{4}}x)\chi(2x)+\epsilon^\frac{1}{4}\widehat{U}_0(\epsilon^{-\frac{5}{4}}x)+\epsilon^\frac{1}{4}\chi(\epsilon^{-\frac{5}{4}}x)\big(\alpha(\epsilon^{-\frac{5}{4}}x)^2+\beta(\epsilon^{-\frac{5}{4}}x)^3\big)+\kappa(-\epsilon),\label{eq:initialphysical}
\end{gather}
where $\alpha,\,\beta$ are parameters whose exact values will be determined by an iterative method below, but 
\begin{align}
    (\alpha,\beta)\in B:=\{(\alpha,\beta):    |\alpha|\leq \frac{5}{4}M^{15}\epsilon,|\beta|\leq\frac{3}{2}M^{25}\epsilon\}.\label{eq:sizealphabeta}
\end{align}
Moreover, we make $U_{\alpha,\beta}(X,-\log\epsilon)+\epsilon^{-\frac{1}{4}}\kappa(-\epsilon)$, to have compact support $[-\epsilon^{-\frac{5}{4}},\epsilon^{-\frac{5}{4}}]$, and also\footnote{It is indeed possible to make $\|\partial_XU(\cdot,-\log\epsilon)\|_{L^2}\leq \sqrt{6}$. Note that
\begin{align*}
    \partial_XU(X,-\log\epsilon)=U_2'(X)\chi(2\epsilon^\frac{5}{4}X)+2\epsilon^\frac{5}{4}U_2(X)\chi'(2\epsilon^\frac{5}{4}X)+\widehat{U}_0'(X)+\chi(X)(2\alpha X+3\beta X^3)+\chi'(X)(\alpha X^2+\beta X^3).
\end{align*}
Using the properties of $\widehat{U}_0$, the support of $U(\cdot,\log\epsilon)$ and $\chi$, and \eqref{eq:sizealphabeta}, we have
\begin{align*}
    \|\partial_XU(\cdot,-\log\epsilon)\|_{L^2}&\leq (1+\epsilon)\Big(\int_\mathbb{R}(1+X^4)^{-\frac{2}{5}}\,dX\Big)^\frac{1}{2}+\epsilon\cdot\epsilon^{-\frac{5}{8}}\\
    &\quad+2C_{\chi'}(1+\epsilon)\epsilon^\frac{5}{4}\Big(\int_{\frac{1}{2}\epsilon^{-\frac{5}{4}}\leq|X|\leq \epsilon^{-\frac{5}{4}}}(1+X^4)^\frac{1}{10}\,dX\Big)^\frac{1}{2} +C_\chi M^{25}\epsilon\\
    &\leq 2.2547(1+\epsilon)+\epsilon^\frac{3}{8}+C_{\chi'}\epsilon^\frac{5}{4}\epsilon^{-\frac{7}{8}}+C_\chi M^{25}\epsilon\leq 2.449<\sqrt{6}.
\end{align*}}
\begin{gather}
    \|\partial_XU(\cdot,-\log\epsilon)\|_{L^2}\leq \sqrt{6},\label{eq:L21xUinitial}\\
    -\|\partial_XU(\cdot,-\log\epsilon)\|_{L^\infty}=\min_X\partial_XU(X,-\log\epsilon)=\partial_XU(0,-\log\epsilon)=-1,\\ \|\partial_X^9U(\cdot,-\log\epsilon)\|_{L^2}\leq M^{65},\label{eq:L29xUinitial}\\
    \|U(\cdot,-\log\epsilon)+\epsilon^{-\frac{1}{4}}\kappa_0\|_{L^\infty}\leq \frac{1}{2}M\epsilon^{-\frac{1}{4}}.
\end{gather}
As a consequence, in physical variables,
\begin{gather*}
    \mathrm{supp}\,u_0\subset [-1,1],\qquad \|u_0\|_{L^\infty}\leq \frac{1}{2}M,\qquad\|u_0\|_{L^2}\leq \frac{\sqrt{2}}{2}M,\\
    -\|u_0'\|_{L^\infty}=\min_x u_0'(x)=u_0'(0)=-1.
\end{gather*}

We denote $U_{\alpha,\beta}$ the solution to \eqref{eq:ansatz} with initial data $U_{\alpha,\beta}(\cdot,-\log\epsilon)$. We omit $\alpha,\beta$ when the estimates hold for all $(\alpha,\,\beta)\in B$.

\subsection{Bootstrap assumptions on $U$ and its spatial derivatives}\label{sec:boostrapassumption}
First, due to the $L^2$ conservation in physical variables, we have 
\begin{align*}
    \|U(\cdot,s)+e^{\frac{1}{4}s}\kappa\|_{L^2}=e^{\frac{7}{8}s}\|u_0\|_{L^2}\leq \frac{\sqrt{2}}{2}Me^{\frac{7}{8}s}.
\end{align*}
For pointwise estimates, we need to separate $X$ into near, middle and far fields. Let $l=(\log M)^{-2}$
be a very small near field threshold, and $\frac{1}{2}e^{\frac{5}{4}s}$ be the far field threshold (the power of $s$ comes from the repelling speed of the Lagragian trajectories). Let $\widetilde{U}:=U-U_2$.

Near field assumptions: for $|X|\leq l$,
\begin{align}
|\partial_X^n\widetilde{U}(X,s)|&\leq \epsilon^\frac{1}{5}|X|^{6-n}+\epsilon^\frac{1}{2}\leq 2l^{6-n}\epsilon^\frac{1}{5}\quad n=0,...,5,\label{eq:assumptionnear0-5}\\
|\partial_X^6\widetilde{U}(X,s)|&\leq \epsilon^\frac{1}{5},\qquad|\partial_X^7\widetilde{U}(X,s)|\leq M\epsilon^\frac{1}{5},\qquad
|\partial_X^8\widetilde{U}(X,s)|\leq M^3\epsilon^\frac{1}{5}.\label{eq:assumptionnear6-8}
\end{align}
Middle field assumptions: for $(X,s)$ such that $l\leq |X|\leq \frac{1}{2}e^{\frac{5}{4}s}$,
\begin{align}
|\widetilde{U}(X,s)|&\leq \epsilon^\frac{3}{20}(1+X^4)^\frac{1}{20},\label{eq:assumptionUmiddle}\\
|\partial_X\widetilde{U}(X,s)|&\leq \epsilon^\frac{1}{20}(1+X^4)^{-\frac{1}{5}},\label{eq:assumption1xmiddle}\\
|\partial_X^nU(X,s)|&\leq M^{n^2}(1+X^4)^{-\frac{1}{5}},\qquad n=2,...,8.\label{eq:assumptionnxmiddle}
\end{align}
Far field assumptions: for $(X,s)$ such that $|X|\geq \frac{1}{2}e^{\frac{5}{4}s}$,
\begin{align}
|\partial_XU(X,s)|&\leq 4e^{-s},\label{eq:assumption1xfar}\\
|\partial_X^nU(X,s)|&\leq 2M^{n^2}e^{-s}\qquad n=2,...,8.\label{eq:assumptionnxfar}
\end{align}
Assumptions on $L^2$ norms,\footnote{We want $\|\partial_X^8U\|_{L^\infty}\leq M^{64}$ to be a nontrivial bound, i.e. better than the one from interpolation. So the lowest $M$ power is 69, but an even number makes expressions look better.}
\begin{align}
   \|\partial_XU(\cdot,s)\|_{L^2}\leq \sqrt{7},\qquad \|\partial_X^9U(\cdot,s)\|_{L^2}\leq M^{70}.\label{eq:assumptionL2}
\end{align}
As a consequence of Sobolev interpolation Lemma \ref{lem:Sinterpolation},
\begin{align*}
    \|\partial_X^nU(\cdot,s)\|_{L^2}&\leq \|\partial_XU(\cdot,s)\|_{L^2}^{1-\frac{n-1}{8}}\|\partial_X^9U(\cdot,s)\|_{L^2}^\frac{n-1}{8}\\
    &\leq 7^{\frac{1}{2}-\frac{n-1}{16}}M^\frac{35(n-1)}{4},\qquad n=2,...,8.
\end{align*}
We also have the $L^\infty$-norm assumption
\begin{align}
    \|U(\cdot,s)+e^{\frac{1}{4}s}\kappa\|_{L^\infty}\leq Me^{\frac{1}{4}s},\label{eq:assumptionULinfty}
\end{align}
so that the solution $u$ in physical coordinates is bounded: $\|u(\cdot,t)\|_{L^\infty}\leq M$ for $t\in [-\epsilon,T_*]$.\\
Finally, the assumptions at $X=0$:
\begin{align}
    |\partial_X^2U(0,s)|\leq \epsilon^\frac{1}{10}e^{-\frac{3}{4}s},\qquad |\partial_X^3U(0,s)|\leq M^{27}e^{-s},\qquad 
    |\partial_X^5\widetilde{U}(0,s)|\leq \epsilon^\frac{1}{2}.\label{eq:assumptionx=0}
\end{align}

\begin{remark}
Unlike \cite{Yang}, here we can not obtain
\begin{align*}
    -\|\partial_XU(\cdot,s)\|_{L^\infty}=-1=\partial_XU(0,s).
\end{align*}
The problem comes from the near field. Middle and far fields still hold. Clearly, for $|X|\geq \frac{1}{2}e^{\frac{5}{4}s}$, by \eqref{eq:assumption1xfar},
\begin{align*}
    |\partial_XU(X,s)|\leq 4e^{-s}<1.
\end{align*}
For $l\leq|X|\leq\frac{1}{2}e^{\frac{5}{4}s}$, by \eqref{eq:assumption1xmiddle} and  \eqref{eq:U2away0},
\begin{align*}
    |\partial_XU(X,s)|&\leq |U_2'(X)|+|\partial_X\widetilde{U}(X,s)|\\
    &\leq (1-2l^4+\epsilon^\frac{1}{20})(1+X^4)^{-\frac{1}{5}}\leq (1+X^4)^{-\frac{1}{5}}<1.
\end{align*}
But for $|X|\leq l$, the Taylor expansion near $X=0$ is
\begin{align*}
    \partial_XU(X,s)=-1+\partial_X^2U(0,s)X+\frac{1}{2}\partial_X^3U(0,s)X^2+\frac{1}{24}\partial_X^5U(0,s)X^4+\frac{1}{120}\partial_X^6U(X',s)X^5,
\end{align*}
i.e.\ the linear term does not vanish. Nevertheless, since \eqref{eq:assumptionnear6-8} and
\begin{align*}
    \partial_X^5U(0,s)\geq 120-\epsilon^\frac{1}{2}>119,
\end{align*}
we can obtain
\begin{gather}
    \partial_XU(X,s)\geq -1-\epsilon^\frac{1}{10}e^{-\frac{3}{4}s}l-\frac{1}{2}M^{27}e^{-s}l^2+\frac{119}{48}X^4\geq -1-\epsilon^\frac{1}{10}e^{-\frac{3}{4}s},\\
    \|\partial_XU(\cdot,s)\|_{L^\infty}\leq 1+\epsilon^\frac{1}{10}e^{-\frac{3}{4}s}.\label{eq:conse1xULinfty}
\end{gather}
As $s\to+\infty$, we recover the property that minimum slope is $-1$ at $X=0$.
\end{remark}

\subsection{Bootstrap assumptions on parameter derivatives}
All the assumptions on parameter derivatives hold uniformly for all $(\alpha,\beta)\in B$ where $B$ is defined in \eqref{eq:sizealphabeta}.

For first order parameter derivatives, the only important estimates are at $X=0$, which need to be sharp. But to close the bootstrap, we need to supply bounds for all $X$, but they do not need to be as sharp. We only go to $\partial_\alpha\partial_X^7$, i.e. one $X$ derivative lower than without parameter derivative, because the forcing term of $\partial_\alpha\partial_X^nU$ contains $\partial_X^{n+1}U$.\\
At $X=0$:
\begin{gather}
\begin{aligned}
    \epsilon^\frac{3}{4}e^{\frac{3}{4}s}\leq \partial_\alpha\partial_X^2U(0,s)\leq 4\epsilon^\frac{3}{4}e^{\frac{3}{4}s},\quad |\partial_\alpha\partial_X^3U(0,s)|\leq \epsilon e^{\frac{1}{2}s},\\
|\partial_\beta\partial_X^2U(0,s)|\leq \epsilon e^{\frac{3}{4}s},\quad 4\epsilon^\frac{1}{2}e^{\frac{1}{2}s}\leq \partial_\beta\partial_X^3U(0,s)\leq 8\epsilon^\frac{1}{2}e^{\frac{1}{2}s},
\end{aligned}\label{eq:assumptionjacobian}\\
|\partial_\alpha\partial_X^5U(0,s)|\leq \epsilon^\frac{3}{8}e^{\frac{1}{8}s}.\label{eq:assmumptionpara5x0}
\end{gather}
In the following $\alpha$ can be replaced by $\beta$.\\
Near field: for $|X|\leq l$,
\begin{align}
|\partial_\alpha\partial_X^nU(X,s)|&\leq Ml^\frac{1}{2}\epsilon^\frac{3}{4}e^{\frac{3}{4}s} \qquad n=0,...,6,\label{eq:assumptionparanear0-6}\\
|\partial_\alpha\partial_X^7U(X,s)|&\leq M\epsilon^\frac{3}{4}e^{\frac{3}{4}s}.\label{eq:assumptionparanear7}
\end{align}
Middle field: for $l\leq|X|\leq \frac{1}{2}e^{\frac{5}{4}s}$,
\begin{align}
|\partial_\alpha\partial_X^nU(X,s)|\leq M^{(n+2)^2}\epsilon^\frac{3}{4}e^{\frac{3}{4}s}(1+X^4)^{-\frac{1}{5}},\qquad n=1,...,7.\label{eq:assumptionparanxmiddle}
\end{align}
Far field: for $|X|\geq \frac{1}{2}e^{\frac{5}{4}s}$
\begin{align}
|\partial_\alpha\partial_X^nU(X,s)|\leq 4M^{(n+2)^2}\epsilon^\frac{3}{4}e^{-\frac{1}{4}s},\qquad n=1,...,7.\label{eq:assumptionparanxfar}
\end{align}
On $\partial_\alpha U$, we need to add the $\kappa$ part, 
\begin{align}
    |\partial_\alpha U(X,s)+ e^{\frac{1}{4}s}\partial_\alpha\kappa| \leq\begin{cases}
    M^4\epsilon^\frac{3}{4}e^{\frac{3}{4}s},\quad&\text{if }|X|\leq \frac{1}{2}e^{\frac{5}{4}s},\\
    2M^4\epsilon^\frac{3}{4}e^{\frac{3}{4}s},\quad&\text{if }|X|\geq \frac{1}{2}e^{\frac{5}{4}s}.
    \end{cases} \label{eq:assumptionparaU}
\end{align}
Again we include $L^2$-bounds to deal with the Hilbert transform (but here we do not have the $L^2$ conservation)\footnote{In the proof of $\|e^{-\frac{1}{4}s}\partial_\alpha U+\partial_\alpha\kappa\|_{L^2}$, there is a competition between $e^{\frac{5}{8}s}$ from the ``damping" and $\max\{\|\partial_\alpha\partial_XU\|_{L^2},\|\partial_\alpha U+e^{\frac{1}{4}s}\partial_\alpha\kappa\|_{L^\infty}\}$. The latter cannot grow more slowly than $e^{\frac{5}{8}s}$. Although we can make it arbitrarily close to $\frac{5}{8}$, we choose $\frac{3}{4}$ to leave some room for other estimates.}
\begin{gather}
\|\partial_\alpha U(\cdot,s)+e^{\frac{1}{4}s}\partial_\alpha\kappa\|_{L^2}\leq M^{15}\epsilon^\frac{3}{4}e^s,\label{eq:assumption[araUL2}\\
     \|\partial_\alpha\partial_XU(\cdot, s)\|_{L^2}\leq M^{13}\epsilon^\frac{3}{4}e^{\frac{3}{4}s},\qquad \|\partial_\alpha\partial_X^8U(\cdot,s)\|_{L^2}\leq M^{90}\epsilon^\frac{3}{4}e^{\frac{3}{4}s}.\label{eq:assumptionparaixL2}
\end{gather}
By interpolation, we have
\begin{align*}
    \|\partial_\alpha\partial_X^nU(\cdot,s)\|_{L^2}&\leq \|\partial_\alpha\partial_XU\|_{L^2}^{1-\frac{n-1}{7}} \|\partial_\alpha\partial_X^8U\|_{L^2}^\frac{n-1}{7}\\
    &\leq M^{11n+2}  \epsilon^\frac{3}{4}e^{\frac{3}{4}s},\qquad n=1,...,8.
\end{align*}

Then for second order parameter derivatives, the assumptions are even less sharp. In fact, we only need $L^2$ and $L^\infty$ assumptions. In the following mixed partial $\partial_{\alpha\beta}^2$ can be replaced by $\partial^2_{\alpha\alpha}$ or $\partial_{\beta\beta}^2$. First, $L^\infty$ assumptions:
\begin{align}
    \|\partial^2_{\alpha\beta}\partial_X^nU(\cdot,s)\|_{L^\infty}&\leq M^{(n+5)^2}\epsilon^\frac{3}{2}e^{\frac{3}{2}s},\quad n=1,...,6,\label{eq:assumption2paranx}\\
    \|\partial_{\alpha\beta}^2U+e^{\frac{1}{4}s}\partial_{\alpha\beta}^2\kappa\|_{L^\infty}&\leq M^{25}\epsilon^\frac{3}{2}e^{\frac{3}{2}s}.\label{eq:assumption2paraU}
\end{align}
Then $L^2$ assumptions:
\begin{gather}
    \|\partial_{\alpha\beta}^2U(\cdot,s)+e^{\frac{1}{4}s}\partial_{\alpha\beta}^2\kappa\|_{L^2}\leq M^{38}\epsilon^\frac{3}{2}e^{\frac{7}{4}s},\label{eq:assumption2paraUL2}\\
    \|\partial_{\alpha\beta}^2\partial_XU(\cdot,s)\|_{L^2}\leq M^{37}\epsilon^\frac{3}{2}e^{\frac{3}{2}s},\qquad \|\partial_{\alpha\beta}^2\partial_X^7U(\cdot,s)\|_{L^2}\leq M^{130}\epsilon^\frac{3}{2}e^{\frac{3}{2}s}.\label{eq:assumption2paranxL2}
\end{gather}
So by Sobolev interpolation,
\begin{align*}
    \|\partial_{\alpha\beta}^2\partial_X^nU(\cdot,s)\|&\leq \|\partial_{\alpha\beta}^2\partial_XU\|_{L^2}^{1-\frac{n-1}{6}}\|\partial_{\alpha\beta}^2\partial_X^7U\|_{L^2}^\frac{n-1}{6}\\
    &\leq M^{\frac{31}{2}n+11\frac{1}{2}}\epsilon^\frac{3}{2}e^{\frac{3}{2}s},\qquad
     n=1,...,7.
\end{align*}

\subsection{Iterative assumptions}
For a fixed $\widehat{U}_0$ given above, there is a unique pair $(\alpha,\beta)$ such that the solution $U_{\alpha,\beta}$ corresponding to the initial datum given by \eqref{eq:initialdata} asymptotically converges to $U_2$. To find such $\alpha,\,\beta$, denote $s_n=-\log\epsilon+n$, $n=0,1,2,...$. Initialize 
\begin{align*}
    \alpha_0=-\frac{1}{2}\widehat{U}''(0),\qquad \beta_0=-\frac{1}{6}\widehat{U}^{(3)}(0),
\end{align*}
so that 
\begin{align*}
    \partial_X^2U_{\alpha_0,\beta_0}(0,s_0)=0,\quad \partial_X^3U_{\alpha_0,\beta_0}(0,s_0)=0.
\end{align*}
The iterative assumption is, for each $n=0,1,2,...$, given $\alpha_n,\,\beta_n$ that make
\begin{align}
    \partial_X^2U_{\alpha_n,\beta_n}(0,s_n)=0,\qquad \partial_X^3U_{\alpha_n,\beta_n}(0,s_n)=0,\label{eq:inductassumption}
\end{align}
we can find $\alpha_{n+1},\,\beta_{n+1}$ in the neighbourhood 
\begin{align}
    B_n:=\{(\alpha,\beta):|\alpha-\alpha_n|\leq M^{15}\epsilon^{-\frac{3}{4}}e^{-\frac{7}{4}s_n}+\epsilon^{-\frac{3}{10}}e^{-\frac{3}{2}s_n},\ |\beta-\beta_n|\leq M^{25}\epsilon^{-\frac{1}{2}}e^{-\frac{3}{2}s_n}\}\label{eq:sizealphabetan}
\end{align}
such that 
\begin{align*}
     \partial_X^2U_{\alpha_{n+1},\beta_{n+1}}(0,s_{n+1})=0,\qquad \partial_X^3U_{\alpha_{n+1},\beta_{n+1}}(0,s_{n+1})=0.
\end{align*}
For this step we need to furthermore assume, for $s_n\leq s\leq s_{n+1}$,
\begin{align}
    |\partial_X^2U_{\alpha_n,\beta_n}(0,s)|\leq M^{13\frac{1}{2}}e^{-s},\qquad |\partial_X^3U_{\alpha_n,\beta_n}(0,s)|\leq M^{22}e^{-s}.\label{eq:finerinductassumptioin}
\end{align}

\subsection{Bootstrap assumptions on modulation variables}
Without parameter derivatives,
\begin{align}
    |\kappa-\dot{\xi}|\leq \epsilon^\frac{1}{5}e^{-s},\quad |\dot{\xi}|\leq 2M,\quad |\dot{\tau}|\leq \epsilon^\frac{1}{5}e^{-\frac{3}{4}s},\quad |\dot{\kappa}|\leq M^2e^{\frac{1}{8}s},\qquad |T_*|\leq 2\epsilon^{\frac{39}{20}}.\label{eq:assumptionmodulation}
\end{align}
First order parameter derivatives
\begin{align}
    |\partial_\alpha\dot{\tau}|\leq \epsilon^\frac{1}{2},\quad|\partial_\alpha(\kappa-\dot{\xi})|\leq M\epsilon^\frac{1}{2}e^{-\frac{1}{2}s},\quad |\partial_\alpha\dot{\xi}|\leq M\epsilon^\frac{1}{2},\quad|\partial_\alpha\kappa|\leq \epsilon^\frac{1}{2},\quad |\partial_\alpha\dot{\kappa}|\leq \epsilon^\frac{1}{2}se^{\frac{1}{2}s}.\label{eq:assumptionparamodulation}
\end{align}
Second order parameter derivatives
\begin{equation}
    \begin{gathered}
     |\partial_{\alpha\beta}^2\dot{\tau}|\leq \epsilon e^{\frac{3}{4}s},\quad |\partial_{\alpha\beta}^2\dot{\xi}|\leq 5M\epsilon^\frac{5}{4}e^s,\quad |\partial_{\alpha\beta}^2\dot{\kappa}|\leq 2M\epsilon^\frac{5}{4}e^{2s},\\
    |\partial_{\alpha\beta}^2\kappa|\leq 3M\epsilon^\frac{5}{4}e^s,\quad |\partial_{\alpha\beta}^2(\kappa-\dot{\xi})|\leq M\epsilon^\frac{5}{4}e^{s}.
    \end{gathered}\label{eq:assumption2paramodulation}
\end{equation}
These assumptions are far from being sharp, yet sufficient.

\subsection{Statement of the theorems}
\begin{theorem}[self-similar profile]\label{thm:selfsimilar}
There exists a sufficiently large parameter $M\gg 1$, a sufficiently small parameter $0<\epsilon=\epsilon(M)\ll 1$, such that for each $\widehat{U}_0(X):\mathbb{R}\to\mathbb{R}$ satisfying the conditions \eqref{eq:hatU0vanish0}-\eqref{eq:hatU0nxfar}, we can uniquely find $\alpha,\,\beta$ with
\begin{align*}
    |\alpha|\leq \frac{5}{4}M^{15}\epsilon,\qquad|\beta|\leq \frac{6}{5}M^{25}\epsilon,
\end{align*}
such that the self-similar Burgers-Hilbert equation \eqref{eq:ansatz} with the initial datum
$U_{\alpha,\beta}(X,-\log\epsilon)$ given in Section \ref{sec:initialdata} has a unique solution $U(X,s)\in\mathcal{C}\big([-\log\epsilon,+\infty);\mathcal{C}^8\cap H^9(\mathbb{R})\big)$ that satisfies the assumptions given in Section \ref{sec:boostrapassumption}. Moreover, the limit $\nu:=\lim_{s\to+\infty}\partial_X^5U(0,s)$ exists, and for each $X\in\mathbb{R}$, we have
\begin{align*}
    \lim_{s\to+\infty}U(X,s)=\Big(\frac{\nu}{120}\Big)^{-\frac{1}{4}}U_2\bigg(\Big(\frac{\nu}{120}\Big)^\frac{1}{4}X\bigg):=U_2^\nu(X).
\end{align*}
\end{theorem}

\begin{theorem}[main result]\label{thm:main}
There exits a sufficiently large $M\gg 1$ and a sufficiently small $0<\epsilon=\epsilon(M)\ll 1$, such that for smooth initial data $u_0$ at $t_0=-\epsilon$ with minimum initial slope $-1/\epsilon$ given in Section \ref{sec:initialdata} after the self-similar transformation \eqref{eq:selftransform}, there exists a unique solution $u\in \mathcal{C}\big([-\epsilon,T_*);\mathcal{C}^8\cap H^9(\mathbb{R})\big)$ to the Burgers-Hilbert equation \eqref{eq:BH} with the following properties
\begin{enumerate}
    \item The blowup time $|T_*|\leq 2\epsilon^\frac{39}{20}$, hence the lifespan of $u$ is $O(\epsilon)$.
    \item Solution is bounded: $\|u(\cdot,t)\|_{L^\infty}\leq M$ for all $t\in[-\epsilon, T_*]$.
    \item The blowup location $x_*:=\xi(T_*)$ satisfies $|x_*|\leq 3M\epsilon$ (which is small).
    \item Shock formation and rate: as $t\to T_*$,
    \begin{gather*}
         \partial_xu\big(\xi(t),t\big)\to-\infty,\\
         \frac{1}{2(T_*-t)}\leq \|\partial_xu(\cdot,t)\|_{L^\infty}\leq \frac{2}{T_*-t},
    \end{gather*}
    and for any $x\neq x_*$, $\partial_xu(x,t)$ remains finite. 
    \item Shock profile is a cusp: for $t$ sufficiently close to $T_*$ (depending on $x$), 
    \begin{align*}
    |\partial_xu(x,t)|
        \begin{cases}
        \leq 4\quad&\text{if }|x-x_*|>\frac{1}{2}\\
        \sim |x-x_*|^{-\frac{5}{4}}\quad&\text{if }|x-x_*|\leq \frac{1}{2}.
        \end{cases}
    \end{align*}
    In particular, $u(\cdot, T_*)\in C^\frac{1}{5}(\mathbb{R})$ has a cusp singularity at $x_*$.
\end{enumerate}
\end{theorem}

\begin{corollary}[codimension 2 of initial data]
\label{cor:codimension}
The set of initial data such that the conclusion in Theorem \ref{thm:main} holds is a codimension 2 subset of $H^9(\mathbb{R})$.
\end{corollary}

\section{Forcing bounds without parameter derivatives}\label{sec:forcing}
\subsection{Hilbert transform terms}\label{sec:Hilbertbound}
We first bound $H[U+e^{\frac{1}{4}s}\kappa]$. We consider the standard near-middle-far field separation:
\begin{align*}
H[U+e^{\frac{1}{4}s}\kappa](X,s)&=\frac{1}{\pi}\bigg[\lim_{\delta\to 0+}\int_{\delta\leq |X-Y|\leq 1}\frac{U(Y,s)-U(X,s)}{X-Y}\,dY\\
&\qquad+\int_{1\leq |X-Y|\leq \frac{1}{2}e^{\frac{5}{4}s}}
+\int_{|X-Y|\geq \frac{1}{2}e^{\frac{5}{4}s}}\frac{U(Y,s)+e^{\frac{1}{4}s}\kappa}{X-Y}\,dY\bigg]\\
&\lesssim 1+Me^{\frac{1}{4}s}(\frac{5}{4}s-\log 2)+\|U+e^\frac{s}{4}\kappa\|_{L^2}\Big(\int_{|X-Y|\geq \frac{1}{2}e^{\frac{5}{4}s}}\frac{1}{|X-Y|^2}\,dY\Big)^\frac{1}{2}\\
&\lesssim Mse^{\frac{1}{4}s}+Me^{\frac{1}{4}s}\lesssim Me^{\frac{3}{8}s}.\numberthis \label{eq:HilbertU}
\end{align*}
By interpolation,
\begin{align*}
    \|H[\partial_X^n U](\cdot,s)\|_{L^\infty} &\lesssim \|H[\partial_X U](\cdot,s)\|_{L^2}^\frac{17-2n}{16}\|H[\partial_X^9U](\cdot,s)\|_{L^2}^\frac{2n-1}{16}\\
    &\lesssim \|\partial_XU(\cdot,s)\|_{L^2}^\frac{17-2n}{16}\|\partial_X^9(\cdot,s)\|_{L^2}^\frac{2n-1}{16}\lesssim M^\frac{35(2n-1)}{8},\qquad n=1,...,8.\numberthis\label{eq:HilbertnxLinfty}
\end{align*}
Now we need spatial and temporal decays in the middle field and far field respectively. We use  \eqref{eq:assumption1xmiddle}, \eqref{eq:assumption1xfar}, \eqref{eq:assumptionnxmiddle} and \eqref{eq:assumptionnxfar}. For the middle field $l\leq|X|\leq \frac{1}{2}e^{\frac{5}{4}s}$,
\begin{align*}
H[\partial_XU](X,s)&=\frac{1}{\pi}\bigg[\int_{|X-Y|<(1+X^4)^{-\frac{1}{5}}}\frac{\partial_XU(Y,s)-\partial_XU(X,s)}{X-Y}\,dY\\
&\qquad+\int_{(1+X^4)^{-\frac{1}{5}}\leq |X-Y|\leq 1}\frac{\partial_XU(Y,s)}{X-Y}\,dY+\int_{|X-Y|> 1}\frac{\partial_XU(Y,s)}{X-Y}\,dY\bigg]\\
&=I_{\text{near}}+I_{\text{middle}}+I_\text{far}.
\end{align*}
Since $\|\partial_X^2U(\cdot,s)\|_{L^\infty}\leq M^4$, 
\begin{align*}
|I_\text{near}|\lesssim \|\partial_X^2U(\cdot,s)\|_{L^\infty}(1+X^4)^{-\frac{1}{5}}\lesssim M^4 (1+X^4)^{-\frac{1}{5}}.
\end{align*}
For $I_\text{middle}$, same argument as \cite{Yang},
\begin{align*}
|I_\text{middle}|\lesssim (1+X^4)^{-\frac{1}{5}}\log(1+X^4).
\end{align*}
For $I_\text{far}$, split into cases as \cite{Yang}. 
In all three cases we get
\begin{align*}
|I_\text{far}|\lesssim (1+X^4)^{-\frac{1}{8}}+(1+X^4)^{-\frac{1}{5}}\log(1+X^4).
\end{align*}
Putting the three regions together,
\begin{align}
|H[\partial_XU](X,s)|\lesssim M^4(1+X^4)^{-\frac{1}{5}}+(1+X^4)^{-\frac{1}{8}}+(1+X^4)^{-\frac{1}{5}}\log(1+X^4).\label{eq:Hilbert1xmiddle}
\end{align}
For the far field $|X|\geq \frac{1}{2}e^{\frac{5}{4}s}$, 
\begin{align*}
H[\partial_XU](X,s)&=\frac{1}{\pi}\bigg[\int_{|X-Y|<e^{-s}}\frac{\partial_XU(Y,s)-\partial_XU(X,s)}{X-Y}\,dY\\
&\qquad+\int_{e^{-s}\leq |X-Y|\leq 1}\frac{\partial_XU(Y,s)}{X-Y}\,dY+\int_{|X-Y|> 1}\frac{\partial_XU(Y,s)}{X-Y}\,dY\bigg]\\
&=\mathit{II}_{\text{near}}+\mathit{II}_{\text{middle}}+\mathit{II}_\text{far}\\
&\lesssim M^4 e^{-s}+se^{-s}+\max\{e^{-\frac{3}{8}s},e^{-\frac{5}{8}s},se^{-s}\}\lesssim e^{-\frac{3}{8}s}.\numberthis\label{eq:Hilbert1xfar}
\end{align*}
Using the same idea, for higher order derivatives $n=2,...,7$, we have
\begin{align}
    |H[\partial_X^nU](X,s)|\lesssim\begin{cases}
    &M^{\max\{(n+1)^2,\frac{35(n-1)}{4}\}}(1+X^4)^{-\frac{1}{5}}+M^{n^2}(1+X^4)^{-\frac{1}{5}}\log(1+X^4)\\
    &\qquad\qquad\qquad\text{if }l\leq|X|\leq\frac{1}{2}e^{\frac{5}{4}s},\\
    &M^{n^2}e^{-\frac{3}{8}s}\qquad\text{if }|X|\geq \frac{1}{2}e^{\frac{5}{4}s}.
    \end{cases} \label{eq:Hilbertnx}
\end{align}

The estimate for $H[\partial_X^8U]$ is a little different, since we only make an assumption on $\|\partial_X^9U\|_{L^2}$ but not the $L^\infty$-norm. In this case, we use Sobolev embedding (Morrey's inequality) $H^1(\mathbb{R})\subset \mathcal{C}^{0,\frac{1}{2}}(\mathbb{R})$ to deal with the near field. When $l\leq |X|\leq\frac{1}{2}e^{\frac{5}{4}s}$,
\begin{align*}
&\Big|\int_{|X-Y|<(1+X^4)^{-\frac{1}{5}}}\frac{\partial^8_XU(Y,s)-\partial_X^8U(X,s)}{X-Y}\,dY\Big|\\
\leq& \int_{|X-Y|<(1+X^4)^{-\frac{1}{5}}}\frac{|\partial^8_XU(Y,s)-\partial_X^8U(X,s)|}{|X-Y|^\frac{1}{2}}\frac{1}{|X-Y|^\frac{1}{2}}\,dY\\
\leq& \|\partial_X^8U(\cdot,s)\|_{\mathcal{C}^{0,\frac{1}{2}}}\int_{|X-Y|<(1+X^4)^{-\frac{1}{5}}}\frac{1}{|X-Y|^\frac{1}{2}}\,dY\\
\lesssim & (\|\partial_X^8U(\cdot,s)\|_{L^2}+\|\partial_X^9U(\cdot,s)\|_{L^2})(1+X^4)^{-\frac{1}{10}}\\
\lesssim & (M^{\frac{7}{8}\cdot 70}+M^{70})(1+X^4)^{-\frac{1}{10}}\lesssim M^{70}(1+X^4)^{-\frac{1}{10}}.
\end{align*}
When $|X|\geq \frac{1}{2}e^{\frac{5}{4}s}$, instead of $|X-Y|<(1+X^4)^{-\frac{1}{5}}$ we use
$|X-Y|<e^{-s}$ as the region for the near field integral. Using the same argument, the near field integral can be bounded by 
\begin{align*}
\int_{|X-Y|<e^{-s}}\Big|\frac{\partial^8_XU(Y,s)-\partial^8_XU(X,s)}{X-Y}\Big|\,dY\lesssim
M^{70} e^{-\frac{1}{2}s}.
\end{align*}
Combining with the (unchanged) middle field and far field integrals, we get
\begin{align*}
|H[\partial_X^8U](X,s)|&\lesssim M^{70}(1+X^4)^{-\frac{1}{10}}+M^{64}(1+X^4)^{-\frac{1}{5}}\log(1+X^4)+M^{61\frac{1}{4}}(1+X^4)^{-\frac{1}{8}}\\
&\lesssim M^{70}(1+X^4)^{-\frac{1}{10}}\qquad\text{for }l\leq |X|\leq \frac{1}{2}e^{\frac{5}{4}s},\numberthis\label{eq:Hilbert8xmiddle}\\
|H[\partial_X^8U](X,s)|&\lesssim M^{70} e^{-\frac{1}{2}s}+M^{64}se^{-s}+M^{64}e^{-\frac{3}{8}s}\\
&\lesssim M^{70}e^{-\frac{1}{2}s}+M^{64} e^{-\frac{3}{8}s}\qquad\text{for }|X|\geq \frac{1}{2}e^{\frac{5}{4}s}.\numberthis\label{eq:Hilbert8xfar}
\end{align*} 
The slightly weaker bounds do not affect the decay of $\partial_X^8U(X,s)$.

\subsection{$F_U$}
By \eqref{eq:eqkappadot}, 
\begin{align*}
    F_{U}(X,s)&=\frac{1}{1-\dot{\tau}}\Big\{-e^{\frac{1}{4}s}(\kappa-\dot{\xi})-e^{-s}H[U+e^{\frac{1}{4}s}\kappa](0,s)+e^{-s}H[U+e^{\frac{1}{4}s}\kappa](X,s)\\
    |F_U(X,s)|&\lesssim \epsilon^\frac{1}{5}e^{-\frac{3}{4}s}+Me^{-\frac{5}{8}s}\lesssim Me^{-\frac{5}{8}s}.\numberthis\label{eq:forcingU}
\end{align*}

\subsection{Near field}
Take $n=6$ in \eqref{eq:eqnxtildeU}, we have
\begin{align*}
F_{\tilde{U}}^{(6)}=&\frac{e^{-s}}{1-\dot{\tau}}H[\partial_X^6U]-\frac{\dot{\tau}}{1-\dot{\tau}}\sum_{j=0}^3 {7\choose j}\partial_X^jU_2\partial_X^{7-j}U_2\\
&-\frac{1}{1-\dot{\tau}}\Big[e^\frac{s}{4}(\kappa-\dot{\xi})\partial_X^{7}U_2+\sum_{j=0}^{5}{7\choose j}\partial_X^j\widetilde{U}\partial_X^{7}U_2+\sum_{j=2}^3{7\choose j}\partial_X^j\widetilde{U}\partial_X^{7-j}\widetilde{U} \Big].
\end{align*}
We bound the first term by
\begin{align*}
\frac{e^{-s}}{1-\dot{\tau}}\big|H[\partial_X^6U](X,s)\big|\leq \frac{e^{-s}}{1-\dot{\tau}}\|H[\partial_X^6U](\cdot,s)\|_{L^\infty}\lesssim M^{48\frac{1}{8}}e^{-s}.
\end{align*}
Since $|\partial_X^jU_2(X)|\leq C=C(U_2)$ for $|X|\leq l$ (near 0), and $|\dot{\tau}|\leq \epsilon^\frac{1}{5}e^{-\frac{3}{4}s}$, we can bound terms with $\dot{\tau}$ by
\begin{align*}
\Big|\frac{\dot{\tau}}{1-\dot{\tau}}\sum_{j=0}^3 {7 \choose j}\partial_X^jU_2\partial_X^{7-j}U_2\Big|\lesssim \epsilon^\frac{1}{5} e^{-\frac{3}{4}s}.
\end{align*}
To bound terms without $\dot{\tau}$, we use the bootstrap assumption \eqref{eq:assumptionnear0-5}
and $|\kappa-\dot{\xi}|\leq \epsilon^\frac{1}{5}e^{-s}$. Then, the second line in the sum of $F_{\tilde{U}}^{(6)}$ can be bounded by (up to a constant factor)
\begin{align*}
e^{\frac{1}{4}s}\epsilon^\frac{1}{5}e^{-s}+\epsilon^\frac{1}{5}(l+\cdots+l^6)+\epsilon^\frac{2}{5}l^5\lesssim \epsilon^\frac{1}{5}l+\epsilon^\frac{1}{5}e^{-\frac{3}{4}s}.
\end{align*}
Putting all three terms together, we get
\begin{align*}
|F_{\tilde{U}}^{(6)}|\lesssim M^{48\frac{1}{8}}e^{-s}+\epsilon^\frac{1}{5}e^{-\frac{3}{4}s}+\epsilon^\frac{1}{5}l+\epsilon^\frac{1}{5}e^{-\frac{3}{4}s}\lesssim \epsilon^\frac{1}{5}l,
\end{align*}
the second inequality uses the choice that $l=(\log M)^{-1}$ (or $-2$) and $\epsilon$ sufficiently small according to $M$.

Applying the same procedure to $F_{\tilde{U}}^{(7)},\,F_{\tilde{U}}^{(8)}$ using the additional bootstrap assumptions \eqref{eq:assumptionnear6-8} and corresponding bounds \eqref{eq:HilbertnxLinfty} for $H[\partial_X^7U],\,H[\partial_X^8U]$, we get
\begin{align*}
|F_{\tilde{U}}^{(7)}|\lesssim \epsilon^\frac{1}{5},\quad |F_{\tilde{U}}^{(8)}|\lesssim M\epsilon^\frac{1}{5}.
\end{align*}
\subsection{Middle field} For $l\leq |X|\leq \frac{1}{2}e^{\frac{5}{4}s}$,
we first  consider $F_{\tilde{U}}$:
\begin{align*}
F_{\tilde{U}}(X,s)&=-\frac{e^{\frac{1}{4}s}}{1-\dot{\tau}}(\kappa-\dot{\xi})-\frac{e^{-s}}{1-\dot{\tau}}H[U+e^\frac{s}{4}\kappa](0,s)+\frac{e^{-s}}{1-\dot{\tau}}H[U+e^{\frac{1}{4}s}\kappa](X,s)\\&\qquad-\frac{\dot{\tau}}{1-\dot{\tau}}U_2(X)\partial_XU_2(X)-\frac{e^{\frac{1}{4}s}}{1-\dot{\tau}}(\kappa-\dot{\xi})\partial_XU_2(X)
\\
&=-\frac{e^{\frac{1}{4}s}}{1-\dot{\tau}}(\kappa-\dot{\xi})\big(1+\partial_XU_2(X)\big)+\frac{e^{-s}}{1-\dot{\tau}}\big(H[U+e^{\frac{1}{4}s}\kappa](X,s)-H[U+e^{\frac{1}{4}s}\kappa](0,s)\big)\\
&\qquad-\frac{\dot{\tau}}{1-\dot{\tau}}U_2(X)\partial_XU_2(X)\\
&\lesssim e^{\frac{1}{4}s}\epsilon^\frac{1}{5}e^{-s}+e^{-s}Me^{\frac{3}{8}s}+\epsilon^\frac{1}{5}e^{-\frac{3}{4}s}(1+X^4)^{\frac{1}{20}-\frac{1}{5}}\lesssim Me^{-\frac{5}{8}s}.
\end{align*}
Then we bound $F^{(1)}_{\tilde{U}}$, by \eqref{eq:Hilbert1xmiddle} and Lemma \ref{lem:U2},
\begin{align*}
F^{(1)}_{\tilde{U}}&=\frac{e^{-s}}{1-\dot{\tau}}H[\partial_XU]-\frac{\dot{\tau}}{1-\dot{\tau}}\big[(\partial_XU_2)^2+U_2\partial_X^2U_2\big]
-\frac{1}{1-\dot{\tau}}\big[e^{\frac{1}{4}s}(\kappa-\dot{\xi})\partial_X^2U_2+\widetilde{U}\partial_X^2U_2\big]
\\
% &\lesssim e^{-s}(1+X^4)^{-\frac{1}{5}}\big(M^4 +\log(1+X^4)\big)+e^{-s}(1+X^4)^{-\frac{1}{8}}+\epsilon^\frac{1}{5}e^{-\frac{3}{4}s}\big[(1+X^4)^{-\frac{2}{5}}+(1+X^4)^{\frac{1}{20}-\frac{9}{20}}\big]\\
% &\qquad +e^{\frac{1}{4}s}\epsilon^\frac{1}{5}e^{-s}(1+X^4)^{-\frac{9}{20}}+\epsilon^\frac{3}{20}(1+X^4)^{\frac{1}{20}-\frac{9}{20}}\\
&\lesssim e^{-s}(1+X^4)^{-\frac{1}{5}}\big(M^4 +\log(1+X^4)\big)+e^{-\frac{5}{8}s}(1+X^4)^{-\frac{1}{5}}+\epsilon^\frac{3}{20}(1+X^4)^{-\frac{2}{5}}.
\end{align*}
% where we use $(1+X^4)^{\frac{1}{5}-\frac{1}{8}}\lesssim X^\frac{3}{10}\lesssim e^{\frac{3}{8}s}$ to get the second term in the sum.
Similarly, we have
\begin{align*}
F^{(2)}_U=\frac{e^{-s}}{1-\dot{\tau}}H[\partial_X^2U]
&\lesssim e^{-s}\big[M^9(1+X^4)^{-\frac{1}{5}}+M^4(1+X^4)^{-\frac{1}{5}}\log(1+X^4)+M^{\frac{35}{4}}(1+X_4)^{-\frac{1}{8}}\big]\\
&\lesssim (M^9 e^{-s}+M^\frac{35}{4}e^{-\frac{5}{8}s})(1+X^4)^{-\frac{1}{5}}.\numberthis\label{eq:F2Umiddle}
\end{align*}
For $n=3,...,8$, by \eqref{eq:assumptionnxmiddle}\footnote{Here is the reason why we choose $n^2$ in the power of $M$. Suppose 
\begin{align*}
|\partial_X^nU(X,s)|\leq M^{a(n)}(1+X^4)^{-\frac{1}{5}}\qquad\text{for } 4\leq n\leq 8,
\end{align*}
then from later weighted estimate, we can close bootstrap if
\begin{align*}
2a(\frac{n+1}{2})\mathbbm{1}_{n\text{ odd}}+\sum_{j=2}^{\lfloor\frac{n}{2}\rfloor}a(j)+a(n-j+1)< a(n).
\end{align*}
So $a(n)=n^2$ is good.},
\begin{align*}
F^{(n)}_U&\lesssim_n \big(\|\partial_X^{n+1}U\|_{L^\infty}e^{-s}+\|\partial_X^nU\|_{L^\infty}e^{-\frac{5}{8}s}+\|\partial_X^nU\|_{L^2}e^{-\frac{5}{8}s}\big)(1+X^4)^{-\frac{1}{5}}\\
&\qquad+\frac{1}{1-\dot{\tau}}\Big[\mathbbm{1}_{n\text{ odd}}{n\choose \frac{n-1}{2}}\|\partial_X^\frac{n+1}{2}U\|_{L^\infty}^2+\sum_{j=2}^{\lfloor \frac{n}{2}\rfloor}\|\partial_X^jU\|_{L^\infty}\|\partial_X^{n-j+1}U\|_{L^\infty}  \Big](1+X^4)^{-\frac{2}{5}}\\
&\lesssim(M^{(n+1)^2}e^{-s}+M^{n^2}e^{-\frac{5}{8}s}+M^{\frac{35(n-1)}{4}}e^{-\frac{5}{8}s})(1+X^4)^{-\frac{1}{5}}\\
&\qquad+\Big(\mathbbm{1}_{n\text{ odd}}M^\frac{(n+1)^2}{2}+\sum_{j=2}^{\lfloor\frac{n}{2}\rfloor}M^{j^2+(n-j+1)^2}\Big)(1+X^4)^{-\frac{2}{5}},\numberthis\label{eq:FnUmiddle}
\end{align*}
where the $n$-dependence comes from binomial coefficients. 

\subsection{Far field}
For $|X|\geq \frac{1}{2}e^{\frac{5}{4}s}$, using the same idea, 
\begin{align}
|F_{U}^{(1)}|&\lesssim e^{-(1+\frac{3}{8})s}\lesssim e^{-\frac{5}{8}s},\label{eq:F1xUfar}\\
|F_{U}^{(2)}|&\lesssim M^4e^{-(1+\frac{3}{8}s)}\lesssim M^4e^{-\frac{5}{8}s},\nonumber\\
F_U^{(n)}&\lesssim_n M^{n^2}e^{-\frac{5}{8}s}+\sum_{j=2}^{\lfloor\frac{n+1}{2}\rfloor}M^{j^2+(n-j+1)^2}e^{-2s}\qquad{n=3,...,7},\nonumber\\
F_U^{(8)}&\lesssim M^{70}e^{-\frac{5}{8}s}+\sum_{j=2}^{4}M^{j^2+(9-j)^2}e^{-2s}.\label{eq:FnxUfar}
\end{align}

\section{Forcing bounds with parameter derivatives}\label{sec:paraforce}
\subsection{Hilbert transform terms}
We first bound $H[\partial_\alpha U+e^{\frac{1}{4}s}\partial_\alpha\kappa]$. We use the same strategy used for $H[U+e^{\frac{1}{4}s}\kappa]$ in Section \ref{sec:Hilbertbound}, i.e.\ the usual near-middle-far field decomposition, and obtain
\begin{align}
    \|H[\partial_\alpha U+e^{\frac{1}{4}s}\partial_\alpha\kappa]\|_{L^\infty}\lesssim M^4\epsilon^\frac{3}{4}e^{\frac{7}{8}s}.\label{eq:HilbertparaU}
\end{align}
Moreover, at $X=0$,
\begin{align*}
\big|H[\partial_\alpha U+e^{\frac{1}{4}s}\partial_\alpha\kappa](0,s)\big|&\lesssim \|\partial_\alpha\partial_XU\|_{L^\infty}+\|\partial_\alpha U+e^{\frac{1}{4}s}\partial_\alpha\kappa\|_{L^\infty}\int_{1<|X-Y|< e^{\frac{5}{4}s}}\frac{1}{|X-Y|}\,dY\\
&\qquad+\|\partial_\alpha U+e^{\frac{1}{4}s}\partial_\alpha\kappa\|_{L^2}\Big(\int_{|X-Y|\geq e^{\frac{5}{4}s}}\frac{1}{|X-Y|^2}\,dY\Big)^\frac{1}{2}\\
&\lesssim M^9\epsilon^\frac{3}{4}e^{\frac{3}{4}s}+M^4\epsilon^\frac{3}{4}se^{\frac{3}{4}s}+M^{15}\epsilon^\frac{3}{4}e^se^{-\frac{5}{8}s}\lesssim M^9\epsilon^\frac{3}{4}se^{\frac{3}{4}s}.\numberthis\label{eq:HilbertparaUat0}
\end{align*}
We then use interpolation to obtain
\begin{align*}
    \|H[\partial_\alpha\partial_X^nU](\cdot,s)\|_{L^\infty}&\lesssim \|\partial_\alpha\partial_XU(\cdot,s)\|_{L^2}^{1-\frac{2n-1}{14}}\|\partial_\alpha\partial_X^8U(\cdot,s)\|_{L^2}^\frac{2n-1}{14}\\
    &\lesssim M^{11n+\frac{15}{2}}\epsilon^\frac{3}{4}e^{\frac{3}{4}s}\qquad n=1,...,7.\numberthis\label{eq:HilbertparanxULinfty}
\end{align*}

\subsection{Middle field for first order parameter derivatives}
For $n=1,...,6$
\begin{align*}
\big|H[\partial_\alpha\partial_X^nU](X,s)\big|
&\lesssim M^{(n+3)^2}(1+X^4)^{-\frac{1}{5}}\epsilon^\frac{3}{4}e^{\frac{3}{4}s}+\big(M^{(n+2)^2}+M^{10\frac{5}{7}n+\frac{16}{7}}\big)e^{\frac{3}{8}s}(1+X^4)^{-\frac{1}{5}}\epsilon^\frac{3}{4}e^{\frac{3}{4}s},
\end{align*}
and 
\begin{align*}
    \big|H[\partial_\alpha\partial_X^7U](X,s)\big|&\lesssim M^{90}(1+X^4)^{-\frac{1}{10}}\epsilon^\frac{3}{4}e^{\frac{3}{4}s}+M^{81}e^{\frac{3}{8}s}(1+X^4)^{-\frac{1}{5}}\epsilon^\frac{3}{4}e^{\frac{3}{4}s}.
\end{align*}
So for $F_{U,\alpha}^{(1)}$ with the weight, we have
\begin{align*}
|(1+X^4)^\frac{1}{5}F_{U,\alpha}^{(1)}|&\lesssim M^{16}\epsilon^\frac{3}{4}e^{-\frac{1}{4}s}+(M^9+M^{13})\epsilon^\frac{3}{4}e^{\frac{1}{8}s}+\epsilon^\frac{1}{2}(M^4e^{-s}+Me^{-\frac{5}{8}s}+\sqrt{7}e^{-\frac{5}{8}s})\\
&\quad+M^4(M^4\epsilon^\frac{3}{4}e^{\frac{3}{4}s}+M\epsilon^\frac{1}{2}e^{\frac{1}{4}s})+\epsilon^\frac{1}{2}(1+X^4)^{-\frac{1}{5}}+\epsilon^\frac{1}{2}M^4(2Me^{\frac{1}{4}s}+\epsilon^\frac{1}{5}e^{-\frac{3}{4}s})\\
&\lesssim M^8\epsilon^\frac{3}{4}e^{\frac{3}{4}s}+\epsilon^\frac{1}{2}(1+X^4)^{-\frac{1}{5}}+\text{non-dominant terms}.
\end{align*}
We can use the Hilbert bounds above, \eqref{eq:assumptionparanxmiddle} and \eqref{eq:assumptionnxmiddle} to estimate $F_{U,\alpha}^{(n)}$ for $n=2,...,7$ all together (the Hilbert part is different for $n=7$ but we omit the separate calculation because it is non-dominant)
\begin{align*}
|(1+X^4)^\frac{1}{5}F_{U,\alpha}^{(n)}|&\lesssim_n M^{(n+3)^2}\epsilon^\frac{3}{4}e^{-\frac{1}{4}s}+\big(M^{(n+2)^2}+M^{10\frac{5}{7}n+\frac{16}{7}}\big)\epsilon^\frac{3}{4}e^{\frac{1}{8}s}\\
&\qquad+\epsilon^\frac{1}{2}\big(M^{(n+1)^2}e^{-s}+M^{n^2}e^{-\frac{5}{8}s}+M^{\frac{35(n-1)}{4}}e^{-\frac{5}{8}s}\big)\\
&\qquad+M^{(n+1)^2}\big(M^4 \epsilon^\frac{3}{4}e^{\frac{3}{4}s}+M\epsilon^\frac{1}{2}e^{\frac{1}{4}s}\big)+\epsilon^\frac{1}{2}M^{(n+1)^2}(2Me^{\frac{1}{4}s}+\epsilon^\frac{1}{5}e^{-\frac{3}{4}s})\\
&\qquad+(1+X^4)^{-\frac{1}{5}}\epsilon^\frac{3}{4}e^{\frac{3}{4}s}\sum_{j=1}^{n-1}M^{(j+2)^2+(n-j+1)^2}+\epsilon^\frac{1}{2}(1+X^4)^{-\frac{1}{5}}\sum_{j=1}^{\lfloor \frac{n+1}{2}\rfloor} M^{j^2+(n-j+1)^2}\\
&\lesssim_n \big(M^{(n+1)^2+4}+\sum_{j=1}^{n-1} M^{(j+2)^2+(n-j+1)^2}\big)\epsilon^\frac{3}{4}e^{\frac{3}{4}s}+\text{non-dominant terms}.\numberthis\label{eq:forcingparamiddle}
\end{align*}
The $M$ powers are below $(n+2)^2$. 

\subsection{Far field for first order parameter derivatives}
First we estimate the Hilbert transforms
\begin{align*}
\big|H[\partial_\alpha\partial_X^nU](X,s)\big|&\lesssim M^{(n+2)^2}\epsilon^\frac{3}{4}e^{\frac{3}{8}s},\qquad n=1,...,6,\\
\big|H[\partial_\alpha\partial_X^7U](X,s)\big|&\lesssim M^{90}\epsilon^\frac{3}{4}e^{\frac{1}{4}s}+M^{81}\epsilon^\frac{3}{4}e^{\frac{3}{8}s},
\end{align*}
following the same argument as before without parameter derivative. Then
\begin{align*}
|e^s F_{U,\alpha}^{(1)}|&\lesssim M^9\epsilon^\frac{3}{4}e^{\frac{3}{8}s}+\epsilon^\frac{1}{2}e^{-\frac{3}{8}s}+M^4(M^4\epsilon^\frac{3}{4}e^{\frac{3}{4}s}+M\epsilon^\frac{1}{2}e^{\frac{1}{4}s})+\epsilon^\frac{1}{2}e^{-s}+\epsilon^\frac{1}{2}M^4(2Me^{\frac{1}{4}s}+\epsilon^\frac{1}{5}e^{-\frac{3}{4}s}),\\
|e^sF^{(n)}_{U,\alpha}|&\lesssim_n M^{(n+2)^2}e^{\frac{3}{8}s}+\epsilon^\frac{1}{2}M^{n^2}e^{-\frac{3}{8}s}+M^{(n+1)^2}\big(M^4 \epsilon^\frac{3}{4}e^{\frac{3}{4}s}+M\epsilon^\frac{1}{2}e^{\frac{1}{4}s}\big)\\
&\quad+\epsilon^\frac{1}{2}M^{(n+1)^2}(2Me^\frac{1}{4}+\epsilon^\frac{1}{5}e^{-\frac{3}{4}s})+\sum_{j=1}^{n-1}M^{(j+2)^2+(n-j+1)^2}\epsilon^\frac{3}{4}e^{\frac{3}{4}s}+\epsilon^\frac{1}{2}\sum_{j=1}^{\lfloor\frac{n+1}{2}\rfloor}M^{j^2+(n-j+1)^2}e^{-s}.
\end{align*}
We only need to look at dominant terms. Then for $n=1,...,7$,
\begin{align}
|e^sF_{U,\alpha}^{(n)}|\lesssim_n \big(M^{4+(n+1)^2}+\sum_{j=1}^{n-1}M^{(j+2)^2+(n-j+1)^2}\big)\epsilon^\frac{3}{4}e^{\frac{3}{4}s}.\label{eq:forcingparafar}
\end{align}

\subsection{Forcing for second order parameter derivatives}
We first use interpolation to bound Hilbert transform terms. We have
\begin{align*}
    \|H[\partial_{\alpha\beta}^2\partial_X^nU]\|_{L^\infty}&\lesssim \|\partial_{\alpha\beta}^2\partial_XU\|_{L^2}^{\frac{11}{12}-\frac{1}{6}n}\|\partial_{\alpha\beta}^2\partial_X^7U\|_{L^2}^{\frac{1}{6}n+\frac{1}{12}}\\
    &\lesssim M^{\frac{31}{2}n+44\frac{3}{4}}\epsilon^\frac{3}{2}e^{\frac{3}{2}s},\qquad n=1,...,6.
\end{align*}
And using the near-middle-far decomposition, we have (the $\log$ term in the middle field dominates)
\begin{align*}
\|H[\partial_{\alpha\beta}^2U+e^{\frac{1}{4}s}\partial_{\alpha\beta}^2\kappa]\|_{L^\infty}\lesssim M^{25}\epsilon^\frac{3}{2}se^{\frac{3}{2}s}.
\end{align*}

\section{Closure of bootstrap without parameter derivatives}\label{sec:closure}
\subsection{$L^2$ estimates}
\begin{proposition}[uniform-in-$s$ $L^2$ bound of $\partial_XU$] We have
\begin{align*}
    \|\partial_XU(\cdot,s)\|_{L^2}\leq \sqrt{7}.
\end{align*}
\end{proposition}
\begin{proof}
Consider \eqref{eq:eqUderivative}. Take $L^2$ inner product with $\partial_XU$ itself, we get
\begin{gather*}
% \Big(\partial_s+1+\frac{\partial_XU}{1-\dot{\tau}}\Big)\partial_XU+\Big(\frac{U+e^\frac{s}{4}(\kappa-\dot{\xi})}{1-\dot{\tau}}+\frac{5}{4}X\Big)\partial_X^2U=\frac{e^{-s}}{1-\dot{\tau}}H[\partial_XU]\\
\frac{1}{2}\frac{d}{ds}\|\partial_XU\|_{L^2}^2+\|\partial_XU\|_{L^2}^2+\frac{1}{1-\dot{\tau}}\int_\mathbb{R}(\partial_XU)^3\,dX-\frac{1}{2(1-\dot{\tau})}\int_\mathbb{R}(\partial_XU)^3\,dX-\frac{5}{8}\|\partial_XU\|_{L^2}^2=0\\
\frac{d}{ds}\|\partial_XU\|_{L^2}^2+\frac{3}{4}\|\partial_XU\|_{L^2}^2=-\frac{1}{1-\dot{\tau}}\int_\mathbb{R}(\partial_XU)^3.\,dX\numberthis\label{eq:1xinner}
\end{gather*}
For $|X|\leq \frac{1}{2}e^{\frac{5}{4}s}$, by \eqref{eq:conse1xULinfty}
\begin{align*}
|\partial_XU(X,s)|\leq (1+\epsilon^\frac{17}{20})(1+X^4)^{-\frac{1}{5}},
\end{align*}
and for $|X|\geq \frac{1}{2}e^{\frac{5}{4}s}$ we have \eqref{eq:assumption1xfar}, so the right hand side of \eqref{eq:1xinner} is
\begin{align*}
\frac{1}{1-\dot{\tau}}\Big|\int_\mathbb{R}(\partial_XU)^3\,dX\Big|&\leq (1+2\epsilon^\frac{1}{5})(1+\epsilon^\frac{17}{20})\int_{|X|\leq \frac{1}{2}e^{\frac{5}{4}s}}(1+X^4)^{-\frac{3}{5}}\,dX\\
&\qquad+4(1+2\epsilon^\frac{1}{5})\int_{|X|\geq \frac{1}{2}e^{\frac{5}{4}s}}e^{-s}(\partial_XU)^2\,dX\\
&\leq 4,
\end{align*}
because
\begin{align*}
\int_\mathbb{R}(1+X^4)^{-\frac{3}{5}}\,dX=\frac{2\Gamma(\frac{7}{20})\Gamma(\frac{5}{4})}{\Gamma(\frac{3}{5})}\approx 3.09944.
\end{align*}
So
\begin{align*}
\frac{d}{ds}\|\partial_XU\|_{L^2}^2+\frac{5}{8}\|\partial_XU\|_{L^2}^2&\leq 4\\
\|\partial_XU(\cdot,s)\|_{L^2}^2&\leq \|\partial_XU(\cdot,-\log\epsilon)\|_{L^2}^2e^{-\frac{5}{8}(s+\log\epsilon)}+4\int_{-\log\epsilon}^se^{-\frac{5}{8}(s-s')}\,ds'\\
&=\frac{32}{5}+\big(\|\partial_XU(\cdot,-\log\epsilon)\|_{L^2}^2-\frac{32}{5}\big)e^{-\frac{5}{8}(s+\log\epsilon)}.
\end{align*}
Since $\|\partial_XU(\cdot,-\log\epsilon)\|_{L^2}^2\leq 6<\frac{32}{5}$, then the right hand side is increasing and tend to $\frac{32}{5}$ as $s\to \infty$. Then $\|\partial_XU(\cdot,s)\|_{L^2}^2\leq \frac{32}{5}<7$.
\end{proof}

\begin{proposition}[Uniform-in-$s$ $L^2$ bound of $\partial_X^9U$]
We have
\begin{align*}
    \|\partial_X^9U(\cdot,s)\|_{L^2}\leq M^{70}.
\end{align*}
\end{proposition}
\begin{proof}
Taking $L^2$ inner product of \eqref{eq:eqUderivative} with $\partial_X^9U$ and integrating by parts, we have
\begin{align*}
% \Big(\partial_s+11+\frac{10\partial_XU}{1-\dot{\tau}}\Big)\partial_X^9U+\Big(\frac{U+e^{\frac{1}{4}s}(\kappa-\dot{\xi})}{1-\dot{\tau}}+\frac{5}{4}X\Big)\partial_X^{10}U&=
% \frac{e^{-s}}{1-\dot{\tau}}H[\partial_X^9U]-\frac{126}{1-\dot{\tau}}(\partial_X^5U)^2\\&\quad-\frac{1}{1-\dot{\tau}}\sum_{j=2}^4{10\choose j}\partial_X^jU\partial_X^{10-j}U,\\
\frac{1}{2}\frac{d}{ds}\|\partial_X^9U\|_{L^2}^2+11\|\partial_X^9U\|_{L^2}^2+\frac{10}{1-\dot{\tau}}\int_\mathbb{R}\partial_XU(\partial_X^9U)^2\,dX-&\frac{5}{8}\|\partial_X^9U\|_{L^2}^2\\
&-\frac{1}{2(1-\dot{\tau})}\int_\mathbb{R}\partial_XU(\partial_X^9U)^2\,dX\\
=-\frac{126}{1-\dot{\tau}}\int_\mathbb{R}(\partial_X^5U)^2\partial_X^9U\,dX-\frac{1}{1-\dot{\tau}}\sum_{j=2}^4{10\choose j}&\int_\mathbb{R}\partial_X^jU\partial_X^{10-j}U\partial_X^9U\,dX,\\
\frac{1}{2}\frac{d}{ds}\|\partial_X^9U\|_{L^2}^2+10\frac{3}{8}\|\partial_X^9U\|_{L^2}^2+\frac{19}{2(1-\dot{\tau})}\int_\mathbb{R}\partial_XU(\partial_X^9U)^2\,dX&=\text{right hand side above}.
\end{align*}
One way to bound the right hand side is to use $L^\infty$-$L^2$-$L^2$ H\"older inequality, and we get
\begin{align*}
    \frac{1}{2}\frac{d}{ds}\|\partial_X^9U\|_{L^2}^2+\frac{3}{4}\|\partial_X^9U\|_{L^2}^2&\leq C(M^{25+\frac{1}{2}\cdot 70}+M^{4+\frac{7}{8}\cdot 70}+M^{9+\frac{3}{4}\cdot 70} +M^{16+\frac{5}{8}\cdot 70})\|\partial_X^9U\|_{L^2}\\
    \frac{d}{ds}\|\partial_X^9U\|_{L^2}+\frac{3}{4}\|\partial_X^9U\|_{L^2}&\leq CM^{65\frac{1}{4}} \\
    \|\partial_X^9U(\cdot,s)\|_{L^2}&\leq \|\partial_X^9U(\cdot,-\log\epsilon)\|_{L^2}e^{-\frac{3}{4}(s+\log\epsilon)}+CM^{65\frac{1}{4}}\int_{-\log\epsilon}^se^{-\frac{3}{4}(s-s')}\,ds'\\
    &=\frac{4}{3}CM^{65\frac{1}{4}}+\big(\|\partial_X^9U(\cdot,-\log\epsilon)\|_{L^2}-\frac{4}{3}CM^{65\frac{1}{4}}\big)e^{-\frac{3}{4}(s+\log\epsilon)}\\
    &\leq \frac{4}{3}CM^{65\frac{1}{4}}\leq \frac{1}{2}M^{70},
\end{align*}
since $\|\partial_X^9U(\cdot,-\log\epsilon)\|_{L^2}\leq M^{65}< \frac{4}{3}CM^{65\frac{1}{4}}$.
\end{proof}

\subsection{Lagrangian trajectory}
Let $\Phi:\mathbb{R}\times[s_0,\infty)\to\mathbb{R}$ be the Lagrangian trajectory of $U$, i.e. for each starting point $X_0$, $\Phi(X_0,s)$ is the position of the particle at time $s$ such that
\begin{equation}
\begin{aligned}
\frac{d}{ds}\Phi(X_0,s)&=V\circ \Phi(X_0,s),\\
\Phi(X_0,s_0)&=X_0,
\end{aligned}\label{eq:trajectory}
\end{equation}
where $V$ is the transport speed defined in \eqref{eq:speed}.

\begin{lemma}[lower bound on transport speed]
For $|X_0|\geq l$, we have
\begin{align*}
|\Phi(X_0,s)|\geq |X_0|e^{\frac{1}{5}(s-s_0)}.
\end{align*}
%Moreover, for $(X_0,s_0)$ such that $|X_0|\geq \frac{1}{2}e^{s_0}$, we have
%\begin{align*}
%|\Phi(X_0,s)|\geq |X_0|e^{\frac{23}{16}(s-s_0)}.
%\end{align*}
In other words, once a particle is at least away from 0 by a distance $l$, it will escape to infinity exponentially fast. 
\label{lem:lowerLagrangian}
\end{lemma}

\begin{proof}
By the mean value theorem and \eqref{eq:conse1xULinfty},
%the consequence of bootstrap assumptions $\|\partial_XU\|_{L^\infty}=1$, and the constraint $U(0,s)=0$, 
we have
\begin{align*}
|U(X,s)|\leq |U(0,s)|+ \|\partial_XU(\cdot,s)\|_{L^\infty}|X|\leq (1+\epsilon^\frac{17}{20})|X|.
\end{align*}
By \eqref{eq:eqkappadot}, \eqref{eq:assumptionx=0}, we have
\begin{align*}
e^{\frac{1}{4}s}(\kappa-\dot{\xi})&=e^{-\frac{3}{4}s}\dot{\kappa}-e^{-s}H[U+e^{\frac{1}{4}s}\kappa](0,s)\\
&=\frac{e^{-s}H[\partial_X^4U](0,s)-10\partial_X^2U(0,s)\partial_X^3U(0,s)}{\partial_X^5U(0,s)}\\
&\leq \frac{Ce^{-s}M^{30\frac{5}{8}}+10\epsilon^\frac{1}{10}e^{-\frac{3}{4}s}\cdot M^{27}e^{-s}}{120-\epsilon^\frac{1}{2}}\\
&\leq CM^{30\frac{5}{8}}e^{-s}\leq \frac{1}{10}l.
\end{align*}
Then for $|X|\geq l$, we have
\begin{align*}
|V(X,s)|&\geq \frac{5}{4}|X|-(1+2\epsilon^\frac{3}{4})(1+\epsilon^\frac{17}{20})|X|-(1+2\epsilon^\frac{3}{4})\frac{1}{10}|X|\\
&=\frac{5}{4}|X|-(1+2\epsilon^\frac{3}{4})(1+\epsilon^\frac{17}{20}+\frac{1}{10})|X|\\
&\geq \frac{1}{5}|X|.
\end{align*}
Hence,
\begin{align*}
\frac{1}{2}\frac{d}{ds}\Phi^2(X_0,s)&=\big(V\circ\Phi(X_0,s)\big)\Phi(X_0,s)\geq \frac{1}{5}\Phi^2(X_0,s)\\
\implies\qquad|\Phi(X_0,s)|^2&\geq|X_0|^2e^{\frac{2}{5}(s-s_0)}\\
|\Phi(X_0,s)|&\geq|X_0|e^{\frac{1}{5}(s-s_0)}.
\end{align*}
\end{proof}

\begin{lemma}[upper bound on Lagrangian trajectory]
For all $(X_0,s_0)$,
%such that $X_0\geq \frac{1}{2}e^{s_0}$, 
we have
\begin{align*}
|\Phi(X_0,s)|\leq (|X_0|+3Me^{\frac{1}{4}s_0})e^{\frac{5}{4}(s-s_0)}.
\end{align*}\label{lem:upperLagrangian}
\end{lemma}

\begin{proof}
Multiplying the first equation of \eqref{eq:trajectory} by the integrating factor $e^{-\frac{5}{4}s}$ and plugging in the expression for $V$, we get
\begin{align*}
\frac{d}{ds}\big(e^{-\frac{5}{4}s}\Phi(X_0,s)\big)&=\frac{1}{1-\dot{\tau}}\Big(e^{-\frac{5}{4}s}U\circ\Phi(X_0,s)+e^{-s}\kappa-e^{-s}\dot{\xi}\Big)\\
e^{-\frac{5}{4}s}\Phi(X_0,s)&=e^{-\frac{5}{4}s_0}X_0+\frac{1}{1-\dot{\tau}}\int_{s_0}^se^{-s'}\Big(e^{-\frac{1}{4}s'}U\circ\Phi(X_0,s')+\kappa\Big)-e^{-s'}\dot{\xi}\,ds',\\
|\Phi(X_0,s)-X_0e^{\frac{5}{4}(s-s_0)}|&\leq 3e^{\frac{5}{4}s}\int_{s_0}^s Me^{-s'}\,ds'\\
&=3Me^{\frac{5}{4}s}(e^{-s_0}-e^{-s})\\
&=3Me^{\frac{1}{4}s_0}(1-e^{-(s_0-s)})e^{\frac{5}{4}(s-s_0)}\leq 2Me^{\frac{1}{4}s_0}e^{\frac{5}{4}(s-s_0)}.
\end{align*}
By triangle inequality, the proof is complete.
\end{proof}

\subsection{$L^\infty$ estimates of $U$}
\begin{proposition}
The $L^\infty$-norm of $e^{-\frac{1}{4}s}U(\cdot,s)+\kappa$ is bounded uniformly in $s$. More precisely,
\begin{align*}
\|e^{-\frac{1}{4}s}U(\cdot,s)+\kappa\|_{L^\infty}\leq M,\quad\text{i.e.}\quad
\|U(\cdot,s)+e^\frac{s}{4}\kappa\|_{L^\infty}\leq Me^{\frac{1}{4}s}.
\end{align*}
\end{proposition}

\begin{proof}
We consider the equation for $e^{-\frac{1}{4}s}U(\cdot,s)+\kappa$ 
\begin{align}
    \partial_s(e^{-\frac{1}{4}s}U+\kappa)+V\partial_X(e^{-\frac{1}{4}s}U+\kappa)=\frac{e^{-\frac{5}{4}s}}{1-\dot{\tau}}H[U+e^{\frac{1}{4}s}\kappa].\label{eq:U+kappa}
\end{align}
Recall that $|H[U+e^\frac{s}{4}\kappa](X,s)|\lesssim Me^{\frac{3}{8}s}$, so we can bound the forcing term on the right hand side
\begin{align*}
\Big|\frac{e^{-\frac{5}{4}s}}{1-\dot{\tau}}H[U+e^\frac{s}{4}\kappa]\Big|\lesssim Me^{-\frac{7}{8}s}.
\end{align*}
Composing \eqref{eq:U+kappa} with the Lagrangian trajectory, we have
\begin{align*}
    |(e^{-\frac{1}{4}s}U+\kappa)\circ\Phi(X_0,s)|&\leq |e^{\frac{1}{4}\log{\epsilon}}U(X_0,-\log{\epsilon})+\kappa(0)|\\
    &\qquad+\int_{-\log{\epsilon}}^s\frac{e^{-\frac{5}{4}s'}}{1-\dot{\tau}}\big|H[U+e^\frac{s'}{4}\kappa]\circ\Phi(X_0,s')\big|\,ds'\\
    &\leq \frac{M}{2}+C\int_{-\log{\epsilon}}^\infty Me^{-\frac{7}{8}s'}\,ds'\\
    &\leq \frac{M}{2}+CM\epsilon^\frac{7}{8}\leq \frac{3}{4}M 
\end{align*}
for all $X_0$. We thus close the assumption.
\end{proof}

\subsection{Pointwise estimates: near field}
\begin{proposition}[$\partial_X^n\widetilde{U}$ for $n=6,7,8$] 
For all $(X,s)$ such that $|X|\leq l$, we have
\begin{align*}
    |\partial_X^6\widetilde{U}(X,s)|\leq \epsilon^\frac{1}{5},\quad |\partial_X^7\widetilde{U}(X,s)|\leq M\epsilon^\frac{1}{5},\quad |\partial_X^8\widetilde{U}(X,s)|\leq M^3\epsilon^\frac{1}{5}.
\end{align*}
\end{proposition}

\begin{proof}
We begin with $\partial_X^6\widetilde{U}$. The damping is
\begin{align*}
D_{\tilde{U}}^{(6)}=\frac{29}{4}+\frac{7\partial_XU}{1-\dot{\tau}}\geq \frac{29}{4}-7(1+2\epsilon^\frac{3}{4})(1+\epsilon^\frac{17}{20}) \geq \frac{1}{8}.
\end{align*}
Observe that if $|\Phi(X_0,s)|\leq l$ then $|X_0|\leq l$ with $s_0=-\log\epsilon$. Composing the equation with the Lagragian trajectory and using the forcing bound $|F_{\tilde{U}}^{(6)}|\lesssim \epsilon^\frac{1}{5}l$, we get
\begin{align*}
|\partial^6_X\widetilde{U}\circ \Phi(X_0,s)|\leq |\partial_X^6\widetilde{U}(X_0,-\log\epsilon)|e^{-\frac{1}{16}(s+\log\epsilon)}+C\epsilon^\frac{1}{5}l\int_{s_0}^se^{-\frac{1}{16}(s-s')}\,ds'.
\end{align*}
We then use the initial condition that when $|X|\leq l\ll 1$,
\begin{align*}
    |\partial_X^6\widetilde{U}(X,-\log\epsilon)|&=\big|\partial_X^6\big(\widehat{U}_0(X)+\alpha X^2+\beta X^3\big)\big|\\
    &=|\partial_X^6\widehat{U}_0(X)|\leq \epsilon,
\end{align*}
to obtain
\begin{align*}
|\partial_X^6\widetilde{U}\circ\Phi(X_0,s)|\leq \epsilon+C\epsilon^\frac{1}{5}l\leq \frac{3}{4}\epsilon^\frac{1}{5},
\end{align*}
where we used $l=(\log M)^{-2}$ and take $M$ sufficiently large. We carry out the similar procedures to $\partial_X^7\widetilde{U}$ and $\partial_X^8\widetilde{U}$. The damping terms are bounded from below by
\begin{align*}
|D_{\tilde{U}}^{(7)}|=\frac{5}{4}\cdot 6+1 -8(1+2\epsilon^\frac{3}{4})(1+\epsilon^\frac{17}{20})\geq \frac{1}{4},\\
|D_{\tilde{U}}^{(7)}|=\frac{5}{4}\cdot 7+1-9(1+2\epsilon^\frac{3}{4})(1+\epsilon^\frac{17}{20})\geq \frac{1}{2},
\end{align*}
and the initial conditions are both $\leq \epsilon$.
\end{proof}

\begin{proposition}[$\partial_X^n\widetilde{U}$ for $n=0,...,5$]
For $|X|\leq l$, we have
\begin{align*}
    |\partial_X^n\widetilde{U}(X,s)|\leq \epsilon^\frac{1}{5}|X|^{6-n}+\epsilon^\frac{1}{2}\leq 2l^{6-n}\epsilon^\frac{1}{5},\qquad\text{for }n=0,...,5.
\end{align*}
\end{proposition}

\begin{proof}
Using the fundamental theorem of calculus, for $|X|\leq l$
\begin{align*}
\partial_X^5\widetilde{U}(X,s)=\partial_X^5\widetilde{U}(0,s)+\int_0^X\partial_X^6\widetilde{U}(Y,s)\,ds\leq \epsilon^\frac{1}{2}+l\epsilon^\frac{1}{5}\leq \frac{5}{4}l\epsilon^\frac{1}{5}.
\end{align*}
Using the fundamental theorem of calculus repeatedly, we get
\begin{align*}
\partial_X^4\widetilde{U}(X,s)&=\partial_X^4\widetilde{U}(0,s)+\int_0^X\partial_X^5\widetilde{U}(Y,s)\,ds\leq \frac{5}{4}l^2\epsilon^\frac{1}{5},\\
\partial_X^3\widetilde{U}(X,s)&=\partial_X^3\widetilde{U}(0,s)+\int_0^X\partial_X^4\widetilde{U}(Y,s)\,ds\leq 
M^{27} e^{-s}+
\frac{5}{4}l^3\epsilon^\frac{1}{5}\leq \frac{3}{2}l^3\epsilon^\frac{1}{5},\\
\partial_X^2\widetilde{U}(X,s)&=\partial_X^2\widetilde{U}(0,s)+\int_0^X\partial_X^3\widetilde{U}(Y,s)\,ds\leq \epsilon^\frac{1}{10} e^{-\frac{3}{4}s}+
\frac{3}{2}l^4\epsilon^\frac{1}{5}\leq \frac{7}{4}l^4\epsilon^\frac{1}{5},\\
\partial_X\widetilde{U}(X,s)&=\partial_X\widetilde{U}(0,s)+\int_0^X\partial_X^2\widetilde{U}(Y,s)\,ds\leq \frac{7}{4}l^5\epsilon^\frac{1}{5},\\
\widetilde{U}(X,s)&=\widetilde{U}(0,s)+\int_0^X\partial_X\widetilde{U}(Y,s)\,ds\leq \frac{7}{4}l^6\epsilon^\frac{1}{5}.
\end{align*}
%Because this is $\epsilon^\frac{1}{5}$ (all come from $\partial_X^6\widetilde{U}$), middle field should not exceed this power. Of course we can raise it slightly, but not any helpful, by all means it must be smaller that $\frac{1}{4}$.
\end{proof}

\subsection{Pointwise estimates: middle field}
We introduce the weight $(1+X^4)^\mu$ for $\mu\neq 0$. Suppose we have a transport equation for function $f(X,s)$
\begin{align*}
\partial_sf+D_f f+V\partial_X f=F_f
\end{align*}
where $D_f$ and $F_f$ denote the damping and forcing terms respectively. We consider $$g(X,s):=(1+X^4)^\mu f(X,s)$$ for some $\mu\in\mathbb{R}$. Then 
\begin{align*}
\partial_sg+(\underbracket[0.55pt]{D_f-\frac{4\mu X^3}{1+X^4}V}_{:=D_{f,\mathrm{w}}})g+V\partial_Xg=\underbracket[0.55pt]{(1+X^4)^\mu F_f}_{:=F_{f,\mathrm{w}}},
\end{align*}
where the ``w" stands for ``weighted".
\begin{proposition}[weighted estimate of $\widetilde{U}$]\label{prop:weitedUtilde}
For $(X,s)$ such that $l\leq |X|\leq \frac{1}{2}e^{\frac{5}{4}s}$, we have
\begin{align*}
    |\widetilde{U}(X,s)|\leq \epsilon^\frac{3}{20}(1+X^4)^\frac{1}{20}.
\end{align*}
\end{proposition}
\begin{proof}
We take
$f=\widetilde{U}$ and $\mu=-\frac{1}{20}$ in the aforementioned framework. From \eqref{eq:eqtilde},
\begin{align*}
    D_{\tilde{U}}=-\frac{1}{4}+\frac{\partial_XU_2}{1-\dot{\tau}},
\end{align*}
So the new damping for $g=(1+X^4)^{-\frac{1}{20}}\widetilde{U}$ is
\begin{align*}
D_{\widetilde{U},\mathrm{w}}&=-\frac{1}{4}+\frac{\partial_XU_2}{1-\dot{\tau}}+\frac{X^4}{4(1+X^4)}+\frac{X^3}{5(1+X^4)}\frac{U+e^{\frac{1}{4}s}(\kappa-\dot{\xi})}{1-\dot{\tau}}\\
&=-\frac{1}{4(1+X^4)}+\frac{\partial_X U_2}{1-\dot{\tau}}+\frac{X^3}{5(1+X^4)}\frac{U+e^{\frac{1}{4}s}(\kappa-\dot\xi)}{1-\dot{\tau}},\\
|D_{\widetilde{U},\mathrm{w}}|&\leq \frac{1}{4}(1+X^4)^{-1}+(1+2\epsilon^\frac{1}{5})(1+X^4)^{-\frac{1}{5}}\\
&\qquad+\frac{1}{5}(1+X^4)^{-\frac{1}{4}}(1+2\epsilon^\frac{1}{5})\Big((1+\epsilon^\frac{3}{20})(1+X^4)^\frac{1}{20}+\epsilon^\frac{1}{5}e^{-\frac{3}{4}s}\Big)\\
\big(e^{-\frac{3}{4}s}\leq 2(1+X^4)^{-\frac{3}{20}}\big)\quad &\leq 2(1+X^4)^{-\frac{1}{5}}.
\end{align*}
Since $F_{\tilde{U}}\lesssim Me^{-\frac{5}{8}s}$ from \eqref{eq:forcingU},
\begin{align*}
|(1+X^4)^{-\frac{1}{20}}\widetilde{U}\circ\Phi(X_0,s)|&\leq (1+X_0^4)^{-\frac{1}{20}}|\widetilde{U}(X_0,s_0)|\exp{\Big(\int_{s_0}^s 2\big(1+X_0^4e^{\frac{4}{5}(s'-s_0)}\big)^{-\frac{1}{5}}ds'\Big)}\\
&+CM\int_{s_0}^s e^{-\frac{5}{8}s'}\big(1+X_0^4e^{\frac{4}{5}(s'-s_0)}\big)^{-\frac{1}{20}}\exp{\Big(\int_{s'}^s 2\big(1+X_0^4e^{\frac{4}{5}(s''-s_0)}\big)^{-\frac{1}{5}}ds'' \Big)}ds'.
\end{align*}
For $|X_0|\geq l$ (note that $0<l\ll 1$), 
\begin{align*}
\int_{s_0}^s\big(1+X_0^4e^{\frac{4}{5}(s'-s_0)}\big)^{-\frac{1}{5}}\,ds'&\leq \int_{s_0}^s\big(1+l^4e^{\frac{4}{5}(s'-s_0)}\big)^{-\frac{1}{5}}\,ds'\\
&\leq \int_{s_0}^{s_0+5\log\frac{1}{l}}1\,ds,+\int_{s_0+5\log\frac{1}{l}}^sl^{-\frac{4}{5}}e^{-\frac{4}{25}(s'-s_0)}\,ds'\\
&\leq 5\log\frac{1}{l}+\frac{25}{4}\leq 10\log\frac{1}{l},
\end{align*}
if $l$ is sufficiently small (corresponding to $M$ being sufficiently large). Same bound holds if we replace $ds'$ by $ds''$ and replace the lower bound $s_0$ by $s'$.
From the initial data \eqref{eq:initialdata}, we have
\begin{align*}
    \big|\widetilde{U}_{\alpha,\beta}(X,-\log\epsilon)\big|&=\big| U_2(X)\chi(2\epsilon^\frac{5}{4}X)-U_2(X)+\widehat{U}_0(X)+\chi(X)(\alpha X^2+\beta X^3)\big|\\
    &=\big|\widehat{U}_0(X)+\chi(X)(\alpha X^2 + \beta X^3)\big|\\
    &\leq \epsilon(1+X^4)^\frac{1}{20}+\chi(X)\big(\frac{4}{5}M^{15}\epsilon X^2+\frac{6}{5}M^{25}\epsilon X^3\big)\\
    &\leq \frac{1}{2}\epsilon^\frac{1}{5}(1+X^4)^\frac{1}{20}+\frac{1}{2}\epsilon^\frac{1}{5}(1+X^4)^\frac{1}{20}\mathbbm{1}_{|X|\leq 2}\\
    &\leq \epsilon^\frac{1}{5}(1+X^4)^\frac{1}{20}\qquad \text{if } l\leq |X|\leq \frac{1}{2}\epsilon^{-\frac{5}{4}},
\end{align*}
where the second equality comes from $\chi(2\epsilon^\frac{5}{4}X)=1$ when $|X|\leq \frac{1}{2}\epsilon^{-\frac{5}{4}}$ and $|\widetilde{U}(l,s)|\leq 2l^6\epsilon^\frac{1}{5}$. Then we get
\begin{align*}
|(1+X^4)^{-\frac{1}{20}}\widetilde{U}\circ\Phi(X_0,s)|&\leq 2l^6 \epsilon^\frac{1}{5}l^{-20}+CMl^{-20}\int_{s_0}^se^{-\frac{5}{8}s'}\big(1+l^4e^{\frac{4}{5}(s'-s_0)}\big)^{-\frac{1}{20}}\,ds'\\
&\leq 2\epsilon^\frac{1}{5}l^{-14}+CMl^{-20}\epsilon^\frac{5}{8}\leq \epsilon^\frac{3}{20},
\end{align*}
which closes the bootstrap.
\end{proof}

\begin{proposition}[weighted estimate of $\partial_X\widetilde{U}$]
For $(X,s)$ such that $l\leq |X|\leq \frac{1}{2}e^{\frac{5}{4}s}$, we have
\begin{align*}
    |\partial_X\widetilde{U}(X,s)|\leq \epsilon^\frac{1}{20}(1+X^4)^{-\frac{1}{5}}.
\end{align*}
\end{proposition}
\begin{proof}
We first bound damping
\begin{align*}
D_{\partial_X\widetilde{U},\mathrm{w}}&=1+\frac{2\partial_XU_2+\partial_X\widetilde{U}}{1-\dot{\tau}}-\frac{\frac{4}{5}X^3}{1+X^4}\Big(\frac{U+e^\frac{s}{4}(\kappa-\dot{\xi})}{1-\dot{\tau}}+\frac{5}{4}X\Big)\\
&=\frac{1}{1+X^4}+\frac{\partial_XU_2+\partial_XU}{1-\dot{\tau}}-\frac{4X^3}{5(1-\dot{\tau})(1+X^4)}\Big(U+e^\frac{s}{4}(\kappa-\dot{\xi})\Big),\\
|D_{\partial_X\widetilde{U},\mathrm{w}}| &\leq (1+X^4)^{-\frac{1}{5}}+(1+2\epsilon^\frac{1}{5})(2+\epsilon^\frac{1}{20})(1+X^4)^{-\frac{1}{5}}\\
&\qquad+\frac{4}{5}(1+2\epsilon^\frac{1}{5})(1+X^4)^{-\frac{1}{4}}\Big((1+\epsilon^\frac{3}{20})(1+X^4)^\frac{1}{20}+CM^{30\frac{5}{8}}e^{-s}\Big)\\
&\leq 4(1+X^4)^{-\frac{1}{5}}.
\end{align*}
With the forcing bound on $F_{\tilde{U}}^{(1)}$ and bounds on $\Phi(X_0,s)$, we have
\begin{align*}
|(1+X^4)^\frac{1}{5}F^{(1)}_{\tilde{U}}|&\lesssim e^{-s}\big(M^4 +\log(1+X^4)\big)+e^{-\frac{5}{8}s}+\epsilon^\frac{3}{20}(1+X^4)^{-\frac{1}{5}},\\
|(1+X^4)^\frac{1}{5}F^{(1)}_{\tilde{U}}\circ\Phi(X_0,s)|&\lesssim M^4 e^{-s}+e^{-s}\log\big[1+(|X_0|+2Me^{\frac{1}{4}s_0})^4e^{5(s-s_0)}\big]+e^{-\frac{5}{8}s}+\epsilon^\frac{3}{20}\big(1+X_0^4e^{\frac{4}{5}(s-s_0)}\big)^{-\frac{1}{5}}\\
(|X_0|\leq \frac{1}{2}e^{\frac{5}{4}s_0})\qquad &\lesssim M^4 e^{-s}+4e^{-s}\log(\frac{1}{2}e^{\frac{5}{4}s_0}+2Me^{\frac{1}{4}s_0})+5(s-s_0)e^{-s}\\
&\qquad+e^{-\frac{5}{8}s}+\epsilon^\frac{3}{20}\big(1+X_0^4e^{\frac{4}{5}(s-s_0)}\big)^{-\frac{1}{5}}\\
&\lesssim M^4 e^{-s}+s_0e^{-s}\log M+5(s-s_0)e^{-s}+e^{-\frac{5}{8}s}+\epsilon^\frac{3}{20}\big(1+X_0^4e^{\frac{4}{5}(s-s_0)}\big)^{-\frac{1}{5}}\\
(s_0\geq -\log\epsilon)\qquad &\lesssim e^{-\frac{5}{8}s}+\epsilon^\frac{3}{20}\big(1+X_0^4e^{\frac{4}{5}(s-s_0)}\big)^{-\frac{1}{5}},
\end{align*}
if we choose $\epsilon$ sufficiently small to absorb the factors in $M$ and $s$. Since from the calculation in the proof of Proposition \ref{prop:weitedUtilde},
\begin{align*}
\int_{s_0}^s(1+X^4_0e^{\frac{4}{5}(s'-s_0)})^{-\frac{1}{5}}\,ds'\leq 10\log\frac{1}{l},
\end{align*}
the contribution from the forcing term is
\begin{align*}
&\lesssim \int_{s_0}^s \Big(e^{-\frac{5}{8}s'}+\epsilon^\frac{3}{20}(1+l^4e^{\frac{4}{5}(s'-s_0)}\big)^{-\frac{1}{5}}\Big) \exp\Big(\int_{s'}^s4\big(1+X_0^4e^{\frac{4}{5}(s''-s_0)}\big)^{-\frac{1}{5}}\,ds''\Big)ds'\\
&\lesssim l^{-40}\int_{s_0}^s e^{-\frac{5}{8}s'}+\epsilon^\frac{3}{20}(1+l^4e^{\frac{4}{5}(s'-s_0)}\big)^{-\frac{1}{5}}\,ds'\\
&\lesssim l^{-40}\epsilon^\frac{5}{8}+10l^{-40}\log\frac{1}{l}\epsilon^\frac{3}{20}.
\end{align*}
The initial data when $l\leq |X|\leq \frac{1}{2}\epsilon^{-\frac{5}{4}}$, by previous calculations, is
\begin{align*}
    \partial_X\widetilde{U}_{\alpha,\beta}(X,-\log\epsilon)&=\widehat{U}'_0(X)+\begin{cases}
    0,\quad&\text{if }|X|\geq 2,\\
    2\alpha X+3\beta X^2,\quad&\text{if }l\leq |X|\leq 1,\\
    \chi'(X)(\alpha X^2+\beta X^3)+\chi(X)(2\alpha X+3\beta X^2),\quad&\text{if }1\leq |X|\leq 2,
    \end{cases}\\
    |\partial_X\widetilde{U}_{\alpha,\beta}(X,-\log\epsilon)|&\leq \epsilon(1+X^4)^{-\frac{1}{5}}+\begin{cases}
    0,\quad&\text{if }|X|\geq 2,\\
    \frac{5}{2}M^{15}\epsilon |X|+\frac{18}{5}M^{25}\epsilon X^2,\quad&\text{if }l\leq |X|\leq 1,\\
    C_\chi M^{25}\epsilon |X|^3,\quad&\text{if }1\leq |X|\leq 2,
    % |\chi'(X)|(\frac{5}{4}M^{15}\epsilon X^2+\frac{6}{5}M^{25}\epsilon X^3)+\chi(X)(\frac{5}{2}M^{15}\epsilon |X|+\frac{18}{5}M^{25}\epsilon X^2)
    \end{cases}
\end{align*}
where we have used \eqref{eq:sizealphabeta}. Then for $1\leq |X_0|\leq 2$, the value along its trajectory satisfies
\begin{align*}
    |(1+X^4)^\frac{1}{5}\partial_X\widetilde{U}\circ \Phi(X_0,s)|&\leq \big(\epsilon+(1+X_0^4)^\frac{1}{5}C_\chi M^{25}\epsilon|X_0|^3\big)\exp\Big(\int_{s_0}^s 4\big(1+X_0^4e^{-\frac{4}{5}(s'-s_0)}\big)^{-\frac{1}{5}}\,ds'\Big)\\
    &\qquad +Cl^{-40}(\epsilon^\frac{5}{8}+\log\frac{1}{l}\epsilon^\frac{3}{20})\\
    &\leq \epsilon(1 + C_\chi M^{25})l^{-40}+Cl^{-40}(\epsilon^\frac{5}{8}+\log\frac{1}{l}\epsilon^\frac{3}{20})    \leq \epsilon^\frac{1}{10},
\end{align*}
where the value of $C_\chi$ changes from line to line.

The cases when $|X_0|\geq 2$ and $l\leq |X_0|\leq 1$ are easier, and we also have
\begin{align*}
    |(1+X^4)^\frac{1}{5}\partial_X\widetilde{U}\circ\Phi(X_0,s)|\leq \epsilon^\frac{1}{10}.
\end{align*}
For those $X$ on a trajectory crossing $X_0=l$ for some $s_0>-\log\epsilon$, we use the near field estimate \eqref{eq:assumptionnear0-5} to get
\begin{align*}
    |(1+X^4)^\frac{1}{5}\partial_X\widetilde{U}\circ\Phi(X_0,s)|&\leq |(1+X_0^4)^\frac{1}{5}\partial_X\widetilde{U}(X_0,s_0)|\exp\Big(\int_{s_0}^s 4\big(1+X_0^4e^{-\frac{4}{5}(s'-s_0)}\big)^{-\frac{1}{5}}\,ds'\Big)\\
&\qquad +Cl^{-40}(\epsilon^\frac{5}{8}+\log\frac{1}{l}\epsilon^\frac{3}{20})\\
&\leq 2\epsilon^\frac{1}{5}l^5(1+l^4)^\frac{1}{5}l^{-40}+Cl^{-40}(\epsilon^\frac{5}{8}+\log\frac{1}{l}\epsilon^\frac{3}{20})\leq \epsilon^\frac{1}{10}.
\end{align*}
Hence, for all cases, we have
\begin{align*}
    |\partial_X\widetilde{U}(X,s)|\leq \epsilon^\frac{1}{10}(1+X^4)^{-\frac{1}{5}}\leq \frac{1}{2}\epsilon^\frac{1}{20}(1+X^4)^{-\frac{1}{5}},
\end{align*}
so we close the bootstrap.
\end{proof}

\begin{proposition}[weighted estimate of $\partial_X^n U$ for $n=2,...,8$]
For $(X,s)$ such that $|X|\leq \frac{1}{2}e^{\frac{5}{4}s}$, we have
\begin{align*}
    |\partial_X^nU(X,s)|\leq M^{n^2}(1+X^4)^{-\frac{1}{5}},\quad\text{where }n=2,...,8.
\end{align*}
\end{proposition}
\begin{proof}
Denote $D_{n,\mathrm{w}}$ the damping for $(1+X^4)^\frac{1}{5}\partial_X^nU$, then
\begin{align*}
D_{n,\mathrm{w}}&=\frac{5}{4}n-\frac{1}{4}+\frac{n+1}{1-\dot{\tau}}\partial_XU-\frac{4X^3}{5(1+X^4)}\Big[\frac{U+e^\frac{s}{4}(\kappa-\dot{\xi})}{1-\dot{\tau}}+\frac{5}{4}X\Big]\\
&=\frac{5}{4}n-\frac{1}{4}+\frac{n+1}{1-\dot{\tau}}\partial_XU-\frac{X^4}{1+X^4}-\frac{4X^3}{5(1+X^4)}\frac{U+e^\frac{s}{4}(\kappa-\dot{\xi})}{1-\dot{\tau}}\\
&\geq \frac{5}{4}(n-1)-(1+2\epsilon^\frac{1}{5})(n+1)(1+\epsilon^\frac{1}{20})(1+X^4)^{-\frac{1}{5}}-\frac{1}{1+X^4}\\
&\qquad-\frac{4}{5}(1+2\epsilon^\frac{1}{5})(1+X^4)^{-\frac{1}{4}}\big[(1+\epsilon^\frac{3}{20})(1+X^4)^\frac{1}{20}+CM^{30\frac{5}{8}}e^{-s}   \big]\\
&\geq \frac{5}{4}(n-1)-\frac{11}{10}(n+3)(1+X^4)^{-\frac{1}{5}}.
\end{align*}
By previous calculations,
\begin{align*}
    \int_{s_0}^s-D_{n,\mathrm{w}}\circ\Phi(X_0,s')\,ds'&\leq -\frac{5}{4}(n-1)(s-s_0)+11(n+3)\log\frac{1}{l}.
\end{align*}
For $n=2$, we have
\begin{align*}
    (1+X^4)^\frac{1}{5}F_U^{(2)}\lesssim M^9 e^{-s}+M^\frac{35}{4}e^{-\frac{5}{8}s},
\end{align*}
so forcing contribution
\begin{align*}
\leq C\int_{s_0}^s(M^9 e^{-s'}+M^\frac{35}{4})e^{-\frac{5}{8}s'}l^{-55}e^{-\frac{5}{4}(s-s')}\,ds'\leq CM^9 l^{-55} e^{-s}+CM^\frac{35}{4} l^{-55}e^{-\frac{5}{8}s},
\end{align*}
where this constant $C$ does not depend on the fact that the order of derivative is 2 (it contains the $1/\pi$, $1/(1-\dot{\tau})$ and log factor absorption). Note the initial data when $l\leq |X|\leq \frac{1}{2}\epsilon^{-\frac{5}{4}}$,
\begin{align*}
    \partial_X^2U(X,-\log\epsilon)=U_2^{(2)}(X) + \widehat{U}_0^{(2)}(X)+\begin{cases}
    0,\quad&\text{if }|X|\geq 2,\\
    2\alpha+6\beta X,\quad&\text{if }l\leq |X|\leq 1,\\
    \frac{d^2}{dX^2}\big[\chi(X)(\alpha X^2+\beta X^3)\big],\quad&\text{if }1\leq |X|\leq 2,
    \end{cases}
\end{align*}
where the second derivative of the product is
\begin{align*}
    \frac{d^2}{dX^2}\big[\chi(X)(\alpha X^2+\beta X^3)\big]=\chi''(X)(\alpha X^2+\beta X^3)+2\chi'(X)(2\alpha X+3\beta X^2)+\chi(X)(2\alpha+6\beta X).
\end{align*}
The hardest case is when $1\leq |X|\leq 2$, where we have
\begin{align*}
    |\partial_X^2U(X,-\log\epsilon)|\leq C_2(1+X^4)^{-\frac{9}{20}}+ \epsilon(1+X^4)^{-\frac{1}{5}}+C_\chi M^{25}\epsilon,\\
    \implies\qquad (1+X^4)^\frac{1}{5}|\partial_X^2U(X,-\log\epsilon)|\leq C_2(1+X^4)^{-\frac{1}{4}}+\epsilon+C_\chi\epsilon.
\end{align*}
Combining forcing with the damped initial data part, we get
\begin{align*}
    |(1+X^4)^\frac{1}{5}\partial_X^2U\circ\Phi(X_0,s)|&\leq \big(C_2(1+X_0^4)^{-\frac{1}{4}}+\epsilon+C_\chi\epsilon\big)e^{-\frac{5}{4}(s-s_0)}l^{-55}+CM^9 l^{-55}\epsilon+CM^\frac{35}{4} l^{-55}\epsilon^\frac{5}{8}\\
    &\leq \frac{1}{2}M^4.
\end{align*}
For the case when $|X|\geq 2$, the initial data contribution only has $C_2(1+X_0^4)^{-\frac{1}{4}}+\epsilon$ term, and for the case when $l\leq |X|\leq 1$, the $C_\chi$ takes on a different (constant) value. For those $X=\Phi(l,s_0)$ for $s_0>-\log\epsilon$, we use the near field estimate evaluated at $X_0=l$ and obtain
\begin{align*}
    |(1+X^4)^\frac{1}{5}\partial_X^2U\circ \Phi(l,s)|&\leq (1+l^4)^\frac{1}{5}\big(2l^4\epsilon^\frac{1}{5}+|U_2''(l)|\big)e^{-\frac{5}{4}(s-s_0)}l^{-55}+CM^9l^{-55}\epsilon+CM^\frac{35}{4}l^{-55}\epsilon^\frac{5}{8}\\
    &\leq \frac{1}{2}M^4.
\end{align*}
So in all cases, we have
\begin{align*}
    |\partial_X^2U(X,s)|\leq \frac{1}{2}M^4(1+X^4)^{-\frac{1}{5}}.
\end{align*}
For $n=3,...,8$, we can do them all together
\begin{align*}
    |(1+X^4)^\frac{1}{5}F_{U}^{(n)}|&\lesssim_n M^{(n+1)^2}e^{-s}+M^{n^2}e^{-\frac{5}{8}s}+M^{\frac{35(n-1)}{4}}e^{-\frac{5}{8}s}\\
    &\qquad+(\mathbbm{1}_{n\text{ odd}}M^{\frac{(n+1)^2}{2}}+\sum_{j=2}^{\lfloor \frac{n}{2}\rfloor}M^{j^2+(n-j+1)^2})(1+X^4)^{-\frac{1}{5}}.
\end{align*}
So the forcing contribution is 
\begin{align*}
    &\int_{s_0}^s(1+X^4)^\frac{1}{5}F_U^{(n)}\circ\Phi(X_0,s')\exp\Big(\int_{s'}^s-D_{n,\mathrm{w}}\circ\Phi(X_0,s'')\,ds''\Big)ds'\\
\leq& \int_{s_0}^s(1+X^4)^\frac{1}{5}F_U^{(n)}\circ\Phi(X_0,s')e^{-\frac{5}{4}(n-1)(s-s')}l^{-11(n+3)}\,ds'\\
\lesssim&_n M^{(n+1)^2}l^{-11(n+3)}e^{-s}+\big(M^{n^2}+M^{\frac{35(n-1)}{4}}\big)l^{-11(n+3)}e^{-\frac{5}{8}s}\\
&\quad+l^{-11(n+3)}(\mathbbm{1}_{n\text{ odd}}M^{\frac{(n+1)^2}{2}}+\sum_{j=2}^{\lfloor \frac{n}{2}\rfloor}M^{j^2+(n-j+1)^2})\int_{s_0}^s\big(1+l^4e^{\frac{4}{5}(s-s_0)}\big)^{-\frac{1}{5}}e^{-\frac{5}{4}(n-1)(s-s')}\,ds'\\
\lesssim&_n M^{(n+1)^2}l^{-11(n+3)}e^{-s}+(M^{n^2}+M^{\frac{35(n-1)}{4}})l^{-11(n+3)}e^{-\frac{5}{8}s}\\
&\quad+l^{-11(n+3)}(\mathbbm{1}_{n\text{ odd}}M^{\frac{(n+1)^2}{2}}+\sum_{j=2}^{\lfloor \frac{n}{2}\rfloor}M^{j^2+(n-j+1)^2})(e^{-\frac{5}{4}(n-1)(s-s_0)}l^{-\frac{25}{4}(n-1)}+l^{-\frac{4}{5}}e^{-\frac{4}{25}s}).
\end{align*}
The initial data for $l\leq |X|\leq \frac{1}{2}\epsilon^{-\frac{5}{4}}$ is
\begin{align*}
    \partial_X^nU(X,-\log\epsilon)&=U_2^{(n)}(X)+\widehat{U}_0^{(n)}(X)+\frac{d^n}{dX^n}\big[\chi(X)(\alpha X^2+\beta X^3)]\mathbbm{1}_{|X|\leq 2},\\
    |\partial_X^nU(X,-\log\epsilon)|&\leq C_n(1+X^4)^{\frac{1}{20}-\frac{1}{4}n}+\epsilon(1+X^4)^{-\frac{1}{5}}+C_\chi M^{25}\epsilon\mathbbm{1}_{|X|\leq 2},\\
    (1+X^4)^\frac{1}{5}|\partial_X^nU(X,-\log\epsilon)|&\leq C_n(1+X^4)^{-\frac{1}{4}(n-1)}+\epsilon+C_\chi M^{25}\epsilon\mathbbm{1}_{|X|\leq 2}.
\end{align*}
So the initial data contribution for those $X=\Phi(X_0,-\log\epsilon)$ for some $l\leq |X_0|\leq \frac{1}{2}\epsilon^{-\frac{5}{4}}$ is
\begin{align*}
    (1+X_0^4)^\frac{1}{5}|\partial_X^nU(X_0,-\log\epsilon)\exp\Big(\int_{-\log\epsilon}^s-D_{n,\mathrm{w}}\circ\Phi(X_0,s')\,ds'\Big)\\
    \leq \big(C_n(1+X^4)^{-\frac{1}{4}(n-1)}+\epsilon+C_\chi M^{25}\epsilon\mathbbm{1}_{|X|\leq 2}\big)e^{-\frac{5}{4}(n-1)(s+\log\epsilon)}l^{-11(n+3)},
\end{align*}
and the initial data contribution for those $X=\Phi(l,s_0)$ for some $s_0>-\log\epsilon$ is
\begin{align*}
    &\lesssim_n (1+l^4)^\frac{1}{5}\big(l^{6-n}\epsilon^\frac{1}{5}+|U_2^{(n)}(l)|\big)e^{-\frac{5}{4}(n-1)(s-s_0)}l^{-11(n+3)}\qquad n=3,...,6,\\
    &\lesssim_n (1+l^4)^\frac{1}{5}\big(M\epsilon^\frac{1}{5}+|U_2^{(7)}(l)|\big)e^{-\frac{5}{4}(n-1)(s-s_0)}l^{-11(n+3)}\qquad n=7,\\
     &\lesssim_n (1+l^4)^\frac{1}{5}\big(M^3\epsilon^\frac{1}{5}+|U_2^{(8)}(l)|\big)e^{-\frac{5}{4}(n-1)(s-s_0)}l^{-11(n+3)}\qquad n=8.
\end{align*}
Combined with the forcing contribution, we are able to absorb the $l$ factors and get
\begin{align*}
    |\partial_X^nU(X,s)|\leq \frac{1}{2}M^{n^2}(1+X^4)^{-\frac{1}{5}}.
\end{align*}
\end{proof}

\subsection{Pointwise estimates: far field}
\begin{proposition}[temporal decay of $\partial_XU$] 
For $(X,s)$ such that $|X|\geq \frac{1}{2}e^{\frac{5}{4}s}$,
\begin{align*}
    |\partial_XU(X,s)|\leq 4e^{-s}.
\end{align*}
\end{proposition}

\begin{proof}
 Consider
\begin{align*}
\Big(\partial_s+\frac{\partial_XU}{1-\dot{\tau}}\Big)(e^s\partial_XU)+V\partial_X(e^s\partial_XU)=\frac{1}{1-\dot{\tau}}H[\partial_XU].
\end{align*}
The damping term satisfies
\begin{align*}
|D_{1,s}|=\frac{1}{1-\dot{\tau}}|\partial_XU|\leq 5e^{-s},
\end{align*}
using the bootstrap assumption. Then the damping contribution is
\begin{align*}
\exp\Big(\int_{s_0}^s-D_{1,s}\circ\Phi(X_0,s')\,ds'\Big)&\leq \exp\Big(\int_{s_0}^s5e^{-s'}\,ds'\Big)\leq e^{5e^{-s_0}}\leq e^{5\epsilon}\leq 1+3\cdot 5\epsilon\leq \frac{17}{16}.
\end{align*}
Moreover,
\begin{align*}
|F_{1,s}|=\frac{1}{1-\dot{\tau}}\big|H[\partial_XU]\big|\lesssim e^{-\frac{3}{8}s}.
\end{align*}
The initial data for $|X|\geq \frac{1}{2}\epsilon^{-\frac{5}{4}}$ is
\begin{align*}
    \partial_XU(X,-\log\epsilon)&=\widehat{U}_0'(X)+\begin{cases}
    0,\quad&\text{if }|X|\geq \epsilon^{-\frac{5}{4}},\\
    U_2'(X)\chi(2\epsilon^\frac{5}{4}X)+2\epsilon^\frac{5}{4} U_2(X)\chi'(2\epsilon^\frac{5}{4}X),\quad&\text{if }\frac{1}{2}\epsilon^{-\frac{5}{4}}\leq |X|\leq \epsilon^{-\frac{5}{4}},
    \end{cases}\\
   |\partial_XU(X,-\log\epsilon)|&\leq\frac{1}{2} \epsilon+C_\chi\epsilon\mathbbm{1}_{\frac{1}{2}\epsilon^{-\frac{5}{4}}\leq |X|\leq \epsilon^{-\frac{5}{4}}}\leq \frac{7}{2}\epsilon,
\end{align*}
where, by choosing $\chi$ carefully (almost decay linearly, i.e.\ $\chi'\approx 1$), we can make $C_\chi=3$. 

If $X=\Phi(\pm \frac{1}{2}e^{\frac{5}{4}s_0},s_0)$ for some $s_0>-\log\epsilon$, then we use the middle field estimate and Lemma \ref{lem:U2} to get
\begin{align*}
    |\partial_XU(\pm\frac{1}{2}e^{\frac{5}{4}s_0},s_0)|\leq (\epsilon^\frac{1}{20}+\frac{5}{24})2^\frac{4}{5}e^{-s_0}\leq \frac{7}{2}e^{-s_0}.
\end{align*}
Hence,
\begin{align*}
|e^s\partial_XU\circ\Phi(X_0,s)|&\leq \frac{7}{2}\cdot\frac{17}{16}+C\int_{s_0}^se^{-\frac{3}{8}s'}\,ds'\leq 3\frac{21}{32}+C\epsilon^\frac{3}{8}\leq 3\frac{3}{4},\\
\implies\qquad|\partial_XU(X,s)|&\leq 3\frac{3}{4}e^{-s}.
\end{align*}
\end{proof}

\begin{proposition}[Temporal decay of $\partial_X^nU$, $n=2,...,8$]
For $(X,s)$ such that $|X|\geq \frac{1}{2}e^{\frac{5}{4}s}$,
\begin{align*}
|\partial_X^nU(X,s)|\leq 2M^{n^2}e^{-s},\qquad n=2,...,8.
\end{align*}
\end{proposition}
\begin{proof}
Consider the equation
\begin{align*}
\Big(\partial_s+\frac{5}{4}(n-1)+\frac{n+1}{1-\dot{\tau}}\partial_XU\Big)(e^s\partial_X^nU)+V\partial_X^{n+1}(e^s\partial_X^nU)=e^sF_U^{(n)}.
\end{align*}
The damping is
\begin{align*}
D_{n,s}=\frac{5}{4}(n-1)+\frac{n+1}{1-\dot{\tau}}\partial_XU\geq \frac{5}{4}(n-1)-4(n+1)(1+2\epsilon^\frac{1}{5})\epsilon\geq n-1.
\end{align*}
For forcing, we have
\begin{align*}
F_{n,s}&\lesssim_n M^{n^2}e^{-\frac{3}{8}s}+\sum_{j=2}^{\lfloor \frac{n+1}{2}\rfloor}M^{j^2+(n-j+1)^2}e^{-s},\\
\int_{s_0}^sF_{n,s}\circ\Phi(X_0,s')e^{-(n-1)(s-s')}\,ds'&\lesssim_n (M^{n^2}+\sum_{j=2}^{\lfloor\frac{n+1}{2}\rfloor}M^{j^2+(n-j+1)^2})e^{-\frac{3}{8}s_0}.
\end{align*}
The initial data is
\begin{align*}
    \partial_X^nU(X,-\log\epsilon)&=\widehat{U}_0^{(n)}(X)+\frac{d^n}{dX^n}\big[\chi(2\epsilon^\frac{5}{4}X)U_2(X)\big],\\
    |\partial_X^nU(X,-\log\epsilon)|&\leq \frac{1}{2}M^{n^2}\epsilon+C_\chi\epsilon.
\end{align*}
In the other case, we have
\begin{align*}
    |\partial_X^nU(\pm\frac{1}{2}e^{\frac{5}{4}s_0},s_0)|\leq M^{n^2} 2^\frac{4}{5}e^{-s_0}.
\end{align*}
Hence,
\begin{align*}
    |e^s\partial_X^nU\circ\Phi(X_0,s)|&\leq \max\{\frac{1}{2}M^{n^2}+C_\chi,2^\frac{4}{5}M^{n^2}\}+C_n(M^{n^2}+\sum_{j=2}^{\lfloor\frac{n+1}{2}\rfloor}M^{j^2+(n-j+1)^2})e^{-\frac{3}{8}s_0}\\
    &\leq \frac{15}{8}M^{n^2},\\
    \implies\qquad|\partial_X^nU(X,s)|&\leq \frac{15}{8}M^{n^2}e^{-s}.
\end{align*}
\end{proof}

\subsection{Modulation variables}\label{sec:modulation}
We first note that 
\begin{align*}
    |\kappa(t)|=|u\big(\xi(t)\big)|\leq M,
\end{align*}
so as a consequence \eqref{eq:kappa-xi},
\begin{align*}
|\dot{\xi}|\leq \epsilon^\frac{1}{5}e^{-s}+M\leq \frac{3}{2}M <2M.
\end{align*}

To prove the bound for $|\kappa-\dot{\xi}|$, we use \eqref{eq:dottaukappa-xi}. We have
\begin{align*}
\partial_X^5U(0,s)\geq \partial_X^5U_2(0)-|\partial_X^5\widetilde{U}(0,s)|\geq 120-\epsilon^\frac{1}{2}\geq 119,
\end{align*}
and similarly $\partial_X^5U(0,s)\leq 121$. Hence, by \eqref{eq:kappa-xi},
\begin{align*}
    |\kappa-\dot{\xi}|&\lesssim e^{-\frac{1}{4}s}(M^{\frac{7\cdot 35}{8}}e^{-s}+\epsilon^\frac{1}{10}e^{-\frac{3}{4}s}M^{27} e^{-s})\leq\frac{3}{4}\epsilon^\frac{1}{5}e^{-s}<\epsilon^\frac{1}{5}e^{-s}.
\end{align*}
To prove the bound for $|\dot{\tau}|$, we use \eqref{eq:dottaukappa-xi} to get
\begin{align*}
    |\dot{\tau}|\leq CM^\frac{35}{8}e^{-s}+e^{\frac{1}{4}s}\epsilon^\frac{1}{5}e^{-s}\epsilon^\frac{1}{10}e^{-\frac{3}{4}s}\leq \frac{3}{4}\epsilon^\frac{1}{5}e^{-\frac{3}{4}s}<\epsilon^\frac{1}{5}e^{-\frac{3}{4}s}.
\end{align*}
To prove the bound for $|\dot{\kappa}|$, we use \eqref{eq:eqkappadot} to get
\begin{align*}
    \dot{\kappa}&=e^s(\kappa-\dot{\xi})+e^{-\frac{1}{4}s}H[U+e^{\frac{1}{4}s}\kappa](0,s),\\
    |\dot{\kappa}|&\leq \epsilon^\frac{1}{5}+CMe^{\frac{1}{8}s}\leq \frac{1}{2}M^2e^{\frac{1}{8}s}.
\end{align*}

Finally, we prove the assumption on $T_*$. By the
fundamental theorem of calculus, the assumption that $|T_*|\leq 2\epsilon^{\frac{39}{20}}$, and $\tau(-\epsilon)=0$, we
get 
\begin{align*}
    |\tau(t)|\leq |\tau(-\epsilon)|+\int_{-\epsilon}^{T_*}|\dot{\tau}(t')|\,dt'\leq (2\epsilon^{1+\frac{19}{20}}+\epsilon)\epsilon^{\frac{1}{5}+\frac{3}{4}}\leq \frac{3}{2}\epsilon^{\frac{39}{20}}.
\end{align*}
Now consider $h(t):=t-\tau(t)$. We have
\begin{gather*}
    h(-\epsilon)=-\epsilon,\qquad h(T_*)=0, \qquad \dot{h}=1-\dot{\tau}>0.
\end{gather*}
Hence, we must have $h(t)<0$, which implies $t<\tau(t)$ for all $t\in[-\epsilon,T_*)$.

Note that $\tau(-\epsilon)=0$, $\tau(T_*)=T_*$ is equivalent to
\begin{align*}
    \int_{-\epsilon}^{T_*}\big(1-\dot{\tau}(t)\,dt=\epsilon.
\end{align*}
Since $|\dot{\tau}|\leq \epsilon^\frac{19}{20}$, we have
\begin{align*}
    (1-\epsilon^\frac{19}{20})(T_*+\epsilon)&=\int_{-\epsilon}^{T_*}(1-\epsilon^\frac{19}{20})\,dt\leq \epsilon\\
    T_*&\leq\frac{\epsilon^{\frac{39}{20}}}{1-\epsilon^\frac{19}{20}}\leq \frac{3}{2}\epsilon^{\frac{39}{20}}.
\end{align*}

\section{Closure of bootstrap with first order parameter derivatives}
In this section, all the bounds are uniform in the parameter ball $B$ defined as \eqref{eq:sizealphabeta}.
\subsection{$L^2$ estimates}
\begin{proposition}[$\partial_\alpha U$]
We have \begin{align*}
    \|\partial_\alpha U(\cdot,s)+e^{\frac{1}{4}s}\partial_\alpha\kappa\|_{L^2}\leq M^{15}\epsilon^\frac{3}{4}e^s,
\end{align*}
and the same result holds if we replace $\alpha$ by $\beta$.
\end{proposition}
\begin{proof}
Consider the equation
\begin{align}
\partial_s(e^{-\frac{1}{4}s}\partial_\alpha U+\partial_\alpha\kappa)+e^{-\frac{1}{4}s}V\partial_\alpha\partial_XU+e^{-\frac{1}{4}s}\partial_\alpha V\partial_XU\\
=\frac{e^{-s}}{1-\dot{\tau}}H[e^{-\frac{1}{4}s}\partial_\alpha U+\partial_\alpha\kappa]+\frac{e^{-s}}{(1-\dot{\tau})^2}\partial_\alpha\dot{\tau}H[e^{-\frac{1}{4}s}U+\kappa],\label{eq:eqpartialalphaU}
\end{align}
where 
\begin{align*}
\partial_\alpha V=\frac{1}{1-\dot{\tau}}(\partial_\alpha U+e^{\frac{1}{4}s}\partial_\alpha\kappa)-\frac{e^{\frac{1}{4}s}}{1-\dot{\tau}}\partial_\alpha\dot{\xi}+\frac{\partial_\alpha\dot{\tau}}{(1-\dot{\tau})^2}(U+e^{\frac{1}{4}s}\kappa-e^{\frac{1}{4}s}\dot{\xi}).
\end{align*}
We take $L^2$ inner product with $e^{-\frac{1}{4}s}\partial_\alpha U+\partial_\alpha\kappa$. Note that
\begin{align*}
&\int_\mathbb{R}e^{-\frac{1}{4}s}V\partial_\alpha\partial_XU(e^{-\frac{1}{4}s}\partial_\alpha U+\partial_\alpha\kappa)\,dX\\
=&\frac{1}{1-\dot{\tau}}\int_\mathbb{R}(e^{-\frac{1}{4}s}U+\kappa-\dot{\xi})\partial_\alpha\partial_XU(e^{-\frac{1}{4}s}\partial_\alpha U+\partial_\alpha\kappa)\,dX+\frac{5}{4}\int_\mathbb{R}Xe^{-\frac{1}{4}s}\partial_\alpha\partial_XU(e^{-\frac{1}{4}s}\partial_\alpha U+\partial_\alpha\kappa)\,dX\\
=&I_1+I_2.
\end{align*}
We will move $I_1$ to the right hand side. Using the $L^\infty$-$L^2$-$L^2$ H\"older's inequality, \eqref{eq:assumptionULinfty} and \eqref{eq:assumptionmodulation}, we get
\begin{align*}
|I_1|&\leq (1+2\epsilon^\frac{1}{5}e^{-\frac{3}{4}s})(M+M+\epsilon^\frac{1}{5}e^{-s})\|\partial_\alpha\partial_XU\|_{L^2}\|e^{-\frac{1}{4}s}\partial_\alpha U+\partial_\alpha\kappa\|_{L^2}\\
&\lesssim M^{14}\epsilon^\frac{3}{4}e^{\frac{3}{4}s}\|e^{-\frac{1}{4}s}\partial_\alpha U+\partial_\alpha\kappa\|_{L^2}.
\end{align*}
For $I_2$, although we will not prove any spatial decay for $e^{-\frac{1}{4}s}\partial_\alpha U+\partial_\alpha\kappa$, since this term is a priori $L^2$-integrable, combined with $\|\partial_\alpha\partial_XU\|$ (which is also $L^2$-integrable), the decay must beat $X$ as $X\to\infty$, so we can integrate by parts to get
\begin{align*}
I_2=-\frac{5}{8}\|e^{-\frac{1}{4}s}\partial_\alpha U+\partial_\alpha\kappa\|_{L^2}^2.
\end{align*}
Then we calculate
\begin{align*}
\int_\mathbb{R}e^{-\frac{1}{4}s}\partial_\alpha V\partial_XU(e^{-\frac{1}{4}s}\partial_\alpha U+\partial_\alpha\kappa)\,dX=&\frac{1}{1-\dot{\tau}}\int_\mathbb{R}\partial_XU(e^{-\frac{1}{4}s}\partial_\alpha U+\partial_\alpha\kappa)^2\,dX\\
&-\frac{\partial_\alpha\dot{\xi}}{1-\dot{\tau}}\int_\mathbb{R}\partial_XU(e^{-\frac{1}{4}s}\partial_\alpha U+\partial_\alpha\kappa)\,dX\\
&+\frac{\partial_\alpha\dot{\tau}}{(1-\dot{\tau})^2}\int_\mathbb{R}(e^{-\frac{1}{4}s}U+\kappa-\dot{\xi})\partial_XU(e^{-\frac{1}{4}s}\partial_\alpha U+\partial_\alpha\kappa)\,dX\\
=&\mathit{II}_1+\mathit{II}_2+\mathit{II}_3.
\end{align*}
We will move all three terms to the right hand side. We have
\begin{align*}
|\mathit{II}_1|&\leq (1+2\epsilon^\frac{1}{5}e^{-\frac{3}{4}s})\|e^{-\frac{1}{4}s}\partial_\alpha U+\partial_\alpha\kappa\|_{L^\infty}\|\partial_XU\|_{L^2}\|e^{-\frac{1}{4}s}\partial_\alpha U+\partial_\alpha\kappa\|_{L^2}\\
&\lesssim M^4\epsilon^\frac{3}{4}e^{\frac{1}{2}s}\|e^{-\frac{1}{4}s}\partial_\alpha U+\partial_\alpha\kappa\|_{L^2},\\
|\mathit{II}_2|&\leq M\epsilon^\frac{1}{2}(1+2\epsilon^\frac{1}{5}e^{-\frac{3}{4}s})\|\partial_XU\|_{L^2}\|e^{-\frac{1}{4}s}\partial_\alpha U+\partial_\alpha\kappa\|_{L^2}\\
&\lesssim M\epsilon^\frac{1}{2}\|e^{-\frac{1}{4}s}\partial_\alpha U+\partial_\alpha\kappa\|_{L^2},\\
|\mathit{II}_3|&\leq \epsilon^\frac{1}{2}(1+2\epsilon^\frac{1}{5}e^{-\frac{3}{4}s})^2(M+M+\epsilon^\frac{1}{5}e^{-s})\|\partial_XU\|_{L^2}\|e^{-\frac{1}{4}s}\partial_\alpha U+\partial_\alpha\kappa\|_{L^2}\\
&\lesssim M\epsilon^\frac{1}{2}\|e^{-\frac{1}{4}s}\partial_\alpha U+\partial_\alpha\kappa\|_{L^2}.
\end{align*}
For the inner products with the two Hilbert transform terms on the right hand side of \eqref{eq:eqpartialalphaU}, the first inner product vanishes, and we bound the second by
\begin{align*}
&\Big|\frac{e^{-s}\partial_\alpha\dot{\tau}}{(1-\dot{\tau})^2}\int_\mathbb{R} H[e^{-\frac{1}{4}s}U+\kappa](e^{-\frac{1}{4}s}\partial_\alpha U+\partial_\alpha\kappa)\,dX\Big|\\
\leq&e^{-s}(1+2\epsilon^\frac{1}{5}e^{-\frac{3}{4}s})^2\epsilon^\frac{1}{2}\|e^{-\frac{1}{4}s}U+\kappa\|_{L^2}\|e^{-\frac{1}{4}s}\partial_\alpha U+\partial_\alpha\kappa\|_{L^2}\\
\lesssim & M\epsilon^\frac{1}{2}e^{-\frac{3}{8}s}\|e^{-\frac{1}{4}s}\partial_\alpha U+\partial_\alpha\kappa\|_{L^2},
\end{align*}
where we use the assumptions of $\dot{\tau}$ and $\partial_\alpha\dot{\tau}$ in \eqref{eq:assumptionmodulation} and \eqref{eq:assumptionparamodulation}.\\
Collecting all the terms, we have
\begin{align*}
&\frac{1}{2}\frac{d}{ds}\|e^{-\frac{1}{4}s}\partial_\alpha U+\partial_\alpha\kappa\|_{L^2}^2-\frac{5}{8}\|e^{-\frac{1}{4}s}\partial_\alpha U+\partial_\alpha\kappa\|_{L^2}^2
% &\leq (3M^{14}\epsilon^\frac{3}{4}e^{\frac{3}{4}s}+4M^4\epsilon^\frac{3}{4}e^{\frac{1}{2}s}+10M\epsilon^\frac{1}{2}+3M\epsilon^\frac{1}{2}e^{-\frac{3}{8}s})\|e^{-\frac{1}{4}s}\partial_\alpha U+\partial_\alpha\kappa\|_{L^2}\\
\lesssim M^{14}\epsilon^\frac{3}{4}e^{\frac{3}{4}s}\|e^{-\frac{1}{4}s}\partial_\alpha U+\partial_\alpha\kappa\|_{L^2}.
\end{align*}
For the initial data, We fix $\kappa_0$, and make $\widehat{U}_0(X)=-\kappa_0$ for $|X|\geq \epsilon^{-\frac{5}{4}}$. So $\partial_\alpha$ at $s=-\log\epsilon$ is only $X^2\chi(X)$, so the $L^2$-norm is $O(1)$. So
\begin{align*}
    \|\epsilon^\frac{1}{4}\partial_\alpha U(\cdot,-\log\epsilon)\|_{L^2} \leq \epsilon^\frac{1}{4}\|\chi(X)X^2\|_{L^2}\lesssim \epsilon^\frac{1}{4}.
\end{align*}
Hence,
\begin{align*}
\|e^{-\frac{1}{4}s}\partial_\alpha U+\partial_\alpha\kappa\|_{L^2}&\leq C\epsilon^\frac{1}{4}e^{\frac{5}{8}(s+\log\epsilon)}+CM^{14}\epsilon^\frac{3}{4}\int_{-\log\epsilon}^se^{\frac{5}{8}(s-s')+\frac{3}{4}s'}\,ds'\\
&\leq CM^{14}\epsilon^\frac{3}{4}e^{\frac{3}{4}s}\leq\frac{1}{2}M^{15}\epsilon^\frac{3}{4}e^{\frac{3}{4}s}.
\end{align*}
\end{proof}

\begin{proposition}[$\partial_\alpha\partial_X U$]
We have 
\begin{align*}
    \|\partial_\alpha\partial_XU(\cdot,s)\|_{L^2}\leq M^{13}\epsilon^\frac{3}{4}e^{\frac{3}{4}s},
\end{align*}
and the same result holds if we replace $\alpha$ by $\beta$.
\end{proposition}
\begin{proof}
We take $L^2$-inner product of \eqref{eq:eqnxUalpha} with $n=1$ with $\partial_\alpha\partial_XU$. Similar as before, integration by parts yields
\begin{align*}
\int_\mathbb{R}V\partial_\alpha\partial_X^2U\partial_\alpha\partial_XU\,dX=-\frac{1}{2(1-\dot{\tau})}\int_\mathbb{R}\partial_XU(\partial_\alpha\partial_XU)^2\,dX-\frac{5}{8}\|\partial_\alpha\partial_XU\|_{L^2}^2,
\end{align*}
and the forcing part is
\begin{align*}
\int_\mathbb{R}F^{(1)}_{U,\alpha}\partial_\alpha\partial_XU\,dX=&\frac{e^{-s}}{(1-\dot{\tau})^2}\partial_\alpha\dot{\tau}\int_\mathbb{R}H[\partial_XU]\partial_\alpha\partial_XU\,dX\\
&-\frac{1}{1-\dot{\tau}}\int_\mathbb{R}\partial_X^2U\partial_\alpha\partial_XU(\partial_\alpha U+e^{\frac{1}{4}s}\partial_\alpha\kappa-e^{\frac{1}{4}s}\partial_\alpha\dot{\xi})\,dX\\
&-\frac{\partial_\alpha\dot{\tau}}{(1-\dot{\tau})^2}\int_\mathbb{R}(\partial_XU)^2\partial_\alpha\partial_XU\,dX\\
&-\frac{\partial_\alpha\dot{\tau}}{(1-\dot{\tau})^2}\int_\mathbb{R}(U+e^{\frac{1}{4}s}\kappa-e^{\frac{1}{4}s}\dot{\xi})\partial_X^2U\partial_\alpha\partial_XU\,dX\\
\lesssim &\big(e^{-s}\epsilon^\frac{1}{2}+M^{\frac{35}{4}}(M^4 \epsilon^\frac{3}{4}e^{\frac{3}{4}s}+M\epsilon^\frac{1}{2}e^{-\frac{1}{4}s})+\epsilon^\frac{1}{2}+\epsilon^\frac{1}{2}M^{\frac{35}{4}+1}e^{\frac{1}{4}s}\big)\|\partial_\alpha\partial_XU\|_{L^2}\\
\lesssim& M^{12\frac{3}{4}}\epsilon^\frac{3}{4}e^{\frac{3}{4}s}\|\partial_\alpha\partial_XU\|_{L^2}. 
\end{align*}
Combining all terms we have
\begin{align*}
\frac{1}{2}\frac{d}{ds}\|\partial_\alpha\partial_XU\|_{L^2}^2+\frac{3}{8}\|\partial_\alpha\partial_XU\|_{L^2}^2+\frac{3}{2(1-\dot{\tau})}\int_\mathbb{R}\partial_XU(\partial_\alpha\partial_XU)^2\,dX\lesssim M^{12\frac{3}{4}}\epsilon^\frac{3}{4}e^{\frac{3}{4}s}\|\partial_\alpha\partial_XU\|_{L^2}. 
\end{align*}
We will move the last term on the left hand side to the right. We split the integral
\begin{align*}
\Big|\frac{3}{2(1-\dot{\tau})}\int_\mathbb{R}\partial_XU(\partial_\alpha\partial_XU)^2\,dX\Big|&\lesssim \Big|\int_{|X|\leq \frac{1}{2}e^{\frac{5}{4}s}}+\int_{|X|>\frac{1}{2}e^{\frac{5}{4}s}}\partial_XU(\partial_\alpha\partial_XU)^2\,dX\Big|\\
&\lesssim \Big[M^9\epsilon^\frac{3}{4}e^{\frac{3}{4}s}\Big(\int_{|X|\leq\frac{1}{2}e^{\frac{5}{4}s}}(1+X^4)^{-\frac{2}{5}}\Big)^\frac{1}{2}+M^9\epsilon^\frac{3}{4}e^{-\frac{1}{4}s}\Big]\|\partial_\alpha\partial_XU\|_{L^2}\\
&\lesssim M^9\epsilon^\frac{3}{4}e^{\frac{3}{4}s}\|\partial_\alpha\partial_XU\|_{L^2}.
\end{align*}
Putting all the terms together, note that $\|\partial_\alpha\partial_XU(\cdot,-\log\epsilon)\|_{L^2}\sim O(1)$,
\begin{align*}
\frac{d}{ds}\|\partial_\alpha\partial_XU\|_{L^2}+\frac{3}{8}\|\partial_\alpha\partial_XU\|_{L^2}&\leq CM^{12\frac{3}{4}}\epsilon^\frac{3}{4}e^{\frac{3}{4}s}\\
\|\partial_\alpha\partial_XU(\cdot,s)\|_{L^2}&\leq \|\partial_\alpha\partial_X(\cdot,-\log\epsilon)\|_{L^2}e^{-\frac{3}{8}(s+\log\epsilon)}+CM^{12\frac{3}{4}}\epsilon^\frac{3}{4}\int_{-\log\epsilon}^se^{-\frac{3}{8}(s-s')}e^{\frac{3}{4}s'}\,ds'\\
&\leq CM^{12\frac{3}{4}}\epsilon^\frac{3}{4}e^{\frac{3}{4}s}\leq \frac{3}{4}M^{13}\epsilon^\frac{3}{4}e^{\frac{3}{4}s}.
\end{align*}
\end{proof}

\begin{proposition}[$\partial_\alpha\partial_X^8U$]
We have
\begin{align*}
    \|\partial_\alpha\partial_X^8U(\cdot,s)\|_{L^2}\leq M^{90}\epsilon^\frac{3}{4}e^{\frac{3}{4}s},
\end{align*}
and the same result holds if we replace $\alpha$ by $\beta$.
\end{proposition}
\begin{proof}
Consider equation \eqref{eq:eqnxUalpha} with $n=8$.
% \begin{align*}
% \big(\partial_s+\frac{39}{4}+\frac{9}{1-\dot{\tau}}\partial_XU\big)\partial_\alpha\partial_X^8U+V\partial_\alpha\partial_X^9U=\frac{e^{-s}}{1-\dot{\tau}}H[\partial_\alpha\partial_X^8U]+\frac{e^{-s}}{(1-\dot{\tau})^2}\partial_\alpha\dot{\tau}H[\partial_X^8U]\\
% -\frac{1}{1-\dot{\tau}}\Big\{(\partial_\alpha U+e^{\frac{1}{4}s}\partial_\alpha\kappa-e^{\frac{1}{4}s}\partial_\alpha\dot{\xi})\partial_X^9U+\sum_{j=1}^7{9\choose j}\partial_\alpha\partial_X^jU\partial_X^{9-j}U\Big\}\\
% -\frac{\partial_\alpha\dot{\tau}}{(1-\dot{\tau})^2}\Big\{\partial_X^9U(U+e^{\frac{1}{4}s}\kappa-e^{\frac{1}{4}s}\dot{\xi})+9\partial_XU\partial_X^8U+\sum_{j=1}^4{9\choose j}\partial_X^jU\partial_X^{9-j}U\Big\}.
% \end{align*}
Take $L^2$ inner product with $\partial_\alpha\partial_X^8U$, integrate by part on left hand side, we get
\begin{align*}
\frac{1}{2}\frac{d}{ds}\|\partial_\alpha\partial_X^8U\|_{L^2}^2+\frac{73}{8}\|\partial_\alpha\partial_X^8U\|_{L^2}^2+\frac{17}{2(1-\dot{\tau})}\int_\mathbb{R}\partial_XU(\partial_\alpha\partial_X^8U)^2\,dX=\text{right hand side},
\end{align*}
where right hand side can be bounded by H\"older's inequality,
\begin{align*}
\text{RHS}&\lesssim e^{-s}\epsilon^\frac{1}{2}M^{61\frac{1}{4}}\|\partial_\alpha\partial_X^8U\|_{L^2}+(M^4 \epsilon^\frac{3}{4}e^{\frac{3}{4}s}+M\epsilon^\frac{1}{2}e^{-\frac{1}{4}s})M^{70}\|\partial_\alpha\partial_X^8U\|_{L^2}\\
&\quad+\|\partial_\alpha\partial_X^8U\|_{L^2}\sum_{j=1}^7{9\choose j}M^{(j+2)^2}\epsilon^\frac{3}{4}e^{\frac{3}{4}s}M^{70(1-\frac{j}{8})}+\epsilon^\frac{1}{2}M^{70}(2Me^{\frac{1}{4}s}+\epsilon^\frac{1}{5}e^{-\frac{3}{4}s})\|\partial_\alpha\partial_X^8U\|_{L^2}\\
&\quad+9\epsilon^\frac{1}{2}M^{61\frac{1}{4}}\|\partial_\alpha\partial_X^8U\|_{L^2}+\epsilon^\frac{1}{2}\|\partial_\alpha\partial_X^8U\|_{L^2}\sum_{j=2}^4{9\choose j}M^{j^2+70(1-\frac{j}{8})}\\
&\lesssim M^{89\frac{3}{4}}\epsilon^\frac{3}{4}e^{\frac{3}{4}s}\|\partial_\alpha\partial_X^8U\|_{L^2}.
\end{align*}
Note that the two terms on left hand side
\begin{align*}
\frac{73}{8}\|\partial_\alpha\partial_X^8U\|_{L^2}^2+\frac{17}{2(1-\dot{\tau})}\int_\mathbb{R}\partial_XU(\partial_\alpha\partial_X^8U)^2\,dX&\geq \big(\frac{73}{8}-\frac{17}{2}(1+2\epsilon^\frac{1}{5})(1+\epsilon^\frac{17}{20})\big)\|\partial_\alpha\partial_X^8U\|_{L^2}^2\\
&\geq \frac{1}{2}\|\partial_\alpha\partial_X^8U\|_{L^2}^2.
\end{align*}
So
\begin{align*}
\frac{1}{2}\frac{d}{ds}\|\partial_\alpha\partial_X^8U\|_{L^2}^2+\frac{1}{2}\|\partial_\alpha\partial_X^8U\|_{L^2}^2&\lesssim M^{89\frac{3}{4}}\epsilon^\frac{3}{4}e^{\frac{3}{4}s}\|\partial_\alpha\partial_X^8U\|_{L^2}\\
\frac{d}{ds}\|\partial_\alpha\partial_X^8U\|_{L^2}+\frac{1}{2}\|\partial_\alpha\partial_X^8U\|_{L^2}&\lesssim  M^{89\frac{3}{4}}\epsilon^\frac{3}{4}e^{\frac{3}{4}s}\\
\|\partial_\alpha\partial_X^8U\|_{L^2}&\leq \|\partial_\alpha\partial_X^8U(\cdot,-\log\epsilon)\|_{L^2}e^{-\frac{1}{2}(s+\log\epsilon)}\\
&\qquad+CM^{89\frac{3}{4}}\int_{-\log\epsilon}^s\epsilon^\frac{3}{4}e^{\frac{3}{4}s'}e^{-\frac{1}{2}(s-s')}\,ds'\\
&\leq CM^{89\frac{3}{4}}\epsilon^\frac{3}{4}e^{\frac{3}{4}s}\leq \frac{3}{4}M^{90}\epsilon^\frac{3}{4}e^{\frac{3}{4}s}.
\end{align*}
\end{proof}

\subsection{$L^\infty$ estimates of $\partial_\alpha U$}
\begin{proposition}
We have
\begin{align*}
    |\partial_\alpha U(X,s)+ e^{\frac{1}{4}s}\partial_\alpha\kappa| \leq\begin{cases}
    M^4\epsilon^\frac{3}{4}e^{\frac{3}{4}s},\quad&\text{if }|X|\leq \frac{1}{2}e^{\frac{5}{4}s},\\
    2M^4\epsilon^\frac{3}{4}e^{\frac{3}{4}s},\quad&\text{if }|X|\geq \frac{1}{2}e^{\frac{5}{4}s}.
    \end{cases} 
\end{align*}
\end{proposition}
\begin{proof}
For $|X|\leq l$, we use the near field bound \eqref{eq:assumptionparanear0-6} to get
\begin{align*}
    |\partial_\alpha U(X,s)+e^{\frac{1}{4}s}\partial_\alpha\kappa|\leq Ml^\frac{1}{2}\epsilon^\frac{3}{4}e^{\frac{3}{4}s}+e^{\frac{1}{4}s}\epsilon^\frac{1}{2}\leq \frac{1}{2}M^4\epsilon^\frac{3}{4}e^{\frac{3}{4}s}.
\end{align*}
For $|X|\geq l$, we consider the following equation for $e^{-\frac{1}{4}s}\partial_\alpha U+\partial_\alpha\kappa$
\begin{align*}
    \big(\partial_s+\frac{\partial_XU}{1-\dot{\tau}}\big)\big(e^{-\frac{1}{4}s}\partial_\alpha U+\partial_\alpha\kappa\big)+V\partial_X\big(e^{-\frac{1}{4}s}\partial_\alpha U+\partial_\alpha\kappa)=F_{U,\alpha,s},
\end{align*}
where 
\begin{align*}
    F_{U,\alpha,s}&=-\frac{e^{-s}\partial_\alpha\dot{\tau}}{(1-\dot{\tau})^2}\dot{\kappa}+\frac{e^{-\frac{5}{4}s}}{1-\dot{\tau}}H[\partial_\alpha U+e^{\frac{1}{4}s}\partial_\alpha\kappa]+\frac{e^{-\frac{5}{4}s}\partial_\alpha\dot{\tau}}{(1-\dot{\tau})^2}H[U+e^{\frac{1}{4}s}\kappa]+\frac{\partial_XU}{1-\dot{\tau}}\partial_\alpha\dot{\xi}\\
    &\qquad-\frac{\partial_\alpha\dot{\tau}}{(1-\dot{\tau})^2}\partial_XU(e^{-\frac{1}{4}s}U+\kappa-\dot{\xi}).
\end{align*}
The damping term can be bounded by
\begin{align*}
    \Big|\frac{\partial_XU}{1-\dot{\tau}}\Big|&\leq \begin{cases}
    (1+X^4)^{-\frac{1}{5}},\quad&\text{if }l\leq |X|\leq \frac{1}{2}e^{\frac{5}{4}s},\\
    5e^{-s},\quad&\text{if }|X|\geq \frac{1}{2}e^{\frac{5}{4}s}.
    \end{cases}\\
    \implies\qquad \Big|\int_{s_0}^s \frac{\partial_XU}{1-\dot{\tau}}\circ\Phi(X_0,s')\,ds'\Big|&\leq \begin{cases}
    10\log\frac{1}{l},\quad&\text{if }l\leq |X|\leq\frac{1}{2}e^{\frac{5}{4}s},\\
    5e^{-s_0},\quad&\text{if }|X|\geq \frac{1}{2}e^{\frac{5}{4}s}.
    \end{cases}
\end{align*}
The last term on the first line can be bounded by
\begin{align*}
    \Big|\frac{\partial_XU}{1-\dot{\tau}}\partial_\alpha\dot{\xi}\Big|\lesssim \begin{cases}
    M\epsilon^\frac{1}{2}(1+X^4)^{-\frac{1}{5}},\quad&\text{if }l\leq |X|\leq \frac{1}{2}e^{\frac{5}{4}s},\\
    M\epsilon^\frac{1}{2}e^{-s},\quad&\text{if }|X|\geq \frac{1}{2}e^{\frac{5}{4}s},
    \end{cases}
\end{align*}
and the term on the second line can be bounded by
\begin{align*}
    \Big|\frac{\partial_\alpha\dot{\tau}}{(1-\dot{\tau})^2}\partial_XU(e^{-\frac{1}{4}s}U+\kappa-\dot{\xi})\Big|&\lesssim\begin{cases}
    \epsilon^\frac{1}{2}(1+X^4)^{-\frac{1}{5}}\big(e^{-\frac{1}{4}s}(1+X^4)^\frac{1}{20}+\epsilon^\frac{1}{5}e^{-s}\big),\quad&\text{if }l\leq |X|\leq \frac{1}{2}e^{\frac{5}{4}s},\\
    M\epsilon^\frac{1}{2}e^{-s},\quad&\text{if }|X|\geq \frac{1}{2}e^{\frac{5}{4}s},
    \end{cases}\\
    &\lesssim\begin{cases}
    \epsilon^\frac{1}{2}e^{-\frac{1}{4}s}(1+X^4)^{-\frac{3}{20}},\quad&\text{if }l\leq |X|\leq \frac{1}{2}e^{\frac{5}{4}s},\\
    M\epsilon^\frac{1}{2}e^{-s},\quad&\text{if }|X|\geq \frac{1}{2}e^{\frac{5}{4}s}.
    \end{cases} 
\end{align*}
So we can bound $F_{U,\alpha,s}$ by
\begin{align*}
    |F_{U,\alpha,s}|\lesssim M^4\epsilon^\frac{3}{4}e^{-\frac{3}{8}s}+\begin{cases}
    M\epsilon^\frac{1}{2}(1+X^4)^{-\frac{1}{5}}+\epsilon^\frac{1}{2}e^{-\frac{1}{4}s}(1+X^4)^{-\frac{3}{20}},\quad&\text{if }l\leq |X|\leq \frac{1}{2}e^{\frac{5}{4}s},\\
    M\epsilon^\frac{1}{2}e^{-s},\quad&\text{if }|X|\geq \frac{1}{2}e^{\frac{5}{4}s}.
    \end{cases} 
\end{align*}
We first consider the case when $|X|\leq \frac{1}{2}e^{\frac{5}{4}s}$. If $X_0\geq l$, then
\begin{align*}
    |(e^{-\frac{1}{4}s}\partial_\alpha U+\partial_\alpha\kappa)\circ\Phi(X_0,s)|&\lesssim |\epsilon^\frac{1}{4}\partial_\alpha U(X_0,-\log\epsilon)|l^{-10}\\
    &\qquad+M^4l^{-10}\epsilon^{\frac{3}{4}+\frac{3}{8}}+M\epsilon^\frac{1}{2}l^{-10}\log\frac{1}{l}+\epsilon^{\frac{1}{2}+\frac{1}{4}}l^{-10}\log\frac{1}{l}\\
    &\lesssim \epsilon^\frac{1}{4}l^{-10}+M^4l^{-10}\epsilon^\frac{9}{8}+M\epsilon^\frac{1}{2}l^{-10}\log\frac{1}{l},\\
    \implies\qquad|\partial_\alpha U(X,s)+e^{\frac{1}{4}s}\partial_\alpha\kappa|&\lesssim l^{-10}\epsilon^\frac{1}{4}e^{\frac{3}{4}s}+M^4l^{-10}\epsilon^\frac{9}{8}e^{\frac{1}{4}s}+Ml^{-10}\log\frac{1}{l}\epsilon^\frac{1}{2}e^{\frac{1}{4}s}\leq \frac{1}{2}M^4\epsilon^\frac{3}{4}e^{\frac{3}{4}s}.
\end{align*}
If $X_0=l$ for some $s_0>-\log\epsilon$, then the initial data contribution becomes
\begin{align*}
    (Ml^\frac{1}{2}\epsilon^\frac{3}{4}e^{\frac{1}{2}s}+\epsilon^\frac{3}{4})l^{-10}\lesssim Ml^{-9\frac{1}{2}}\epsilon^\frac{3}{4}e^{\frac{1}{2}s},
\end{align*}
which is dominated by $M^4\epsilon^\frac{3}{4}e^{\frac{3}{4}s}$. So we are done.

Then we consider the case when $|X|\geq \frac{1}{2}e^{\frac{5}{4}s}$. If $X_0\geq \frac{1}{2}e^{\frac{5}{4}s}$, then
\begin{align*}
    |(e^{-\frac{1}{4}s}\partial_\alpha U+\partial_\alpha\kappa)\circ\Phi(X_0,s)|\lesssim \int_{s_0}^s(M^4\epsilon^\frac{3}{4}e^{-\frac{3}{8}s'}+ M\epsilon^\frac{1}{2}e^{-s'}e^{5e^{-s'}})\,ds'\lesssim M^4\epsilon^{\frac{3}{4}+\frac{3}{8}}.
\end{align*}
If $X_0=\pm\frac{1}{2}e^{\frac{5}{4}s_0}$ for some $s_0>-\log\epsilon$, then we have the additional initial data contribution
\begin{align*}
    |(e^{-\frac{1}{4}s_0}\partial_\alpha U+\partial_\alpha\kappa)(\pm\frac{1}{2}e^{\frac{5}{4}s_0},s_0)|e^{5e^{-s_0}}\leq \frac{3}{2}M^4\epsilon^\frac{3}{4}e^{\frac{1}{2}s_0}.
\end{align*}
In both cases, we can establish
\begin{align*}
    |\partial_\alpha U(X,s)+e^{\frac{1}{4}s}\partial_\alpha\kappa|\leq \frac{7}{4}M^4\epsilon^\frac{3}{4}e^{\frac{1}{4}s}.
\end{align*}
\end{proof}

\subsection{Pointwise estimate: near field}
\begin{proposition}[$\partial_\alpha\partial_X^7U$, $\partial_\alpha\partial_X^6U$]
For all $|X|\leq l$, we have
\begin{align*}
    |\partial_\alpha\partial_X^7U(X,s)|\leq M\epsilon^\frac{3}{4}e^{\frac{3}{4}s},\qquad |\partial_\alpha\partial_X^6U(X,s)|\leq Ml^\frac{1}{2}\epsilon^\frac{3}{4}e^{\frac{3}{4}s}.
\end{align*}
\end{proposition}

\begin{proof}
We first work on $\partial_\alpha\partial_X^7U$. We use \eqref{eq:eqnxUalpha} and will compose it with the Lagrangian trajectories. Note that we must have $|X_0|\leq l$. The damping is
\begin{align*}
D^{(7)}_\alpha=\frac{35}{4}-\frac{1}{4}+\frac{8}{1-\dot{\tau}}\partial_XU\geq \frac{34}{4}-8(1+2\epsilon^\frac{1}{5}e^{-\frac{3}{4}s})(1+\epsilon^\frac{17}{20})\geq \frac{1}{4}.
\end{align*}
We also have the forcing
\begin{align*}
|F^{(7)}_{U,\alpha}|
% \lesssim M^{84\frac{1}{2}}\epsilon^\frac{3}{4}e^{-\frac{1}{4}s}+M^{\frac{429}{8}}\epsilon^\frac{1}{2}e^{-s}+(M^4 \epsilon^\frac{3}{4}e^{\frac{3}{4}s}
% +M\epsilon^\frac{1}{2}e^{\frac{1}{4}s})M^3\epsilon^\frac{1}{5}+\epsilon^\frac{1}{2}(2Me^{\frac{1}{4}s}+\epsilon^{\frac{1}{5}}e^{-\frac{3}{4}s})M^3\epsilon^\frac{1}{5}\\
% +Ml^\frac{1}{2}\epsilon^\frac{3}{4}e^{\frac{3}{4}s}M\epsilon^\frac{1}{5}
% +Ml^\frac{1}{2}\epsilon^\frac{3}{4}e^{\frac{3}{4}s}(\epsilon^\frac{1}{5}+2l\epsilon^\frac{1}{5}+\partial_X^5U_2(0)+2l^2\epsilon^\frac{1}{5}+2l^3\epsilon^\frac{1}{5}+2l^4\epsilon^\frac{1}{5})  \\
% +\epsilon^\frac{1}{2}\Big[(1+2l^5\epsilon^\frac{1}{5})M\epsilon^\frac{1}{5}+2l^4\epsilon^\frac{2}{5}+2l^3\epsilon^\frac{1}{5}\big(2l\epsilon^\frac{1}{5}+\partial_X^5U_2(0)\big)+4l^4\epsilon^\frac{2}{5}\Big]\\
\lesssim M^7\epsilon^\frac{19}{20}e^{\frac{3}{4}s}+Ml^\frac{1}{2}\epsilon^\frac{3}{4}e^{\frac{3}{4}s}.
\end{align*}
So
\begin{align*}
|\partial_\alpha\partial_X^7U\circ\Phi(X_0,s)|&\leq |\partial_\alpha\partial^7_XU(X_0,-\log\epsilon)|e^{-\frac{1}{4}(s+\log\epsilon)}+C(M^7\epsilon^\frac{19}{20}+Ml^\frac{1}{2}\epsilon^\frac{3}{4})\int_{-\log\epsilon}^se^{\frac{3}{4}s'}e^{-\frac{1}{4}(s-s')}\,ds'\\
&\lesssim (M^7\epsilon^\frac{19}{20}+Ml^\frac{1}{2}\epsilon^\frac{3}{4})e^{\frac{3}{4}s}\leq \frac{3}{4}M\epsilon^\frac{3}{4}e^{\frac{3}{4}s}.
\end{align*}
 
We use the same idea for $\partial_\alpha\partial_X^6U$, in which case
\begin{align*}
D^{(6)}_\alpha=\frac{29}{4}+\frac{7}{1-\dot{\tau}}\partial_XU\geq \frac{29}{4}-7(1+2\epsilon^\frac{1}{5})(1+\epsilon^\frac{17}{20})\geq \frac{1}{8},
\end{align*}
and
\begin{align*}
    F^{(6)}_{U,\alpha}\lesssim M^5\epsilon^\frac{3}{4}e^{\frac{3}{4}s}\epsilon^\frac{1}{5},
\end{align*}
so
\begin{align*}
    |\partial_\alpha\partial_X^6U\circ\Phi(X_0,s)|&\leq |\partial_\alpha\partial_X^6U(X_0,-\log\epsilon)|e^{-\frac{1}{8}(s+\log\epsilon)}+\int_{-\log\epsilon}^s CM^5\epsilon^\frac{3}{4}e^{\frac{3}{4}s'}\epsilon^\frac{1}{5}e^{-\frac{1}{8}(s-s')}\,ds'\\
&\lesssim M^5\epsilon^\frac{1}{5}\epsilon^\frac{3}{4}e^{\frac{3}{4}s}
\leq \frac{3}{4}Ml^\frac{1}{2}\epsilon^\frac{3}{4}e^{\frac{3}{4}s}.
\end{align*}
\end{proof}

% \begin{align*}
% F^{(6)}_{U,\alpha}=\frac{e^{-s}}{1-\dot{\tau}}H[\partial_\alpha\partial_X^6U]+\frac{e^{-s}}{(1-\dot{\tau})^2}\partial_\alpha\dot{\tau}H[\partial_X^6U]-\frac{1}{1-\dot{\tau}}(\partial_\alpha U+e^{\frac{1}{4}s}\partial_\alpha\kappa-e^{\frac{1}{4}s}\partial_\alpha\dot{\xi})\partial_X^7U
% \\
% -\frac{1}{(1-\dot{\tau})^2}\partial_\alpha\dot{\tau}(U+e^{\frac{1}{4}s}\kappa-e^{\frac{1}{4}s}\dot{\xi})\partial_X^7U-\frac{1}{1-\dot{\tau}}\Big[7\partial_\alpha\partial_XU\partial_X^6U+\sum_{j=2}^{5}{7\choose j}\partial_\alpha\partial_X^jU\partial_X^{7-j}U\Big]
% \\
% -\frac{\partial_\alpha\dot{\tau}}{(1-\dot{\tau})^2}\Big[7\partial_XU\partial_X^6U+{7\choose 2}\partial_X^2U\partial_X^5U+{7\choose 3}\partial_X^3U\partial_X^4U\Big]\\
% \lesssim e^{-s}M^\frac{410}{7}\epsilon^\frac{3}{4}e^{\frac{3}{4}s}+e^{-s}\epsilon^\frac{1}{2}M^{\frac{363}{8}}+(M^4 \epsilon^\frac{3}{4}e^{\frac{3}{4}s}+M\epsilon^\frac{1}{2}e^{\frac{1}{4}s})M\epsilon^\frac{1}{5}+\epsilon^\frac{1}{2}(2Me^{\frac{1}{4}s}+\epsilon^\frac{1}{5}e^{-\frac{3}{4}s})M\epsilon^\frac{1}{5}\\
% +\Big[Ml^\frac{1}{2}\epsilon^\frac{3}{4}e^{\frac{3}{4}s}\epsilon^\frac{1}{5}+Ml^\frac{1}{2}\epsilon^\frac{3}{4}e^{\frac{3}{4}s}2\epsilon^\frac{1}{5}(l+l^2+l^3+l^4)\Big]+\epsilon^\frac{1}{2}\Big[\epsilon^\frac{1}{5}l^5\epsilon^\frac{1}{5}+l^4\epsilon^\frac{1}{5}l\epsilon^\frac{1}{5}+l^3\epsilon^\frac{1}{5}l^2\epsilon^\frac{1}{5}\Big]
% \\\lesssim M^5\epsilon^\frac{3}{4}e^{\frac{3}{4}s}\epsilon^\frac{1}{5}+Ml^\frac{1}{2}\epsilon^\frac{3}{4}e^{\frac{3}{4}s}\epsilon^\frac{1}{5}\text{non dominant terms}
% \end{align*}

\begin{proposition}[$\partial_\alpha\partial_X^nU$ for $n=0,...,5$]
For $|X|\leq l$, we have
\begin{align*}
    |\partial_\alpha\partial_X^n U(X,s)|\leq Ml^\frac{1}{2}\epsilon^\frac{3}{4}e^{\frac{3}{4}s}, \qquad n=0,...,5. 
\end{align*}
\end{proposition}
\begin{proof}
We start with $\partial_\alpha U(X,s)$. This can be done by the mean value theorem (or Taylor expansion)
\begin{align*}
|\partial_\alpha U(X,s)|\leq |\partial_\alpha U(0,s)|+l\sup_{|X|\leq l}|\partial_\alpha\partial_X U(X,s)|\leq Ml^\frac{3}{2}\epsilon^\frac{3}{4}e^{\frac{3}{4}s}\leq \frac{1}{2}Ml^\frac{1}{2}\epsilon^\frac{3}{4}e^{\frac{3}{4}s}.
\end{align*}
Same argument for $\partial_\alpha\partial_XU$ and $\partial_\alpha\partial_X^4U$ due to the constraint of vanishing at $X=0$.

For $\partial_\alpha\partial_X^2U$, same idea but we use the assumption 
\begin{align*}
\frac{1}{2}\epsilon^\frac{3}{4}e^{\frac{3}{4}s}\leq \partial_\alpha\partial_X^2U(0,s)\leq 4\epsilon^\frac{3}{4}e^{\frac{3}{4}s},\qquad
|\partial_\beta\partial_X^2U(0,s)|\leq \epsilon e^{\frac{3}{4}s}.
\end{align*}
So (the $\partial_\beta$ is even easier)
\begin{align*}
|\partial_\alpha\partial_X^2U(X,s)|\leq 4\epsilon^\frac{3}{4}e^{\frac{3}{4}s}+lMl^\frac{1}{2}\epsilon^\frac{3}{4}e^{\frac{3}{4}s}\leq 2Ml^\frac{3}{2}\epsilon^\frac{3}{4}e^{\frac{3}{4}s}\leq \frac{1}{2}Ml^\frac{1}{2}\epsilon^\frac{3}{4}e^{\frac{3}{4}s}.
\end{align*}
For $\partial_\beta\partial_X^3U$, recall the assumption
\begin{align*}
|\partial_\alpha\partial_X^3U(0,s)|\leq \epsilon e^{\frac{1}{2}s},\qquad
\frac{1}{2}\epsilon^\frac{1}{2}e^{\frac{1}{2}s}\leq \partial_\beta\partial_X^3U(0,s)\leq 8\epsilon^\frac{1}{2}e^{\frac{1}{2}s}.
\end{align*}
So ($\partial_\alpha$ is easier in this case)
\begin{align*}
|\partial_\beta\partial_X^3U(X,s)|\leq 8\epsilon^\frac{1}{2}e^{\frac{1}{2}s}+lMl^\frac{1}{2}\epsilon^\frac{3}{4}e^{\frac{3}{4}s}\leq \frac{1}{2}Ml^\frac{1}{2}\epsilon^\frac{3}{4}e^{\frac{3}{4}s}.
\end{align*}
For $\partial_\alpha\partial_X^5U$, we use the assumption (note that $\partial_\alpha\partial_X^5\widetilde{U}=\partial_\alpha\partial_X^5 U$)
\begin{align*}
|\partial_\alpha\partial_X^5\widetilde{U}(0,s)|\leq \epsilon^\frac{3}{8}e^{\frac{1}{8}s},
\end{align*}
and same for $\partial_\beta$. Then
\begin{align*}
|\partial_\alpha\partial_X^5U(X,s)|\leq \epsilon^\frac{3}{8}e^{\frac{1}{8}s}+lMl^\frac{1}{2}\epsilon^\frac{3}{4}e^{\frac{3}{4}s}\leq \frac{1}{2}Ml^\frac{1}{2}\epsilon^\frac{3}{4}e^{\frac{3}{4}s}.
\end{align*}
\end{proof}

\subsection{Pointwise estimate: middle field}
We note that for $l\leq|X|\leq \frac{1}{2}e^{\frac{5}{4}s}$, since 
\begin{align*}
    |\partial_\alpha U(X,s)+e^{\frac{1}{4}s}\partial_\alpha\kappa|\leq M^4\epsilon^\frac{3}{4}e^{\frac{3}{4}s},\qquad |\partial_\alpha\kappa|\leq \epsilon^\frac{1}{2},
\end{align*}
We have (same for $\partial_\beta U$)
\begin{align*}
    |\partial_\alpha U(X,s)|\leq \frac{9}{8}M^4\epsilon^\frac{3}{4}e^{\frac{3}{4}s}.
\end{align*}
\begin{proposition}[$\partial_\alpha\partial_X^nU$, $n=1,...,7$]
For $(X,s)$ such that $l\leq |X|\leq \frac{1}{2}e^{\frac{5}{4}s}$, we have
\begin{align*}
    |\partial_\alpha\partial_X^nU(X,s)|\leq M^{(n+2)^2}\epsilon^\frac{3}{4}e^{\frac{3}{4}s}(1+X^4)^{-\frac{1}{5}},\qquad n=1,...,7.
\end{align*}
\end{proposition}

\begin{proof}
The calculations are very similar to the case without parameter derivative. As what has been done before, for $n\geq 1$,
\begin{align*}
D^{(n)}_{\alpha,\mathrm{w}}&=\frac{5}{4}n-\frac{1}{4}+\frac{n+1}{1-\dot{\tau}}\partial_XU-\frac{X^4}{1+X^4}-\frac{4X^3}{5(1+X^4)}\frac{U+e^{\frac{1}{4}s}(\kappa-\dot{\xi})}{1-\dot{\tau}}\\
&\geq \frac{5}{4}(n-1)-(n+3)(1+X^4)^{-\frac{1}{5}}.
\end{align*}
The subscript $\alpha$ refers to $\partial_\alpha$, ``w" refers to weighted, the superscript refers to $\partial_X^nU$. We have seen before that 
\begin{align*}
\int_{s_0}^s-D_{\alpha,\mathrm{w}}^{(n)}\circ\Phi(X_0,s')\,ds'\leq -\frac{5}{4}(n-1)+10(n+3)\log\frac{1}{l}.
\end{align*}
Forcing contribution is  
\begin{align*}
&\int_{s_0}^s(1+X^4)^\frac{1}{5}F_{U,\alpha}^{(n)}\circ\Phi(X_0,s')\exp\Big(\int_{s'}^s-D_{\alpha,\mathrm{w}}^{(n)}\circ(X_0,s'')\,ds''\Big)ds'\\
\lesssim_n &\Big(M^{(n+1)^2+4}+\sum_{j=1}^{n-1}M^{(j+2)^2+(n-j+1)^2}\Big)\epsilon^\frac{3}{4}\int_{s_0}^se^{\frac{3}{4}s'}e^{-\frac{5}{4}(n-1)(s-s')}l^{-10(n+3)}\,ds'\\
\lesssim_n& M^{(n+1)^2+4}\epsilon^\frac{3}{4}l^{-10(n+3)}e^{\frac{3}{4}s}.
\end{align*}
Hence,
\begin{align*}
|(1+X^4)^\frac{1}{5}\partial_\alpha\partial_X^nU\circ\Phi(X_0,s)|&\leq |(1+X_0^4)^\frac{1}{5}\partial_\alpha\partial_X^nU(X_0,s_0)|l^{-10(n+3)}e^{-\frac{5}{4}(n-1)(s-s_0)}\\
&\qquad +C_nM^{(n+1)^2+4}l^{-10(n+3)}\epsilon^\frac{3}{4}e^{\frac{3}{4}s}.
\end{align*}
Initial data part vanishes for $n\geq 4$, or for $|X_0|\geq 2$. For $n=1,\,2,\,3$, in the case when $l\leq |X_0|\leq 2$, we can simply treat  \begin{align*}
|(1+X_0^4)^\frac{1}{5}\partial_\alpha\partial^n_XU(X_0,-\log\epsilon)|\leq C_n.
\end{align*}
If $X_0=l$ for $s_0>-\log\epsilon$, we use the near field estimates 
\eqref{eq:assumptionparanear0-6},\,\eqref{eq:assumptionparanear7}. In all cases, we can obtain 
\begin{align*}
|(1+X^4)^\frac{1}{5}\partial_\alpha\partial_X^nU\circ\Phi(X_0,s)|\leq \frac{3}{4}M^{(n+2)^2}\epsilon^\frac{3}{4}e^{\frac{3}{4}s}.
\end{align*}
\end{proof}

\subsection{Pointwise estimate: far field}
We consider equations for $e^s\partial_\alpha\partial_X^nU$. 
\begin{align*}
\Big(\partial_s+\frac{5}{4}(n-1)+\frac{n+1}{1-\dot{\tau}}\partial_XU\Big)(e^s\partial_\alpha\partial_X^nU)+V(e^s\partial_\alpha\partial_X^{n+1}U)=e^sF^{(n)}_{U,\alpha}.
\end{align*}
\begin{proposition}[$\partial_\alpha\partial_XU$]
For $(X,s)$ such that $|X|\geq \frac{1}{2}e^{\frac{5}{4}s}$, we have
\begin{align*}
    |\partial_\alpha\partial_XU(X,s)|\leq 4M^9\epsilon^\frac{3}{4}e^{-\frac{1}{4}s}.
\end{align*}
\end{proposition}
\begin{proof}
In this case it is still negative damping,  
\begin{align*}
-D^{(1)}_{\alpha,s}=-\frac{2}{1-\dot{\tau}}\partial_XU\leq 2(1+2\epsilon^\frac{1}{5}e^{-\frac{3}{4}s})4e^{-s}\leq 9e^{-s},
\end{align*}
where the subscript $s$ standards for weighted with $e^s$. Similar as before,
\begin{align*}
\exp\Big(\int_{s_0}^s-D_{\alpha,s}^{(1)}\circ\Phi(X_0,s')\,ds'\Big)\leq e^{-9s_0}\leq 1+2\cdot9\epsilon\leq \frac{5}{4}.
\end{align*}
Hence,
\begin{align*}
|e^s\partial_\alpha\partial_XU\circ\Phi(X_0,s)|&\leq \frac{5}{4}|e^{s_0}\partial_\alpha\partial_XU(X_0,s_0)|+\frac{5}{4}C_1M^8\epsilon^\frac{3}{4}\int_{s_0}^s  e^{\frac{3}{4}s'}\,ds'\\
&\leq \frac{5}{4}|e^{s_0}\partial_\alpha\partial_XU(X_0,s_0)|+\frac{5}{3}C_1M^8\epsilon^\frac{3}{4}e^{\frac{3}{4}s}.
\end{align*}
If $X=\Phi(X_0,-\log\epsilon)$ for $|X_0|\geq \frac{1}{2}\epsilon^{-\frac{5}{4}}$, then 
\begin{align*}
\partial_\alpha\partial_XU(X_0,-\log\epsilon)=0.
\end{align*}
If $X=\Phi(\pm\frac{1}{2}e^{\frac{5}{4}s_0},s_0)$ for $s_0>-\log\epsilon$, then we use the middle field estimate
\begin{align*}
|\partial_\alpha\partial_XU(\pm\frac{1}{2}e^{\frac{5}{4}s_0})|&\leq M^9\big(1+(\frac{1}{2}e^{\frac{5}{4}s_0})^4\big)^{-\frac{1}{5}}\leq M^92^\frac{4}{5}e^{-s_0},\\
e^{s_0}|\partial_\alpha\partial_XU(\pm\frac{1}{2}e^{\frac{5}{4}s_0})|&\leq 2^\frac{4}{5}M^9.
\end{align*}
In both cases, we have
\begin{align*}
    |\partial_\alpha\partial_XU(X,s)|\leq 4M^9\epsilon^\frac{3}{4}e^{-\frac{1}{4}s}.
\end{align*}
\end{proof}

\begin{proposition}[$\partial_\alpha\partial_X^nU$, $n=2,...,7$]
For $(X,s)$ such that $|X|\geq\frac{1}{2}e^{\frac{5}{4}s}$, we have
\begin{align*}
    |\partial_\alpha\partial_X^nU(X,s)|\leq 4M^{(n+2)^2}\epsilon^\frac{3}{4}e^{-\frac{1}{4}s},\qquad n=2,...,7.
\end{align*}
\end{proposition}
\begin{proof}
In this case, damping is positive, 
\begin{align*}
D_{\alpha,s}^{(n)}\geq \frac{5}{4}(n-1)-(n+1)(1+2\epsilon^\frac{1}{5}e^{-\frac{3}{4}s})\cdot 4e^{-s} \geq n-1.
\end{align*}
Then
\begin{align*}
|e^s\partial_\alpha\partial_X^nU\circ\Phi(X_0,s)|&\leq |e^{s_0}\partial_\alpha\partial_X^nU(X_0,s_0)|e^{-(n-1)(s-s_0)}\\
&+C_n(M^{4+(n+1)^2}+\sum_{j=1}^{n-1}M^{(j+2)^2+(n-j+1)^2})\epsilon^\frac{3}{4}\int_{s_0}^se^{\frac{3}{4}s'}e^{-(n-1)(s-s')}\,ds'\\
&\leq |e^{s_0}\partial_\alpha\partial_X^nU(X_0,s_0)|e^{-(n-1)(s-s_0)}+C'_nM^{(n+1)^2+4}\epsilon^\frac{3}{4}e^{\frac{3}{4}s}.
\end{align*}
Same as the $\partial_\alpha\partial_XU$ case, either the initial data contribution vanishes, or from the middle field estimate
\begin{align*}
\big|\partial_\alpha\partial_X^nU\big(\pm\frac{1}{2}e^{\frac{5}{4}s_0}\big)\big|\leq M^{(n+2)^2}2^\frac{4}{5}e^{-s_0}.
\end{align*}
Hence, we have
\begin{align*}
e^s|\partial_\alpha\partial_X^nU\circ\Phi(X_0,s)|&\leq 3M^{(n+2)^2}\epsilon^\frac{3}{4}e^{\frac{3}{4}s}\\
|\partial_\alpha\partial_X^nU(X,s)|&\leq 3M^{(n+2)^2}\epsilon^\frac{3}{4}e^{-\frac{1}{4}s}.
\end{align*}
\end{proof}

\subsection{Modulation variables}
Take $\partial_\alpha$ to the equations of modulation variables, we get
\begin{align*}
-e^{\frac{1}{4}s}\partial_\alpha(\kappa-\dot{\xi})&=e^{-\frac{3}{4}s}\partial_\alpha\dot{\kappa}+e^{-s}H[\partial_\alpha U+e^{\frac{1}{4}s}\partial_\alpha\kappa](0,s),\\
e^{-s}H[\partial_\alpha\partial_XU](0,s)&=\partial_\alpha\dot{\tau}+e^{\frac{1}{4}s}\partial_\alpha(\kappa-\dot{\xi})\partial_X^2U(0,s)+e^{\frac{1}{4}s}(\kappa-\dot{\xi})\partial_\alpha\partial_X^2U(0,s),\\
e^{\frac{1}{4}s}\partial_\alpha(\kappa-\dot{\xi})\partial_X^5U(0,s)&+e^{\frac{1}{4}s}(\kappa-\dot{\xi})\partial_\alpha\partial_X^5U(0,s)\\
&=e^{-s}H[\partial_\alpha\partial_X^4U](0,s)-10\partial_\alpha\partial_X^2U(0,s)\partial_X^3U(0,s)
-10\partial_X^2U(0,s)\partial_\alpha\partial_X^3U(0,s).
\end{align*}
We look at the third equation first. Here we need to treat $\alpha$ and $\beta$ differently.
\begin{align*}
|e^{\frac{1}{4}s}\partial_\alpha(\kappa-\dot{\xi})\partial_X^5U(0,s)|&\lesssim \epsilon^{\frac{1}{5}+\frac{3}{8}}e^{(\frac{1}{4}-1+\frac{1}{8})s}+M^{51\frac{1}{2}}\epsilon^\frac{3}{4}e^{-\frac{1}{4}s}+\epsilon^\frac{3}{4}e^{\frac{3}{4}s}M^{27} e^{-s}+\epsilon^\frac{1}{10}e^{-\frac{3}{4}s}\epsilon e^{\frac{1}{2}s}\\
&\lesssim M^{51\frac{1}{2}}\epsilon^\frac{3}{4}e^{-\frac{1}{4}s}.
\end{align*}
For $\beta$, we have 
\begin{align*}
|e^{\frac{1}{4}s}\partial_\beta(\kappa-\dot{\xi})\partial_X^5U(0,a)|&\lesssim \epsilon^{\frac{1}{5}+\frac{3}{8}}e^{(\frac{1}{4}-1+\frac{1}{8})s}+M^{51\frac{1}{2}}\epsilon^\frac{3}{4}e^{-\frac{1}{4}s}+\epsilon e^{\frac{3}{4}s} M^{27}e^{-s}+\epsilon^\frac{1}{10}e^{-\frac{3}{4}s}\epsilon^\frac{1}{2}e^{\frac{1}{2}s}\\
&\lesssim M^{51\frac{1}{2}}\epsilon^\frac{3}{4}e^{-\frac{1}{4}s}+\epsilon^\frac{3}{5}e^{-\frac{1}{4}s}
\end{align*}
with a different constant. So for both $\alpha$ and $\beta$, since $\partial_X^5U(0,s)\geq 119$, we have (we use $\alpha$ to demonstrate)
\begin{align*}
|\partial_\alpha(\kappa-\dot{\xi})|\lesssim M^{51\frac{1}{2}}\epsilon^\frac{3}{4}e^{-\frac{1}{2}s}+\epsilon^\frac{3}{5}e^{-\frac{1}{2}s}\leq \frac{3}{4}M\epsilon^\frac{1}{2}e^{-\frac{1}{2}s}.
\end{align*}

Now we look at the second equation, 
\begin{align*}
|\partial_\alpha\dot{\tau}|&\leq M\epsilon^\frac{1}{2}e^{-\frac{1}{4}s}\epsilon^\frac{1}{10}e^{-\frac{3}{4}s}+4\epsilon^\frac{3}{4}e^{\frac{3}{4}s}\epsilon^\frac{1}{5}e^{-\frac{3}{4}s}+CM^{18\frac{1}{2}}\epsilon^\frac{3}{4} e^{-\frac{1}{4}s}\\
&\lesssim \epsilon^{\frac{3}{4}+\frac{1}{5}}\leq \frac{1}{2}\epsilon^\frac{1}{2}.
\end{align*}

Finally we look at the first equation
\begin{align*}
|\partial_\alpha\dot{\kappa}|\leq e^{\frac{1}{2}s}M\epsilon^\frac{1}{2}e^{-\frac{1}{2}s}+Ce^{-\frac{1}{4}s}M^9\epsilon^\frac{3}{4}se^{\frac{3}{4}s}\leq \frac{3}{4}\epsilon^\frac{1}{2}se^{\frac{1}{2}s}.
\end{align*}
Note that $\kappa(-\log\epsilon)$ corresponds to $U(0,-\log\epsilon)$, which is independent of $\alpha,\,\beta$, so
\begin{align*}
\partial_\alpha\kappa(s)&=\int_{-\log\epsilon}^s\partial_\alpha\dot{\kappa}\frac{dt}{ds}\,ds',\\
|\partial_\alpha\kappa|&\leq \int_{-\log\epsilon}^\infty \epsilon^\frac{1}{2}s'e^{\frac{1}{2}s'}\frac{e^{-s'}}{1-\dot{\tau}}\,ds'\leq 2(1+2\epsilon^\frac{1}{5}e^{-\frac{3}{4}s})\epsilon^\frac{1}{2}(-\log\epsilon+2)\epsilon^\frac{1}{2}\leq \frac{1}{2}\epsilon^\frac{1}{2}.
\end{align*}
Then taking the sum/difference, we get
\begin{align*}
|\partial_\alpha\dot{\xi}|\leq M\epsilon^\frac{1}{2}e^{-\frac{1}{2}s}+\epsilon^\frac{1}{2}\leq \frac{1}{2}M\epsilon^\frac{1}{2}.
\end{align*}

\section{Closure of bootstrap with second order parameter derivatives}\label{sec:2ndparaclosure}
In this section, we use $\partial^2_{\alpha\beta}$ to represent $\partial^2_{\alpha\alpha},\, \partial^2_{\beta\beta}$ and $\partial^2_{\alpha\beta}$. And the estimates hold for all $(\alpha,\beta)\in B$.
\subsection{$L^2$ estimates}
\begin{proposition}[$\partial^2_{\alpha\beta}U$]
We have
\begin{align*}
    \|\partial_{\alpha\beta}^2U(\cdot,s)+e^{\frac{1}{4}s}\partial_{\alpha\beta}^2\kappa\|_{L^2}\leq M^{38}\epsilon^\frac{3}{2}e^{\frac{7}{4}s}.
\end{align*}
\end{proposition}

\begin{proof}
We consider the equation for $e^{-\frac{1}{4}s}\partial_{\alpha\beta}^2U+\partial_{\alpha\beta}^2\kappa$
\begin{align*}
\partial_s(e^{-\frac{1}{4}s}\partial_{\alpha\beta}^2U+\partial_{\alpha\beta}^2\kappa)+e^{-\frac{1}{4}s}V\partial_{\alpha\beta}^2\partial_XU+e^{-\frac{1}{4}s}\partial_\beta V\partial_\alpha\partial_XU+e^{-\frac{1}{4}s}\partial_\alpha V\partial_\beta\partial_XU+e^{-\frac{1}{4}s}\partial_{\alpha\beta}^2V\partial_XU\\
=\frac{e^{-s}}{1-\dot{\tau}}H[e^{-\frac{1}{4}s}\partial_{\alpha\beta}^2U+\partial_{\alpha\beta}^2\kappa]+\frac{e^{-s}}{(1-\dot{\tau})^2}\partial_\beta\dot{\tau}H[e^{-\frac{1}{4}s}\partial_\alpha U+\partial_\alpha\kappa]\\
+\frac{e^{-s}}{(1-\dot{\tau})^2}\partial_\alpha\dot{\tau}H[e^{-\frac{1}{4}s}\partial_\beta U+\partial_\beta\kappa]+\frac{2e^{-s}}{(1-\dot{\tau})^3}\partial_\alpha\dot{\tau}\partial_\beta\dot{\tau}H[e^{-\frac{1}{4}s}U+\kappa],
\end{align*}
where 
\begin{align*}
\partial_\alpha V&=\frac{1}{1-\dot{\tau}}\big(\partial_\alpha U+e^{\frac{1}{4}s}\partial_\alpha(\kappa-\dot{\xi})\big)+\frac{\partial_\alpha\dot{\tau}}{(1-\dot{\tau})^2}\big(U+e^{\frac{1}{4}s}(\kappa-\dot{\xi})\big),\\
\partial_{\alpha\beta}^2V&=\frac{1}{1-\dot{\tau}}\big(\partial_{\alpha\beta}^2U+e^{\frac{1}{4}s}\partial_{\alpha\beta}^2(\kappa-\dot{\xi})\big)+\frac{\partial_\beta\dot{\tau}}{(1-\dot{\tau})^2}\big(\partial_\alpha U+e^{\frac{1}{4}s}\partial_\alpha(\kappa-\dot{\xi})\big)\\
&\quad +\frac{\partial_\alpha\dot{\tau}}{(1-\dot{\tau})^2}(\partial_\beta U+e^{\frac{1}{4}s}\partial_\beta(\kappa-\dot{\xi})\big)+\frac{\partial_{\alpha\beta}^2\dot{\tau}}{(1-\dot{\tau})^2}\big(U+e^{\frac{1}{4}s}(\kappa-\dot{\xi})\big)+\frac{2\partial_\alpha\dot{\tau}\partial_\beta\dot{\tau}}{(1-\dot{\tau})^3}\big(U+e^{\frac{1}{4}s}(\kappa-\dot{\xi})\big).
\end{align*}
We take $L^2$ inner product, integrate by parts, and we get
\begin{align*}
\frac{1}{2}\frac{d}{ds}\|e^{-\frac{1}{4}s}\partial_{\alpha\beta}^2U+\partial_{\alpha\beta}^2\kappa\|_{L^2}^2-\frac{5}{8}\|e^{-\frac{1}{4}s}\partial_{\alpha\beta}^2U+\partial_{\alpha\beta}^2\kappa\|_{L^2}^2\lesssim M^{37}\epsilon^\frac{3}{2}e^{\frac{3}{2}s}\|e^{-\frac{1}{4}s}\partial_{\alpha\beta}^2U+\partial_{\alpha\beta}^2\kappa\|_{L^2}.
\end{align*}
The initial data part vanishes, so
\begin{align*}
\frac{d}{ds}\|e^{-\frac{1}{4}s}\partial_{\alpha\beta}^2U+\partial_{\alpha\beta}^2\kappa\|_{L^2}-\frac{5}{8}\|e^{-\frac{1}{4}s}\partial_{\alpha\beta}^2U+\partial_{\alpha\beta}^2\kappa\|_{L^2}&\lesssim M^{37}\epsilon^\frac{3}{2}e^{\frac{3}{2}s}\\
\|e^{-\frac{1}{4}s}\partial_{\alpha\beta}^2U+\partial_{\alpha\beta}^2\kappa\|_{L^2}&\lesssim M^{37}\epsilon^\frac{3}{2}\int_{-\log\epsilon}^s e^{\frac{3}{2}s'}e^{\frac{5}{8}(s-s')}\,ds'\\
&\lesssim M^{37}\epsilon^\frac{3}{2}e^{\frac{3}{2}s}\leq \frac{1}{2}M^{38}\epsilon^\frac{3}{2}e^{\frac{3}{2}s}.
\end{align*}
\end{proof}

\begin{proposition}[$\partial_{\alpha\beta}^2\partial_XU$]
We have
\begin{align*}
    \|\partial_{\alpha\beta}^2\partial_XU(\cdot,s)\|_{L^2}\leq M^{37}\epsilon^\frac{3}{2}e^{\frac{3}{2}s}.
\end{align*}
\end{proposition}

\begin{proof}
We take $L^2$ inner product of \eqref{eq:eqnxUalphabeta} with $n=1$ with $\partial_{\alpha\beta}^2\partial_XU$
\begin{align*}
\frac{1}{2}\frac{d}{ds}\|\partial_{\alpha\beta}^2\partial_XU\|_{L^2}^2+\frac{3}{8}\|\partial_{\alpha\beta}^2\partial_XU\|_{L^2}^2+\frac{3}{2(1-\dot{\tau})}\int_\mathbb{R}\partial_XU(\partial_{\alpha\beta}^2\partial_XU)^2\,dX= \text{RHS},
\end{align*}
where by $L^\infty$-$L^2$-$L^2$ H\"older's inequality,
\begin{align*}
|\text{right hand side}|\lesssim  M^{33\frac{3}{4}}  \epsilon^\frac{3}{2}e^{\frac{3}{2}s}\|\partial_{\alpha\beta}^2\partial_XU\|_{L^2},
% \text{RHS}\lesssim \|\partial_{\alpha\beta}^2\partial_XU\|_{L^2}\Big\{2e^{-s}\epsilon^\frac{1}{2} M^{13}\epsilon^\frac{3}{4}e^{\frac{3}{4}s}+e^{-s}(\epsilon^\frac{1}{2}e^{\frac{3}{4}s}+2\epsilon)+2M^9\epsilon^\frac{3}{4}e^{\frac{3}{4}s}M^{13}\epsilon^\frac{3}{4}e^{\frac{3}{4}s}\\
% +2M^4\epsilon^\frac{3}{4}e^{\frac{3}{4}s}M^{\frac{149}{7}}\epsilon^\frac{3}{4}e^{\frac{3}{4}s}+\|\partial_{\alpha\beta}^2U+e^{\frac{1}{4}s}\partial_{\alpha\beta}^2(\kappa-\dot{\xi})\|_{L^\infty}M^{\frac{33}{4}}+2\epsilon^\frac{1}{2}M^9\epsilon^\frac{3}{4}e^{\frac{3}{4}s}+2\epsilon^\frac{1}{2}Me^{\frac{1}{4}s}M^{\frac{149}{7}}\\
% +\epsilon^\frac{1}{2}e^{\frac{3}{4}s}Me^{\frac{1}{4}s}M^{\frac{33}{4}}+2\epsilon^\frac{1}{2}M^4\epsilon^\frac{3}{4}e^{\frac{3}{4}s}M^{\frac{33}{4}}+\epsilon^\frac{1}{2}e^{\frac{3}{4}s}+2\epsilon^\frac{1}{2}M^9+2\epsilon Me^{\frac{1}{4}s}M^{\frac{33}{4}}+\epsilon\Big\}\\
% \lesssim (M^{9+13}+M^{4+\frac{149}{7}}+M^{\frac{33}{4}+36}+\text{non dominant terms})\epsilon^\frac{3}{2}e^{\frac{3}{2}s}\|\partial_{\alpha\beta}^2\partial_XU\|_{L^2}
\end{align*}
since the dominant term is 
\begin{align*}
    \frac{1}{1-\dot{\tau}}\int_\mathbb{R}(\partial_{\alpha\beta}^2U+e^{\frac{1}{4}s}\partial_{\alpha\beta}^2\kappa)\partial_X^2U\partial_{\alpha\beta}^2\partial_XU\,dX.
\end{align*}
We in fact move the third term on the left to the right hand side and use
\begin{align*}
\int_\mathbb{R}\partial_XU(\partial_{\alpha\beta}^2\partial_XU)^2\,dX\leq \|\partial_XU\|_{L^2}\|\partial_{\alpha\beta}^2\partial_XU\|_{L^\infty}\|\partial_{\alpha\beta}^2\partial_XU\|_{L^2}\lesssim M^{36}\epsilon^\frac{3}{2}e^{\frac{3}{2}s}\|\partial_{\alpha\beta}^2\partial_XU\|_{L^2}.
\end{align*}
So
\begin{align*}
\frac{d}{ds}\|\partial_{\alpha\beta}^2\partial_XU\|_{L^2}+\frac{3}{8}\|\partial_{\alpha\beta}^2\partial_XU\|_{L^2}&\lesssim M^{36}\epsilon^\frac{3}{2}e^{\frac{3}{2}s}\\
\|\partial_{\alpha\beta}^2\partial_XU\|_{L^2}&\leq CM^{36}\epsilon^\frac{3}{2}e^{\frac{3}{2}s}\leq \frac{1}{2}M^{37}\epsilon^\frac{3}{2}e^{\frac{3}{2}s},
\end{align*}
since initial data vanishes. 
\end{proof}

\begin{proposition}[$\partial_{\alpha\beta}^2\partial_X^7U$]
We have
\begin{align*}
    \|\partial_{\alpha\beta}^2\partial_X^7U(\cdot,s)\|_{L^2}\leq M^{130} \epsilon^\frac{3}{2}e^{\frac{3}{2}s}.
\end{align*}
\end{proposition}

\begin{proof}
Consider equation \eqref{eq:eqnxUalphabeta} with $n=7$. Note that $F^{(7)}_{\alpha\beta}$ contains $U$ to $\partial_X^8U$, $\partial_\alpha U$ to $\partial_\alpha\partial_X^8U$ (and $\partial_\beta$), and $\partial_{\alpha\beta}^2 U$ to $\partial_{\alpha\beta}^2\partial_X^7 U$. Take $L^2$ inner product with $\partial_{\alpha\beta}\partial_X^7U$, and integrate by parts, we arrive at
\begin{align*}
\frac{1}{2}\frac{d}{ds}\|\partial^2_{\alpha\beta}\partial_X^7U\|_{L^2}^2+\frac{63}{8}\|\partial_{\alpha\beta}^2\partial_X^7U\|_{L^2}^2+\frac{15}{2(1-\dot{\tau})}\int_\mathbb{R}\partial_XU(\partial_{\alpha\beta}^2\partial_X^7U)^2\,dX\lesssim M^{129\frac{3}{4}}\epsilon^\frac{3}{2}e^{\frac{3}{2}s}.
\end{align*}
Hence,
\begin{align*}
\frac{1}{2}\frac{d}{ds}\|\partial_{\alpha\beta}^2\partial_X^7U\|_{L^2}^2+\frac{1}{4}\|\partial_{\alpha\beta}^2\partial_X^7U\|_{L^2}^2&\lesssim M^{129\frac{3}{4}}\epsilon^\frac{3}{2}e^{\frac{3}{2}s}\|\partial_{\alpha\beta}^2\partial_X^7U\|_{L^2}\\
\frac{d}{ds}\|\partial_{\alpha\beta}^2\partial_X^7U\|_{L^2}+\frac{1}{4}\|\partial_{\alpha\beta}^2\partial_X^7U\|_{L^2}&\lesssim M^{129\frac{1}{4}}\epsilon^\frac{3}{2}e^{\frac{3}{2}s}\\
\|\partial_{\alpha\beta}^2\partial_X^7U\|_{L^2}&\leq CM^{129\frac{1}{4}}\epsilon^\frac{3}{2}e^{\frac{3}{2}s}\leq \frac{1}{2}M^{130}\epsilon^\frac{3}{2}e^{\frac{3}{2}s}.
\end{align*}
\end{proof}

\subsection{$L^\infty$ estimates of $\partial^2_{\alpha\beta} U$ and its derivatives}
\begin{proposition}[$\partial_{\alpha\beta}^2U$]
We have
\begin{align*}
    \|\partial_{\alpha\beta}^2U(\cdot,s)+e^{\frac{1}{4}s}\partial_{\alpha\beta}^2\kappa\|_{L^\infty}\leq M^{25}\epsilon^\frac{3}{2}e^{\frac{3}{2}s}.
\end{align*}
\end{proposition}

\begin{proof}
We look at the equation for $e^{-\frac{1}{4}s}\partial_{\alpha\beta}^2U+\partial_{\alpha\beta}^2\kappa$
\begin{align*}
\big(\partial_s+\frac{\partial_XU}{1-\dot{\tau}}\big)(e^{-\frac{1}{4}s}\partial_{\alpha\beta}^2U+\partial_{\alpha\beta}^2\kappa)+V\partial_X( e^{-\frac{1}{4}s}\partial_{\alpha\beta}^2U+\partial_{\alpha\beta}^2\kappa)=F,
\end{align*}
where 
\begin{align*}
    F&=\frac{e^{-\frac{5}{4}s}}{1-\dot{\tau}}H[\partial_{\alpha\beta}^2U+e^{\frac{1}{4}s}\partial_{\alpha\beta}^2\kappa]+\frac{\partial_\beta\dot{\tau}e^{-\frac{5}{4}s}}{(1-\dot{\tau})^2}H[\partial_\alpha U+e^{\frac{1}{4}s}\partial_\alpha\kappa]+\frac{e^{-\frac{5}{4}s}\partial_\alpha\dot{\tau}}{(1-\dot{\tau})^2}H[\partial_\beta U+e^{\frac{1}{4}s}\partial_\beta\kappa]
    \\
    &+\Big(\frac{e^{-\frac{5}{4}s}\partial_{\alpha\beta}^2\dot{\tau}}{(1-\dot{\tau})^2}+\frac{2e^{-\frac{5}{4}s}\partial_\alpha\dot{\tau}\partial_\beta\dot{\tau}}{(1-\dot{\tau})^3}\Big)H[U+e^{\frac{1}{4}s}\kappa]-\big(\frac{\partial_\beta\partial_XU}{1-\dot{\tau}}+\frac{\partial_\beta\dot{\tau}\partial_XU}{(1-\dot{\tau})^2}\big)(e^{-\frac{1}{4}s}\partial_\alpha U+\partial_\alpha\kappa)\\
    &-\Big(\frac{\partial_\beta U+e^{\frac{1}{4}s}\partial_\beta(\kappa-\dot{\xi})}{1-\dot{\tau}}+\frac{\partial_\beta\dot{\tau}\big(U+e^{\frac{1}{4}s}(\kappa-\dot{\xi})\big)}{(1-\dot{\tau})^2}\Big)e^{-\frac{1}{4}s}\partial_\alpha\partial_XU+\frac{\partial_XU\partial_{\alpha\beta}^2\dot{\xi}}{1-\dot{\tau}}+\frac{\partial_\beta\partial_XU\partial_\alpha\dot{\xi}}{1-\dot{\tau}}\\
    &+\frac{\partial_\beta\dot{\tau}\partial_XU\partial_\alpha\dot{\xi}}{(1-\dot{\tau})^2}-\frac{\partial_\alpha\dot{\tau}}{(1-\dot{\tau})^2}\partial_XU\big(e^{-\frac{1}{4}s}\partial_\beta U+\partial_\beta(\kappa-\dot{\xi})\big)-\frac{\partial_\alpha\dot{\tau}}{(1-\dot{\tau})^2}\partial_\beta\partial_XU(e^{-\frac{1}{4}s}U+\kappa-\dot{\xi})\\
    &-\Big(\frac{\partial_{\alpha\beta}^2\dot{\tau}}{(1-\dot{\tau})^2}+\frac{2\partial_\alpha\dot{\tau}\partial_\beta\dot{\tau}}{(1-\dot{\tau})^3}\Big)\partial_XU(e^{-\frac{1}{4}s}U+\kappa-\dot{\xi})-\frac{e^{-s}\partial_\alpha\dot{\tau}\partial_\beta\dot{\kappa}}{(1-\dot{\tau})^2}-\frac{e^{-s}\dot{\kappa}\partial_{\alpha\beta}^2\dot{\tau}}{(1-\dot{\tau})^2}-\frac{2e^{-s}\dot{\kappa}\partial_\alpha\dot{\tau}\partial_\beta\dot{\tau}}{(1-\dot{\tau})^3}.
\end{align*}
We plug in the relevant assumptions to get (the dominant term comes from the first product in the third line)
\begin{align*}
    |F|&\lesssim M^{13}\epsilon^\frac{3}{2}e^{\frac{5}{4}s}.
\end{align*}
For damping, we have
\begin{align*}
\frac{1}{1-\dot{\tau}}\partial_XU\geq -\frac{9}{8},
\end{align*}
So 
\begin{align*}
|(e^{-\frac{1}{4}s}\partial_{\alpha\beta}^2U+\partial_{\alpha\beta}^2\kappa)\circ \Phi(X_0,s)|\leq \int_{-\log\epsilon}^s CM^{13}\epsilon^\frac{3}{2}e^{\frac{5}{4}s'}e^{\frac{9}{8}(s-s')}\,ds'\\
\leq CM^{13}\epsilon^\frac{3}{2}e^{\frac{5}{4}s}\leq \frac{3}{4}M^{25}\epsilon^\frac{3}{2}e^{\frac{5}{4}s}.
\end{align*}
Hence, we have
\begin{align*}
\|\partial_{\alpha\beta}^2U+e^{\frac{1}{4}s}\partial_{\alpha\beta}^2\kappa\|_{L^\infty}\leq \frac{3}{4} M^{25}\epsilon^\frac{3}{2}e^{\frac{3}{2}s}.
\end{align*}
\end{proof}

\begin{proposition}[$\partial_{\alpha\beta}^2\partial_X^nU$, $n=1,...,6$]
We have
\begin{align*}
    \|\partial_{\alpha\beta}^2\partial_X^nU(\cdot,s)\|_{L^\infty}\leq M^{(n+5)^2}\epsilon^\frac{3}{2}e^{\frac{3}{2}s},\qquad n=1,...,6.
\end{align*}
\end{proposition}

\begin{proof}
We can bound the forcing terms
\begin{align*}
F_{U,\alpha,\beta}^{(n)}&\lesssim \big(M^{9+(n+2)^2}+M^{4+(n+3)^2}+\sum_{j=1}^{n-1}M^{(j+5)^2}+M^{(j+2)^2+(n-j+3)^2}\big)\epsilon^\frac{3}{2}e^{\frac{3}{2}s}+M^{25}\epsilon^\frac{3}{2}e^{\frac{3}{2}s}\\
&\lesssim M^{(n+5)^2-1}\epsilon^\frac{3}{2}e^{\frac{3}{2}s}.
\end{align*}
For damping, we have 
\begin{align*}
D_{U,\alpha,\beta}^{(n)}=\frac{5}{4}n-\frac{1}{4}+\frac{n+1}{1-\dot{\tau}}\partial_XU\geq -\frac{5}{4}.
\end{align*}
Notice that $\partial_{\alpha\beta}^2\partial_X^nU(X,-\log\epsilon)\equiv 0$, so
\begin{align*}
|\partial_{\alpha\beta}^2\partial_X^nU\circ\Phi(X_0,s)|&\leq \int_{-\log\epsilon}^s CM^{(n+5)^2-1}\epsilon^\frac{3}{2}e^{\frac{3}{2}s'}e^{\frac{5}{4}(s-s')}\,ds'\\
&\leq CM^{(n+5)^2-1}\epsilon^\frac{3}{2}e^{\frac{3}{2}s}\leq \frac{1}{2}M^{(n+5)^2}\epsilon^{\frac{3}{2}}e^{\frac{3}{2}s}.
\end{align*}
\end{proof}
 
\subsection{Modulation variables}
We consider equations
\begin{gather}
\begin{aligned}
e^{-s}H[\partial_{\alpha\beta}^2\partial_XU](0,s)&=\partial_{\alpha\beta}^2\dot{\tau}+e^{\frac{1}{4}s}\partial_{\alpha\beta}^2(\kappa-\dot{\xi})\partial_X^2U(0,s)+e^{\frac{1}{4}s}\partial_\alpha(\kappa-\dot{\xi})\partial_\beta\partial_X^2U(0,s)\\
&\qquad+e^{\frac{1}{4}s}\partial_\beta(\kappa-\dot{\xi})\partial_\alpha\partial_X^2U(0,s)+e^{\frac{1}{4}s}(\kappa-\dot{\xi})\partial_{\alpha\beta}^2\partial_X^2U(0,s),
\end{aligned}\\
\begin{aligned}
e^{\frac{1}{4}s}\partial_{\alpha\beta}^2(\kappa-\dot{\xi})\partial_X^5U(0,s)+e^{\frac{1}{4}s}\partial_\alpha(\kappa-\dot{\xi})&\partial_\beta\partial_X^5U(0,s)+e^{\frac{1}{4}s}\partial_\beta(\kappa-\dot{\xi})\partial_\alpha\partial_X^5U(0,s)\\+e^{\frac{1}{4}s}(\kappa-\dot{\xi})\partial_{\alpha\beta}^2\partial_X^5U(0,s)
&=e^{-s}H[\partial_{\alpha\beta}^2\partial_X^4U](0,s)-10\partial_{\alpha\beta}^2\partial_X^2U(0,s)\partial_X^3U(0,s)\\&\qquad-10\partial_\alpha\partial_X^2U(0,s)\partial_\beta\partial_X^3U(0,s)
-10\partial_\beta\partial_X^2U(0,s)\partial_\alpha\partial_X^3U(0,s)\\
&\qquad-10\partial_X^2U(0,s)\partial_{\alpha\beta}^2\partial_X^3U(0,s),
\end{aligned}\\
\begin{aligned}
e^{\frac{1}{4}s}\partial_{\alpha\beta}^2(\kappa-\dot{\xi})&=e^{-\frac{3}{4}s}\partial_{\alpha\beta}^2\dot{\kappa}-e^{-s}H[\partial_{\alpha\beta}^2U+e^{\frac{1}{4}s}\partial_{\alpha\beta}^2\kappa](0,s).
\end{aligned}
\end{gather}
We start with the second equation to estimate $\partial_{\alpha\beta}^2(\kappa-\dot{\xi})$. In the estimate of 
\begin{align*}
    \partial_\alpha\partial_X^2U\partial_\beta\partial_X^3U(0,s)+\partial_\beta\partial_X^2U(0,s)\partial_\alpha\partial_X^3U(0,s),
\end{align*}
(where the $\alpha,\,\beta$ here are place holders and can be either true $\alpha$ or $\beta$)
the worst scenario is 
\begin{align*}
    \lesssim \epsilon^\frac{3}{4}e^{\frac{3}{4}s}\cdot\epsilon^\frac{1}{2}e^{\frac{1}{2}s}\lesssim \epsilon^\frac{5}{4}e^{\frac{5}{4}s}.
\end{align*}
\begin{align*}
e^{\frac{1}{4}s}\partial_{\alpha\beta}^2(\kappa-\dot{\xi})\partial_X^5U(0,s)&\lesssim \epsilon^\frac{5}{4}e^{\frac{5}{4}s}\\
\implies \partial_{\alpha\beta}^2(\kappa-\dot{\xi})&\lesssim \epsilon^\frac{5}{4}e^s\leq \frac{1}{2}M\epsilon^\frac{5}{4}e^s.
% \leq 2e^{\frac{1}{4}s}M\epsilon^\frac{1}{2}e^{-\frac{1}{2}s}\epsilon^\frac{3}{8}e^{\frac{1}{8}s}+e^{\frac{1}{4}s}\epsilon^\frac{1}{5}e^{-s}M^{100}\epsilon^\frac{3}{2}e^{\frac{3}{2}s}+e^{-s}CM^{94\frac{7}{12}}\epsilon^\frac{3}{2}e^{\frac{3}{2}s}\\+10M^{49}\epsilon^\frac{3}{2}e^{\frac{3}{2}s}M^\square e^{-s}+10\cdot 4\epsilon^\frac{3}{4}e^{\frac{3}{4}s}\cdot 8\epsilon^\frac{1}{2}e^{\frac{1}{2}s}+10\epsilon e^{\frac{3}{4}s}\epsilon e^{\frac{1}{2}s}+10\epsilon^\frac{1}{10}e^{-\frac{3}{4}s}M^{64}\epsilon^\frac{3}{2}e^{\frac{3}{2}s}\\
% \leq 320\epsilon^\frac{5}{4}e^{\frac{5}{4}s}+\text{non dominant}\\
% e^{\frac{1}{4}s}\partial_{\alpha\beta}^2(\kappa-\dot{\xi})\leq \frac{1}{2} M\epsilon^\frac{5}{4}e^{\frac{5}{4}s}
\end{align*}

Now we consider the first equation
\begin{align*}
\partial_{\alpha\beta}^2\dot{\tau}&\leq M\epsilon^\frac{5}{4}e^{\frac{5}{4}s}\epsilon^\frac{1}{10}e^{-\frac{3}{4}s}+2e^{\frac{1}{4}s}M\epsilon^\frac{1}{2}e^{-\frac{1}{2}s}\cdot 4\epsilon^\frac{3}{4}e^{\frac{3}{4}s}+e^{\frac{1}{4}s}\epsilon^\frac{1}{5}e^{-s}M^{49}\epsilon^\frac{3}{2}e^{\frac{3}{2}s}+e^{-s}M^{60\frac{1}{4}}\epsilon^\frac{3}{2}e^{\frac{3}{2}s}\\
&\leq CM^{49}\epsilon^\frac{17}{10}e^{\frac{3}{4}s}
\leq \frac{1}{2}\epsilon e^{\frac{3}{4}s}.
\end{align*}
(this is the worst scenario where we have $\partial_{\alpha\alpha}^2$). 

Using the last equation,
\begin{align*}
|e^{-\frac{3}{4}s}\partial_{\alpha\beta}^2\dot{\kappa}|&\leq e^{\frac{1}{4}s}M\epsilon^\frac{5}{4}e^{s}+Ce^{-s}M^{25}\epsilon^\frac{3}{2}se^{\frac{3}{2}s}\\
|\partial_{\alpha\beta}^2\dot{\kappa}|&\leq \frac{3}{2}M\epsilon^\frac{5}{4}e^{2s}.
\end{align*}
Using integration (the initial value vanishes)
\begin{align*}
|\partial_{\alpha\beta}^2\kappa|&\leq \int_{-\log\epsilon}^s 2M\epsilon^\frac{5}{4}e^{2s'}\frac{e^{-s'}}{1-\dot{\tau}}\,ds'\leq \frac{5}{2}\int_{-\log\epsilon}^sM\epsilon^\frac{5}{4}e^{s'}\,ds'\\
&\leq \frac{5}{2}M\epsilon^\frac{5}{4} e^s.
\end{align*}

Finally,\begin{align*}
|\partial_{\alpha\beta}^2\dot{\xi}|\leq |\partial_{\alpha\beta}^2(\kappa-\dot{\xi})|+|\partial_{\alpha\beta}^2\kappa|\leq 4M\epsilon^\frac{5}{4}e^s.
\end{align*}

%\begin{align*}
%|(1+X^4)^\frac{1}{5}\partial_X\widetilde{U}\circ \Phi(X_0,s)|\leq (1+X_0^4)^\frac{1}{5}|\partial_X\widetilde{U}(X_0,s_0)|\exp\Big(\int_{s_0}^s4(1+X_0^4e^{\frac{4}{5}(s'-s_0)}\big)^{-\frac{1}{5}}\,ds'\Big)\\
%+C\int_{s_0}^s e^{-s'}M^\circ+e^{-s'}     
% \log(1+X^4)
%\end{align*}
%where we use the upper bound of Lagrangian trajectory

\section{Closure of bootstrap: estimates at $X=0$}\label{sec:close0}
We recall that $s_n:=-\log\epsilon+n$, the initial data \eqref{eq:initialdata}, and the iteration assumption \eqref{eq:inductassumption} and \eqref{eq:sizealphabetan}.
\subsection{Without parameter derivatives}
\begin{proposition}[$\partial_X^2U$, $\partial_X^3U$]
For all $s_n\leq s\leq s_{n+1}$, we have
\begin{align*}
|\partial_X^2U_{\alpha_n,\beta_n}(0,s)|\leq M^{13\frac{1}{2}}e^{-s_n},\qquad |\partial_X^3U_{\alpha_n,\beta_n}(0,s)\leq M^{22}e^{-s_n}.
\end{align*}
We also have, for all $(\alpha,\beta)\in B$
\begin{align*}
    |\partial_X^2U_{\alpha,\beta}(0,s)|\leq \epsilon^\frac{1}{10}e^{-\frac{3}{4}s},\qquad |\partial_X^3U_{\alpha,\beta}(0,s)|\leq M^{27}e^{-s}.
\end{align*}
\end{proposition}
\begin{proof}
Consider the following equations
\begin{equation}
\begin{cases}
\begin{aligned}
&\big(\partial_s+\frac{9}{4}-\frac{3}{1-\dot{\tau}}\big)\partial_X^2U_{\alpha,\beta}(0,s)=-\frac{e^{\frac{s}{4}}}{1-\dot{\tau}}(\kappa-\dot{\xi})\partial_X^3U_{\alpha,\beta}(0,s)+\frac{e^{-s}}{1-\dot{\tau}}H[\partial_X^2U_{\alpha,\beta}](0,s),\\
&\big(\partial_s+\frac{7}{2}-\frac{4}{1-\dot{\tau}}\big)\partial_X^3U_{\alpha,\beta}(0,s)=\frac{e^{-s}}{1-\dot{\tau}}H[\partial_X^3U_{\alpha,\beta}](0,s)-\frac{3}{1-\dot{\tau}}\partial_X^2U_{\alpha,\beta}(0,s)^2.
\end{aligned}
\end{cases}
\end{equation}
This is a non-linear system. We rewrite the system as 
\begin{equation}
\begin{cases}
\begin{aligned}
&(\partial_s-\frac{3}{4})\partial_X^2U_{\alpha,\beta}(0,s)=\frac{3\dot{\tau}}{1-\dot{\tau}}\partial_X^2U_{\alpha,\beta}(0,s)-\frac{e^{\frac{1}{4}s}(\kappa-\dot{\xi})}{1-\dot{\tau}}\partial_X^3U_{\alpha,\beta}(0,s)+\frac{e^{-s}}{1-\dot{\tau}}H[\partial_X^2U_{\alpha,\beta}](0,s),\\
&(\partial_s-\frac{1}{2})\partial_X^3U_{\alpha,\beta}(0,s)=\frac{4\dot{\tau}}{1-\dot{\tau}}\partial_X^3U_{\alpha,\beta}(0,s)+\frac{e^{-s}}{1-\dot{\tau}}H[\partial_X^3U_{\alpha,\beta}](0,s)-\frac{3}{1-\dot{\tau}}\partial_X^2U_{\alpha,\beta}(0,s)^2.
\end{aligned}
\end{cases}\label{eq:systemx=0}
\end{equation}
The reason to do this is that $\dot{\tau}$ decays in $s$, and now damping is constant in time: damping for $\partial_X^2U_{\alpha,\beta}(0,s)$ is $-\frac{3}{4}$, for $\partial_X^3U_{\alpha,\beta}(0,s)$ is $-\frac{1}{2}$\footnote{The damping constants determines the decay rates of $\partial_X^2U(0,s),\,\partial_X^3U(0,s)$.}.

For forcing,
\begin{align*}
F_{U,0}^{(2)}&=\frac{3\dot{\tau}}{1-\dot{\tau}}\partial_X^2U_{\alpha,\beta}(0,s)-\frac{e^{\frac{1}{4}s}}{1-\dot{\tau}}(\kappa-\dot{\xi})\partial_X^3U_{\alpha,\beta}(0,s)+\frac{e^{-s}}{1-\dot{\tau}}H[\partial_X^2U_{\alpha,\beta}](0,s)\\
&\lesssim \epsilon^\frac{1}{5}e^{-\frac{3}{4}s}\epsilon^\frac{1}{10}e^{-\frac{3}{4}s}+\epsilon^\frac{1}{5}e^{-\frac{3}{4}s}M^{27} e^{-s}+e^{-s}M^{12\frac{3}{8}}\lesssim M^{13\frac{1}{8}}e^{-s},\\
F^{(3)}_{U,0}&=\frac{4\dot{\tau}}{1-\dot{\tau}}\partial_X^3U_{\alpha,\beta}(0,s)+\frac{e^{-s}}{1-\dot{\tau}}H[\partial_X^3U_{\alpha,\beta}](0,s)-\frac{3}{1-\dot{\tau}}\partial_X^2U_{\alpha,\beta}(0,s)^2\\
&\lesssim \epsilon^\frac{1}{5}e^{-\frac{3}{4}s}M^{27} e^{-s}+e^{-s}M^{21\frac{7}{8}}+\epsilon^\frac{1}{5}e^{-\frac{3}{2}s}\lesssim M^{21\frac{7}{8}}e^{-s}.
\end{align*}
For any $s_n\leq s\leq s_{n+1}=s_n+1$, we integrate from $s_n$ (note the initial value is 0)
\begin{align*}
\partial_X^2U_{\alpha_n,\beta_n}(0,s)&=\partial_X^2U_{\alpha_n,\beta_n}(0,s_n)e^{\frac{3}{4}(s-s_n)}+\int_{s_n}^se^{\frac{3}{4}(s-s')}F_{U,0}^{(2)}(s')\,ds',\\
|\partial_X^2U_{\alpha_n,\beta_n}(0,s)|&\lesssim M^{13\frac{1}{8}}e^{-s_n}\lesssim M^{13\frac{1}{8}}e^{-s}.
\end{align*}
The constant is independent of $n$. Similarly,
\begin{align*}
|\partial_X^3U_{\alpha_n,\beta_n}(0,s)|\lesssim M^{21\frac{7}{8}}e^{-s}\qquad\text{for }s_n\leq s\leq s_{n+1}.
\end{align*}
Hence, by raising the $M$ exponent, we can make them 
\begin{align*}
\leq \frac{1}{2}M^{13\frac{1}{2}}e^{-s},\qquad \leq \frac{1}{2}M^{22}e^{-s}
\end{align*}
respectively.

By the mean value theorem, for $s_n\leq s\leq s_{n+1}$,
\begin{align*}
|\partial_X^2U_{\alpha,\beta}(0,s)|&\leq |\partial_X^2U_{\alpha_n,\beta_n}(0,s)|+|\alpha-\alpha_n|\sup_{\alpha\in\mathcal{B}_n}|\partial_\alpha\partial_X^2U(0,s)|+|\beta-\beta_n|\sup_{\beta\in\mathcal{B}_n}|\partial_\beta\partial_X^2U(0,s)|\\
&\leq M^{13\frac{1}{2}}e^{-s}+(M^{15}\epsilon^{-\frac{3}{4}}e^{-\frac{7}{4}s_n}+\epsilon^{-\frac{3}{10}}e^{-\frac{3}{2}s_n})4\epsilon^\frac{3}{4}e^{\frac{3}{4}s}+M^{25}\epsilon^{-\frac{1}{2}}e^{-\frac{3}{2}s_n}\epsilon e^{\frac{3}{4}s}\\
&\lesssim M^{13\frac{1}{2}}e^{-s}+{4}M^{15}e^{-s}+\epsilon^\frac{9}{20}e^{-\frac{3}{4}s}+M^{25}\epsilon^\frac{1}{2}e^{-\frac{3}{4}s}\leq \frac{3}{4}\epsilon^\frac{1}{10}e^{-\frac{3}{4}s},\\
|\partial_X^3U_{\alpha,\beta}(0,s)|&\leq |\partial_X^3U_{\alpha_n,\beta_n}(0,s)|+|\alpha-\alpha_n|\sup_{\alpha\in\mathcal{B}_n}|\partial_\alpha\partial_X^3U(0,s)|+|\beta-\beta_n|\sup_{\beta\in\mathcal{B}_n}|\partial_\beta\partial_X^3U(0,s)|\\
&\leq M^{22}e^{-s}+(M^{15}\epsilon^{-\frac{3}{4}}e^{-\frac{7}{4}s_n}+\epsilon^{-\frac{3}{10}}e^{-\frac{3}{2}s_n})\epsilon e^{\frac{1}{2}s}+M^{25}\epsilon^{-\frac{1}{2}}e^{-\frac{3}{2}s_n}8\epsilon^\frac{1}{2}e^{\frac{1}{2}s}\\
&\lesssim M^{21}e^{-s}+\epsilon^\frac{1}{4}e^{-\frac{5}{4}s}+\epsilon^\frac{7}{10}e^{-s}+M^{25}e^{-s}\leq \frac{3}{4}M^{27}e^{-s}.
\end{align*}
\end{proof}

\begin{proposition}[$\partial_X^5\widetilde{U}$]\label{prop:5xtileU}
For all $(\alpha,\beta)\in B$ and all $s\geq -\log\epsilon$, we have
\begin{align*}
    |\partial_X^5\widetilde{U}(0,s)|\leq \epsilon^\frac{1}{2}.
\end{align*}
\end{proposition}
\begin{proof}
We plug in $X=0$ to \eqref{eq:eqnxtildeU} with $n=5$, using that fact that $U_2^{(j)}(0)=0$ when $j=0,2,3,4,6$, $U_2^{(5)}(0)=120$, and the constraints that $\partial_X^j\widetilde{U}(0,s)=0$ when $j=0,1,4$, to get
\begin{align*}
\partial_s\partial_X^5\widetilde{U}(0,s)&=\frac{6\dot{\tau}}{1-\dot{\tau}}\partial_X^5\widetilde{U}(0,s)+\frac{e^{-s}}{1-\dot{\tau}}H[\partial_X^5U](0,s)+\frac{720\dot{\tau}}{1-\dot{\tau}}\\
&\qquad-\frac{e^{\frac{1}{4}s}(\kappa-\dot{\xi})}{1-\dot{\tau}}\partial_X^6\widetilde{U}(0,s)-\frac{10}{1-\dot{\tau}}(\partial_X^3\widetilde{U})^2(0,s).\numberthis\label{eq:eq5xtildex=0}
\end{align*}
We have
\begin{align*}
\text{RHS}\lesssim \epsilon^\frac{1}{5}e^{-\frac{3}{4}s}l\epsilon^\frac{1}{5}+M^{39\frac{3}{8}}e^{-s}+\epsilon^\frac{1}{5}e^{-\frac{3}{4}s}+e^{-\frac{3}{4}s}\epsilon^\frac{1}{5}+M^{54}e^{-2s}\lesssim \epsilon^\frac{1}{5}e^{-\frac{3}{4}s}.
\end{align*}
So once we integrate,
\begin{align*}
|\partial_X^5\widetilde{U}(0,s)|&\leq|\partial_X^5\widetilde{U}(0,-\log\epsilon)|+\int_{-\log\epsilon}^s C\epsilon^\frac{1}{5}e^{-\frac{3}{4}s'}\,ds'\\
&\leq |\partial_X^5\widehat{U}(0)|+C\epsilon^{\frac{1}{5}+\frac{3}{4}}\leq \epsilon+C\epsilon^{\frac{19}{20}}\leq
\frac{1}{2}\epsilon^\frac{1}{2}.
\end{align*}
\end{proof}

\subsection{First order parameter derivatives}
\begin{proposition}[$\partial_\alpha\partial_X^2U,\,\partial_\beta\partial_X^2U,\,\partial_\alpha\partial_X^3U,\,\partial_\beta\partial_X^3U$]
For all $(\alpha,\beta)\in B$, we have
\begin{align*}
    \epsilon^\frac{3}{4}e^{\frac{3}{4}s}\leq \partial_\alpha\partial_X^2U(0,s)\leq 4\epsilon^\frac{3}{4}e^{\frac{3}{4}s},\quad |\partial_\alpha\partial_X^3U(0,s)|\leq \epsilon e^{\frac{1}{2}s},\\
    |\partial_\beta\partial_X^2U(0,s)|\leq \epsilon e^{\frac{3}{4}s},\quad 4\epsilon^\frac{1}{2}e^{\frac{1}{2}s}\leq \partial_\beta\partial_X^3U(0,s)\leq 8\epsilon^\frac{1}{2}e^{\frac{1}{2}s}.
\end{align*}
\end{proposition}
\begin{proof}
Taking $\partial_\alpha$ of \eqref{eq:systemx=0} we have
\begin{align*}
(\partial_s-\frac{3}{4})\partial_\alpha\partial_X^2U(0,s)=\partial_\alpha F_{U,0}^{(2)}(0,s),\\
(\partial_s-\frac{1}{2})\partial_\alpha\partial_X^3U(0,s)=\partial_\alpha F_{U,0}^{(3)}(0,s),
\end{align*}
where
\begin{align*}
\partial_\alpha F_{U,0}^{(2)}&=\frac{3\dot{\tau}}{1-\dot{\tau}}\partial_\alpha\partial_X^2U_{\alpha,\beta}(0,s)+\frac{3\partial_\alpha\dot{\tau}}{(1-\dot{\tau})^2}\partial_X^2U_{\alpha,\beta}(0,s)-\frac{e^{\frac{1}{4}s}}{1-\dot{\tau}}(\kappa-\dot{\xi})\partial_\alpha\partial_X^3U_{\alpha,\beta}(0,s)\\
&\qquad-\frac{e^{\frac{1}{4}s}}{1-\dot{\tau}}\partial_\alpha(\kappa-\dot{\xi})\partial_X^3U_{\alpha,\beta}(0,s)-\frac{e^{\frac{1}{4}s}\partial_\alpha\dot{\tau}}{(1-\dot{\tau})^2}(\kappa-\dot{\xi})\partial_X^3U_{\alpha\beta}(0,s)\\
&\qquad+\frac{e^{-s}}{1-\dot{\tau}}H[\partial_\alpha\partial_X^2U_{\alpha,\beta}](0,s)+\frac{e^{-s}\partial_\alpha\dot{\tau}}{(1-\dot{\tau})^2}H[\partial_X^2U_{\alpha,\beta}](0,s),\\
\partial_\alpha F_{U,0}^{(3)}(0,s)&=\frac{4\dot{\tau}}{1-\dot{\tau}}\partial_\alpha\partial_X^3U_{\alpha,\beta}(0,s)+\frac{4\partial_\alpha\dot{\tau}}{(1-\dot{\tau})^2}\partial_X^3U_{\alpha,\beta}(0,s)+\frac{e^{-s}}{1-\dot{\tau}}H[\partial_\alpha\partial_X^3U_{\alpha,\beta}](0,s)\\
&\qquad+\frac{e^{-s}\partial_\alpha\dot{\tau}}{(1-\dot{\tau})^2}H[\partial_X^3U_{\alpha,\beta}](0,s)-\frac{6}{1-\dot{\tau}}\partial_X^2U_{\alpha,\beta}(0,s)\partial_\alpha\partial_X^2U_{\alpha,\beta}(0,s)\\
&\qquad-\frac{3\partial_\alpha\dot{\tau}}{(1-\dot{\tau})^2}\partial_X^2U_{\alpha,\beta}(0,s)^2.
\end{align*}
Due to our choice of initial data,
\begin{align*}
\partial_\alpha\partial_X^2U(0,-\log\epsilon)=2,\quad \partial_\alpha\partial_X^3U(0,-\log\epsilon)=0,\\
\partial_\beta\partial_X^2U(0,-\log\epsilon)=0,\quad\partial_\beta\partial_X^3U(0,-\log\epsilon)=6,
\end{align*}
we have
\begin{align*}
\partial_\alpha\partial_X^2U(0,s)&=2e^{\frac{3}{4}(s+\log\epsilon)}+\int_{-\log\epsilon}^se^{\frac{3}{4}(s-s')}\partial_\alpha F^{(2)}_{U,0}(0,s')\,ds',\\
\partial_\beta\partial_X^2U(0,s)&=\int_{-\log\epsilon}^se^{\frac{3}{4}(s-s')}\partial_\beta F^{(2)}_{U,0}(0,s')\,ds',\\
\partial_\alpha\partial_X^3U(0,s)&=\int_{-\log\epsilon}^se^{\frac{1}{2}(s-s')}\partial_\alpha F^{(3)}_{U,0}(0,s')\,ds',\\
\partial_\alpha\partial_X^3U(0,s)&=6e^{\frac{1}{2}(s+\log\epsilon)}+\int_{-\log\epsilon}^se^{\frac{1}{2}(s-s')}\partial_\beta F^{(3)}_{U,0}(0,s')\,ds'.
\end{align*}
And we have 
\begin{align*}
|\partial_\alpha F_{U,0}^{(2)}|&\lesssim \epsilon^\frac{1}{5}e^{-\frac{3}{4}s}l^\frac{1}{2}M\epsilon^\frac{3}{4}e^{\frac{3}{4}s}+\epsilon^\frac{1}{2}\epsilon^\frac{1}{10}e^{-\frac{3}{4}s}+e^{\frac{1}{4}s}\epsilon^\frac{1}{5}e^{-s}l^\frac{1}{2}M\epsilon^\frac{3}{4}e^{\frac{3}{4}s}\\
&\qquad+e^{-s}\epsilon^\frac{1}{2}M^{13\frac{1}{8}}\\
&\lesssim Ml^\frac{1}{2}\epsilon^{\frac{3}{4}+\frac{1}{5}}+M^{29\frac{1}{2}}\epsilon^\frac{3}{4}e^{-\frac{1}{4}s}.
\end{align*}
This also holds for $\partial_\beta F^{(2)}_{U,0}$. Hence,
\begin{align*}
\Big|\int_{-\log\epsilon}^se^{\frac{3}{4}(s-s')}\partial_\alpha F_{U,0}^{(2)}(0,s')\,ds'\Big|&\leq CMl^\frac{1}{2}\epsilon^\frac{19}{20}\epsilon^\frac{3}{4}e^{\frac{3}{4}s},\\
\implies\qquad|\partial_\beta\partial_X^2U(0,s)|\leq \frac{1}{2}\epsilon e^{\frac{3}{4}s},\quad
|\partial_\alpha\partial_X^2U(0,s)&-2\epsilon^\frac{3}{4}e^{\frac{3}{4}s}|\leq \frac{1}{2}\epsilon^\frac{3}{4}e^{\frac{3}{4}s}.
\end{align*}

For the third derivatives,
\begin{align*}
|\partial_\alpha F_{U,0}^{(3)}(0,s)|&\lesssim \epsilon^\frac{1}{5}e^{-\frac{3}{4}s}\epsilon^\frac{1}{2}e^{\frac{1}{2}s}+\epsilon^\frac{1}{2}M^{27} e^{-s}+e^{-s}M^{40\frac{1}{2}}\epsilon^\frac{3}{4}e^{\frac{3}{4}s}+e^{-s}\epsilon^\frac{1}{2}M^{21\frac{7}{8}}+\epsilon^\frac{1}{10}e^{-\frac{3}{4}s}\epsilon^\frac{3}{4}e^{\frac{3}{4}s}\\
&\qquad+\epsilon^\frac{1}{2}\epsilon^\frac{1}{5}e^{-\frac{3}{2}s}\lesssim\epsilon^\frac{17}{20}.
\end{align*}
The same holds for $\partial_\beta F_{U,0}^{(3)}$. Hence
\begin{gather*}
\Big|\int_{-\log\epsilon}^se^{\frac{1}{2}(s-s')}\partial_\alpha F_{U,0}^{(3)}(0,s')\,ds'\Big|\leq C\epsilon^{\frac{17}{20}+\frac{1}{2}}e^{\frac{1}{2}s}\leq \frac{1}{2}\epsilon e^{\frac{1}{2}s}\\
|\partial_\alpha\partial_X^3U(0,s)|\leq \frac{1}{2}\epsilon e^{\frac{1}{2}s},\quad |\partial_\beta\partial_X^3U(0,s)-6\epsilon^\frac{1}{2}e^{\frac{1}{2}s}|\leq \frac{1}{2}\epsilon^\frac{1}{2}e^{\frac{1}{2}s}.
\end{gather*}
\end{proof}

\begin{proposition}[$\partial_\alpha\partial_X^5U$]
We have (same holds for $\beta$)
\begin{align*}
    |\partial_\alpha\partial_X^5U(0,s)|\leq \epsilon^\frac{3}{8}e^{\frac{1}{8}s}.
\end{align*}
\end{proposition}
\begin{proof}
We take $\partial_\alpha$ of \eqref{eq:eq5xtildex=0} (note that $\partial_\alpha\partial_X^5U=\partial_\alpha\partial_X^5\widetilde{U}$), and we get
\begin{align*}
    \partial_s\partial_\alpha\partial_X^5U(0,s)&=\frac{6\dot{\tau}}{1-\dot{\tau}}\partial_\alpha\partial_X^5\widetilde{U}(0,s)+\frac{6\partial_\alpha\dot{\tau}}{(1-\dot{\tau})^2}\partial_X^5\widetilde{U}(0,s)+\frac{e^{-s}}{1-\dot{\tau}}H[\partial_\alpha\partial_X^5U](0,s)\\
    &\qquad+\frac{e^{-s}\partial_\alpha\dot{\tau}}{(1-\dot{\tau})^2}H[\partial_X^5U](0,s)+\frac{720\partial_\alpha\dot{\tau}}{(1-\dot{\tau})^2}-\frac{e^{\frac{1}{4}s(\kappa-\dot{\xi})}}{1-\dot{\tau}}\partial_\alpha\partial_X^6U(0,s)\\
    &\qquad+\Big(\frac{e^{\frac{1}{4}s}\partial_\alpha(\kappa-\dot{\xi})}{1-\dot{\tau}}+\frac{e^{\frac{1}{4}s}\partial_\alpha\dot{\tau}(\kappa-\dot{\xi})}{(1-\dot{\tau})^2}\Big)\partial_X^6\widetilde{U}(0,s)\\
    &\qquad-\frac{40}{1-\dot{\tau}}\partial_\alpha\partial_X^3U(0,s)\partial_X^3\widetilde{U}(0,s)-\frac{10\partial_\alpha\dot{\tau}}{(1-\dot{\tau})^2}(\partial_X^3\widetilde{U})^2.
\end{align*}
The right hand side can be bounded by
\begin{align*}
    \mathrm{RHS}\lesssim Ml^\frac{1}{2}\epsilon^{\frac{1}{5}+\frac{3}{4}}+\epsilon^\frac{1}{2}.
\end{align*}
Since $\partial_\alpha\partial_X^5U(0,-\log\epsilon)=0$, after integration, we have
\begin{align*}
    |\partial_\alpha\partial_X^5U(0,s)|&\leq \int_{-\log\epsilon}^s (Ml^\frac{1}{2}\epsilon^{\frac{1}{5}+\frac{3}{4}}+\epsilon^\frac{1}{2})\,ds'\\
    &\leq (Ml^\frac{1}{2}\epsilon^{\frac{1}{5}+\frac{3}{4}}+\epsilon^\frac{1}{2})s\leq \frac{1}{2}\epsilon^\frac{3}{8}e^{\frac{1}{8}s}.
\end{align*}
\end{proof}

\section{Finding $\alpha,\,\beta$ by iterations}\label{sec:findpara}
Let $s_n=-\log\epsilon+n$ for $n=0,1,2,...$. We start with $$\alpha_0=-\frac{1}{2}\widehat{U}_0''(0),\qquad\beta_0=-\frac{1}{6}\widehat{U}_0^{(3)}(0),$$ so that
\begin{align*}
    \partial_X^2U(0,s_0)=0,\qquad \partial_X^3U(0,s_0)=0.
\end{align*}
Suppose we can find $(\alpha_n,\beta_n)$ such that 
\begin{align*}
\partial_X^2U_{\alpha_n,\beta_n}(0,s_n)=\partial_X^3U_{\alpha_n,\beta_n}(0,s_n)=0.
\end{align*}
% where $U_{\alpha,\beta}$ denotes the solution with initial data 
% \begin{align*}
% U(X,-\log\epsilon)=U_2(X)\chi(2\epsilon^\frac{5}{4}X)+\widehat{U}_0(X)+\chi(X)(\alpha X^2+\beta X^3)
% \end{align*}
We want to find $(\alpha_{n+1},\beta_{n+1})\in B_n(\alpha_n,\beta_n)$ defined in \eqref{eq:sizealphabetan} such that at time $s_{n+1}=s_n+1$,
\begin{align*}
\partial_X^2U_{\alpha_{n+1},\beta_{n+1}}(0,s_{n+1})=\partial_X^3U_{\alpha_{n+1},\beta_{n+1}}(0,s_{n+1})=0.
\end{align*}
Define the map $T_n:B_n(\alpha_n,\beta_n)\subset\mathbb{R}^2\to\mathbb{R}^2$ by
\begin{align*}
T_n(\alpha,\beta):=\Big(\partial_X^2U_{\alpha,\beta}(0,s_{n+1}),\partial_X^3U_{\alpha,\beta}(0,s_{n+1})\Big).
\end{align*}
We want to find $(\alpha_{n+1},\beta_{n+1})\in B_n(\alpha_n,\beta_n)$ such that $T_n(\alpha_{n+1},\beta_{n+1})=0$. To this end, we do first order Taylor expansion of $T_n$ near $(\alpha_n,\beta_n)$
\begin{align}
\begin{aligned}
\partial_X^2U_{\alpha_{n+1},\beta_{n+1}}(0,s_{n+1})&=\partial_X^2U_{\alpha_n,\beta_n}(0,s_{n+1})
+\partial_\alpha\partial_X^2U(0,s_{n+1})\Big|_{\alpha_*,\beta_*}(\alpha_{n+1}-\alpha_n)\\
&\qquad+\partial_\beta\partial_X^2U(0,s_{n+1})\Big|_{\alpha_*,\beta_*}(\beta_{n+1}-\beta_n),
\end{aligned}\label{eq:inductionTayloralpha}\\
\begin{aligned}
\partial_X^3U_{\alpha_{n+1},\beta_{n+1}}(0,s_{n+1})&=\partial_X^3U_{\alpha_n,\beta_n}(0,s_{n+1})
+\partial_\alpha\partial_X^3U(0,s_{n+1})\Big|_{\alpha_*,\beta_*}(\alpha_{n+1}-\alpha_n)\\
&\qquad+\partial_\beta\partial_X^3U(0,s_{n+1})\Big|_{\alpha_*,\beta_*}(\beta_{n+1}-\beta_n),
\end{aligned}\label{eq:inductionTaylorbeta}
\end{align}
for some $(\alpha_*,\beta_*)$ such that 
\begin{align*}
|(\alpha_*,\beta_*)-(\alpha_n,\beta_n)|\leq |(\alpha_{n+1},\beta_{n+1})-(\alpha_n,\beta_n)|.
\end{align*}
Since 
\begin{align*}
\partial_X^2U_{\alpha_{n+1},\beta_{n+1}}(0,s_{n+1})=0,\quad \partial_X^3U_{\alpha_{n+1},\beta_{n+1}}(0,s_{n+1})=0,
\end{align*}
we can view $\alpha_{n+1},\beta_{n+1}$ as roots to this system of equations.
% \begin{equation}
% \begin{cases}
% \begin{aligned}
% 0=\partial_X^2U_{\alpha_n,\beta_n}(0,s_{n+1})
% +\partial_\alpha\partial_X^2U(0,s_{n+1})\Big|_{\alpha_n,\beta_n}(\alpha_{n+1}-\alpha_n)+\cdots\\
% 0=\partial_X^3U_{\alpha_n,\beta_n}(0,s_{n+1})
% +\partial_\alpha\partial_X^3U(0,s_{n+1})\Big|_{\alpha_n,\beta_n}(\alpha_{n+1}-\alpha_n)+\cdots
% \end{aligned}
% \end{cases}
% \end{equation}
We denote the system as 
\begin{align*}
(F_1,F_2)(\alpha_n,\beta_n,\alpha,\beta)=F(\alpha_n,\beta_n,\alpha,\beta)
\end{align*}
for convenience. We use Newton's iteration method:
\begin{enumerate}
    \item The $F_1,\,F_2$ are bounded: because $\partial_X^2U_{\alpha_n,\beta_n}(0,s_{n+1})$, $\partial_X^3U_{\alpha_n,\beta_n}(0,s_{n+1})$ are bounded due to \eqref{eq:finerinductassumptioin}.
    \item The matrix consisting of partial derivatives of $F_1,\,F_2$ with respect to $\alpha,\,\beta$ evaluated at $(\alpha_n,\beta_n)$, i.e.\ $D_{\alpha_{n+1},\beta_{n+1}}F(\alpha_n,\beta_n)$, has non-zero determinant for all $(\alpha,\beta)\in B_n(\alpha_n,\beta_n)$: by \eqref{eq:assumptionjacobian},
    \begin{align*}
\mathrm{det} D_{\alpha,\beta} F(\alpha_n,\beta_n)&=\partial_\alpha\partial_X^2U_{\alpha_n,\beta_n}(0,s_{n+1})\partial_\beta\partial_X^3U_{\alpha_n,\beta_n}(0,s_{n+1})\\
&\qquad-\partial_\beta\partial_X^2U_{\alpha_n,\beta_n}(0,s_{n+1})\partial_\alpha\partial_X^3U_{\alpha_n,\beta_n}(0,s_{n+1})\\
&\geq \epsilon^\frac{3}{4}e^{\frac{3}{4}s_{n+1}}\cdot 4\epsilon^\frac{1}{2}e^{\frac{1}{2}s_{n+1}}-\epsilon e^{\frac{1}{2}s_{n+1}}\epsilon e^{\frac{3}{4}s_{n+1}}\\
&=4\epsilon^\frac{5}{4}e^{\frac{5}{4}s_{n+1}}-\epsilon^2e^{\frac{5}{4}s_{n+1}}>3\epsilon^\frac{5}{4}e^{\frac{5}{4}s_{n+1}}.
\end{align*}
\item The Hessian $\nabla^2F(\alpha,\beta)$ is bounded: since \eqref{eq:assumption2paranx}, every entry of $\nabla^2F(\alpha,\beta)(s)$ is $\leq M^{64}\epsilon^\frac{3}{2}e^{\frac{3}{2}s_{n+1}}$ for $s_n\leq s\leq s_{n+1}$. 
\end{enumerate}
Hence Newton's method guarantees a quadratic convergence\footnote{In one variable setting, let $r$ be the root of $f$, we search in the ``correct" interval $I$, $x_n$ be the approximation to $r$ after $n$-th iteration, then
\begin{align*}
|r-x_{n+1}|\leq \frac{1}{2}\sup_{x\in I}\frac{f''(x)}{f'(x)}|r-x_n|^2
\end{align*}}.

For the size of $B$, since $\alpha_0=-\frac{1}{2}\widehat{U}_0''(0)$, $\beta_0=-\frac{1}{6}\widehat{U}_0^{(3)}(0)$, and the size of $B_n$ for each $n\geq 0$, we have the uniform bounds of $\alpha_n,\,\beta_n$ for all $n\geq 0$:
\begin{align*}
    |\alpha_n|&\leq \frac{1}{2}\epsilon+ \sum_{n=0}^\infty M^{15}\epsilon^{-\frac{3}{4}}e^{-\frac{7}{4}(-\log\epsilon+n)}+\sum_{n=0}^{\infty}\epsilon^{-\frac{3}{10}}e^{-\frac{3}{2}(-\log\epsilon + n)}\\
    &=\frac{1}{2}\epsilon+ \frac{M^{15}\epsilon}{1-e^{-\frac{7}{4}}}+\frac{\epsilon^\frac{6}{5}}{1-e^{-\frac{3}{2}}}\leq 1.24 M^{15}\epsilon,\\
    |\beta_n|&\leq \frac{1}{6}\epsilon+ \sum_{n=0}^\infty M^{25}\epsilon^{-\frac{1}{2}}e^{-\frac{3}{2}(-\log\epsilon+n)}=\frac{1}{6}\epsilon+\frac{M^{25}\epsilon}{1-e^{-\frac{3}{2}}}\leq  1.4 M^{25}\epsilon.
\end{align*}
Hence, we have established \eqref{eq:sizealphabeta}.

For the size of $B_n$, recall the iteration step\footnote{Also from Newton's method, we see that the ball $B_n$ cannot be too small, because if we initialize at $(\alpha_n,\ \beta_n)$, after the first iteration we can be as far as roughly
\begin{align*}
|\alpha^1-\alpha_n|= \frac{\partial_X^2U_{\alpha_n,\beta_n}(0,s_{n+1})}{\partial_\alpha\partial_X^2U(0,s)}\sim\epsilon^\frac{1}{10}e^{-\frac{3}{4}s_{n+1}}\cdot\epsilon^{-\frac{3}{4}}e^{-\frac{3}{4}s_{n+1}}=\epsilon^{-\frac{13}{20}}e^{-\frac{3}{2}s_{n+1}}
\end{align*}
for $\beta$ this will be $M^{27} \epsilon^{-\frac{1}{2}}
e^{-\frac{3}{2}s_{n+1}}$, so we see $e^{-\frac{3}{2}s_{n+1}}$ is the smallest size we can consider.}, hence the root exists. Recall the Taylor expansions \eqref{eq:inductionTayloralpha} and 
\eqref{eq:inductionTaylorbeta}, 
% When evaluated at $(\alpha_{n+1},\beta_{n+1})$, the value is 0, so
% \begin{align*}
% 0&=\partial_X^2U_{\alpha_n,\beta_n}(0,s_{n+1})+\partial_\alpha\partial_X^2U(0,s_{n+1})\Big|_{\alpha_n,\beta_n}(\alpha_{n+1}-\alpha_n)+\partial_\beta\partial_X^2U(0,s_{n+1})\Big|_{\alpha_n,\beta_n}(\beta_{n+1}-\beta_n)\\
% &\qquad+\frac{1}{2}\partial_{\alpha\alpha}^2\partial_X^2U(0,s_{n+1})\Big|_{\alpha_*,\beta_*}(\alpha_{n+1}-\alpha_n)^2+\partial_{\alpha\beta}^2\partial_X^2U(0,s_{n+1})\Big|_{\alpha_*,\beta_*}(\alpha_{n+1}-\alpha_n)(\beta_{n+1}-\beta_n)\\
% &\qquad+\frac{1}{2}\partial_{\beta\beta}^2\partial_X^2U(0,s_{n+1})\Big|_{\alpha_*,\beta_*}(\beta_{n+1}-\beta_n)^2.
% \end{align*}
% Similarly, for $\partial_X^3U$, we have:
% \begin{align*}
% 0&=\partial_X^3U_{\alpha_n,\beta_n}(0,s_{n+1})+\partial_\beta\partial_X^3U(0,s_{n+1})\Big|_{\alpha_n,\beta_n}(\beta_{n+1}-\beta_n)+\partial_\alpha\partial_X^3U(0,s_{n+1})\Big|_{\alpha_n,\beta_n}(\alpha_{n+1}-\alpha_n)\\
% &\qquad+\frac{1}{2}(\alpha_{n+1}-\alpha_n,\beta_{n+1}-\beta_n)^T\nabla^2_{\alpha,\beta}\partial_X^3U(0,s_{n+1})\Big|_{\alpha_*,\beta_*}(\alpha_{n+1}-\alpha_n,\beta_{n+1}-\beta_n).
% \end{align*}
% So
\begin{align*}
|\alpha_{n+1}-\alpha_n|&\leq \epsilon^{-\frac{3}{4}}e^{-\frac{3}{4}s_{n+1}}\big(M^{13\frac{1}{2}}e^{-s_n}+\epsilon e^{\frac{3}{4}s_{n+1}}|\beta_{n+1}-\beta_n|\big)\\
% +\frac{1}{2}M^{49}\epsilon^\frac{3}{2}e^{\frac{3}{2}s_{n+1}}(|\alpha_{n+1}-\alpha_n|^2+2|\alpha_{n+1}-\alpha_n||\beta_{n+1}-\beta_n|+|\beta_{n+1}-\beta_n|^2)\big]\\
&\leq eM^{13\frac{1}{2}}\epsilon^{-\frac{3}{4}}e^{-\frac{7}{4}s_{n+1}}+M^{25}\epsilon^{-\frac{1}{4}}e^{-\frac{3}{2}s_n}\\
% +\frac{1}{2}M^{49}\epsilon^\frac{3}{4}e^{\frac{3}{4}s_{n+1}}\big[M^{30}\epsilon^{-\frac{3}{2}}e^{-\frac{7}{2}s_n}
% +2M^{15}\epsilon^{-\frac{21}{20}}e^{-\frac{13}{4}s_n}+
% \\\epsilon^{-\frac{3}{5}}e^{-3s_n}+2(M^{15}\epsilon^{-\frac{3}{4}}e^{-\frac{7}{4}s_n}+\epsilon^{-\frac{3}{10}}e^{-\frac{3}{2}s_n})M^{25}\epsilon^{-\frac{1}{2}}e^{-\frac{3}{2}s_n}
% +M^{50}\epsilon^{-1}e^{-3s_n}\big]\\
&\leq \frac{3}{4}(M^{15}\epsilon^{-\frac{3}{4}}e^{-\frac{7}{4}s_n}+\epsilon^{-\frac{3}{10}}e^{-\frac{3}{2}s_n}),\quad(s_n=-\log\epsilon+n)\\
|\beta_{n+1}-\beta_n|&\leq \frac{1}{4}\epsilon^{-\frac{1}{2}}e^{-\frac{1}{2}s_{n+1}}\big(M^{22}e^{-s_n}+\epsilon e^{\frac{1}{2}s_{n+1}}(M^{15}\epsilon^{-\frac{3}{4}}e^{-\frac{7}{4}s_n}+\epsilon^{-\frac{3}{10}}e^{-\frac{3}{2}s_n})\big)\\
% +\frac{1}{2}M^{64}\epsilon^\frac{3}{2}e^{\frac{3}{2}s_{n+1}}\big((M^{15}\epsilon^{-\frac{3}{4}}e^{-\frac{7}{4}s_n}+\epsilon^{-\frac{3}{10}}e^{-\frac{3}{2}s_n})^2\\
% +2(M^{15}\epsilon^{-\frac{3}{4}}e^{-\frac{7}{4}s_n}+\epsilon^{-\frac{3}{10}}e^{-\frac{3}{2}s_n})M^{25}\epsilon^{-\frac{1}{2}}e^{-\frac{3}{2}s_n}+M^{50}\epsilon^{-1}e^{-3s_n}\big)\big]\\
&\leq \frac{1}{4}e^{-\frac{1}{2}}M^{22}\epsilon^{-\frac{1}{2}}e^{-\frac{3}{2}s_n}+\frac{1}{4}M^{15}\epsilon^{-\frac{1}{4}}e^{-\frac{7}{4}s_n}+\frac{1}{4}\epsilon^\frac{1}{5}e^{-\frac{3}{2}s_n}\\
% +\frac{1}{8}eM^{94}\epsilon^{-\frac{1}{2}}e^{-\frac{5}{2}s_n}+\frac{1}{4}eM^{79}\epsilon^{-\frac{1}{20}}e^{-\frac{9}{4}s_n}\\
% +\frac{1}{8}eM^{64}\epsilon^\frac{2}{5}e^{-2s_n}+\frac{1}{4}eM^{104}\epsilon^{-\frac{1}{4}}e^{-\frac{9}{4}s_n}+\frac{1}{4}eM^{89}\epsilon^\frac{1}{5}e^{-2s_n}+\frac{1}{8}eM^{114}e^{-2s_n}\\
&\leq \frac{3}{4}M^{25}\epsilon^{-\frac{1}{2}}e^{-\frac{3}{2}s_n}.
\end{align*}

From the size of each $B_n$, we see that $\{\alpha_n\}$ and $\{\beta_n\}$ are Cauchy sequences, so our targeting values of $\alpha,\,\beta$ are \begin{align*}
     \alpha_*:=\lim_{n\to\infty}\alpha_n,\qquad\beta_*:=\lim_{n\to\infty}\beta_n
\end{align*}
respectively. Moreover, since for each $\alpha_n,\,\beta_n$, 
\begin{align*}
    \lim_{s\to+\infty}\partial_X^2U_{\alpha_n,\beta_n}(0,s)=\lim_{s\to+\infty}\partial_X^3U_{\alpha_n,\beta_n}(0,s)=0
\end{align*}
by \eqref{eq:assumptionx=0}, and we also have \eqref{eq:inductassumption}, we must have
\begin{align*}
    \lim_{s\to+\infty}\partial_X^2U_{\alpha_*,\beta_*}(0,s)=\lim_{s\to+\infty}\partial_X^3U_{\alpha_*,\beta_*}(0,s)=0.
\end{align*}

\section{Proof of the main theorem}\label{sec:proof}
\begin{proof}[Proof of Theorem \ref{thm:selfsimilar}]
The bootstrap assumptions and iterations to find $\alpha,\,\beta$ are done in Section \ref{sec:closure}-\ref{sec:findpara}. It remains to prove the pointwise convergence of $U$. The convergence at $X=0$ is trivial due to \eqref{eq:constraint}. 

We first show that $\nu:=\lim_{s\to+\infty}\partial_X^5U(0,s)$ exists. Indeed, by the fundamental theorem of calculus, 
\begin{align*}
    \partial_X^5U(0,s)=\partial_X^5U(0,-\log\epsilon)+\int_{-\log\epsilon}^s\partial_s\partial_X^5U(0,s')\,ds'.
\end{align*}
By the proof in Proposition \ref{prop:5xtileU}, $|\partial_s\partial_X^5U(0,s)|=|\partial_s\partial_X^5\widetilde{U}(0,s)|\leq C\epsilon^\frac{1}{5}e^{-\frac{3}{4}s}$. In particular, 
\begin{align*}
    \int_{-\log\epsilon}^{\infty}|\partial_s\partial_X^5U(0,s')|\,ds'<\infty.
\end{align*}
Hence, we have a well-defined limit
\begin{align*}
    \nu=\lim_{s\to\infty}\partial_X^5U(0,-\log\epsilon)+\int_{-\log\epsilon}^\infty\partial_s\partial_X^5U(0,s')\,ds'.
\end{align*}
By our choice of $U(\cdot,-\log\epsilon)$ given by \eqref{eq:initialdata}, we have 
\begin{align*}
    \partial_X^5U(0,-\log\epsilon)=U_2^{(5)}(0)=120.
\end{align*}
And by \eqref{eq:assumptionx=0}, $|\nu-120|\leq \epsilon^\frac{1}{2}$.

Next, we compute the derivatives of $U_2^\nu$:
\begin{align*}
    \frac{d^n}{dX^n}U_2^\nu(X)=\Big(\frac{\nu}{120}\Big)^{\frac{1}{4}(n-1)}U_2^{(n)}\bigg(\Big(\frac{\nu}{120}\Big)^\frac{1}{4}X\bigg),\qquad n=0,1,2,...
\end{align*}
In particular, by Lemma \ref{lem:U2},
\begin{gather*}
    \frac{d^n}{dX^n}U_2^\nu(0)=0,\qquad n=0,2,3,4,\\
    \frac{d}{dX}U_2^\nu(0)=-1,\qquad\frac{d^5}{dX^5}U_2^\nu(0)=\nu.
\end{gather*}
Let $\widetilde{U}^\nu:=U-U_2^\nu$. The Taylor expansion of $\widetilde{U}^\nu(X,s)$ at $X=0$ is
\begin{align*}
    \widetilde{U}^\nu(X,s)=\frac{1}{2}\partial_X^2U(0,s)X^2+\frac{1}{6}\partial_X^3U(0,s)X^3+\frac{1}{120}\big(\partial_X^5U(0,s)-\nu\big)X^5+\frac{1}{720}\partial_X^6\widetilde{U}^\nu(X',s)X^6
\end{align*}
for some $X'$ between 0 and $X$. By \eqref{eq:assumptionx=0},
\begin{align*}
    |\widetilde{U}^\nu(X,s))|\leq \frac{1}{2}\epsilon^\frac{1}{10}e^{-\frac{3}{4}s}X^2+\frac{1}{6}M^{27}e^{-s}|X|^3+\frac{1}{120}|\partial_X^5U(0,s)-\nu||X|^5+CM^{36}X^6,
\end{align*}
where we used $\|\partial_X^6U(\cdot,s)\|_{L^\infty}\leq M^{36}$ and \begin{align*}
    \Big\|\frac{d^4}{dX^4}U_2^\nu\Big\|_{L^\infty}=\Big(\frac{\nu}{120}\Big)^\frac{5}{4}\|U_2^{(6)}\|_{L^\infty}\leq C\Big(\frac{\nu}{120}\Big)^\frac{5}{4}\sim O(1).
\end{align*}
For $X_0\neq 0$ close to 0, fix a small $\delta>0$ such that $\delta<\frac{1}{2}M^{36}X^6_0$. Due to the decay in $s$, there exists a sufficiently large $s_0=s_0(X_0,\delta)$ such that for all $s\geq s_0$, 
\begin{align*}
    |\widetilde{U}^\nu(X_0,s)|\leq CM^{36}X_0^6+\delta.
\end{align*}
Now we need to find an equation for $\widetilde{U}^\nu$. First, note that $U_2^\nu$ also satisfies the self-similar Burgers equation \eqref{eq:U2}. Second, we can rewrite  \eqref{eq:ansatz} as (the left hand side without $\partial_s$ matches the self-similar Burgers equation)
\begin{align*}
    \Big(\partial_s-\frac{1}{4}\Big)U+\Big(U+\frac{5}{4}X\Big)\partial_XU=F_U-\frac{e^{\frac{1}{4}s}(\kappa-\dot{\xi})}{1-\dot{\tau}}\partial_XU-\frac{\dot{\tau}}{1-\dot{\tau}}U\partial_XU.
\end{align*}
Then we have an equation for $\widetilde{U}^\nu$
\begin{align*}
    \Big(\partial_s-\frac{1}{4}+\frac{d}{dX}U_2^\nu\Big)\widetilde{U}^\nu+\Big(U+\frac{5}{4}X\Big)\partial_X\widetilde{U}^\nu=F_U-\frac{e^{\frac{1}{4}s}(\kappa-\dot{\xi})}{1-\dot{\tau}}\partial_XU-\frac{\dot{\tau}}{1-\dot{\tau}}U\partial_XU.
\end{align*}
We denote the right hand side as $F_{\tilde{U}^\nu}$. We claim that 
\begin{align*}
    \int_{-\log\epsilon}^\infty\|F_{\tilde{U}^\nu}(\cdot,s')\|_{L^\infty}\,ds'<\infty.
\end{align*}
Indeed, by \eqref{eq:forcingU}, 
\begin{align*}
    \|F_{\tilde{U}^\nu}(\cdot,s)\|_{L^\infty}\lesssim Me^{-\frac{5}{8}s}+\epsilon^\frac{1}{5}e^{-\frac{3}{4}s}+\epsilon^\frac{1}{5}Me^{-\frac{1}{2}s},
\end{align*}
which is $s$-integrable. 

Let $\Psi(X_0,\cdot):[s_0,\infty)\to\mathbb{R}$ be the Lagrangian trajectory of $\widetilde{U}^\nu$, i.e.
\begin{align*}
    \frac{d}{ds}\Psi(X_0,s)&=\Big(U+\frac{3}{2}X\Big)\circ\Psi(X_0,s)\\
    \Psi(X,s_0)&=X_0.
\end{align*}
By a similar argument as in Lemma \ref{lem:lowerLagrangian} (using the mean value theorem), we can see that $\Psi$ is repelling for all $X_0\neq 0$
\begin{align}
    |\Psi(X_0,s)|\geq |X_0|e^{\frac{1}{4}(s-s_0)}.\label{eq:repelling}
\end{align}
Let $G(X,s)=e^{-\frac{5}{4}(s-s_0)}\widetilde{U}^\nu(X,s)$, then 
\begin{align*}
    \Big(\frac{d}{ds}+1+\frac{d}{dX}U_2^\nu\Big)G\circ\Psi(X_0,s)=e^{-\frac{5}{4}(s-s_0)}F_{\tilde{U}^\nu}\circ\Psi(X_0,s).
\end{align*}
The damping term $1+\frac{d}{dX}U_2^\nu\geq 0$. Hence, 
\begin{align*}
    |G\circ\Psi(X_0,s)|&\leq |G(X_0,s_0)|+\int_{s_0}^se^{-\frac{5}{4}(s'-s_0)}F_{\tilde{U}^\nu}\circ\Psi(X_0,s')\,ds'\\
    e^{-\frac{5}{4}(s-s_0)}|\widetilde{U}^\nu\circ\Psi(X_0,s)|&\leq |\widetilde{U}^\nu(X_0,s_0)|+\int_{s_0}^se^{-\frac{5}{4}(s'-s_0)}F_{\tilde{U}^\nu}\circ\Psi(X_0,s')\,ds'\\
    &\leq CM^{36}X_0^6+\delta+\delta\leq (C+1)M^{36}X_0^6\\
    |\widetilde{U}^\nu\circ\Psi(X_0,s)|&\leq (C+1)M^{36}e^{\frac{5}{4}(s-s_0)}X_0^6,
\end{align*}
if $s_0$ is sufficiently large. Then for $s_0\leq s\leq s_0-\frac{22}{5}\log|X_0| $, we have
\begin{align*}
    |\widetilde{U}^\nu\circ\Psi(X_0,s)|\leq (C+1)M^{36}|X_0|^\frac{1}{2}.
\end{align*}
For any $X$ between $X_0$ and $\Psi(X_0,s_0-\frac{22}{5}\log|X_0|)$, there exists $s_0<s<s_0-\frac{22}{5}\log|X_0|$ such that $X=\Psi(X_0,s)$, so for such $(X,s)$, we have
\begin{align*}
    |\widetilde{U}^\nu(X,s)|\leq (C+1)M^{36}|X_0|^\frac{1}{2}.
\end{align*}
By \eqref{eq:repelling}, this will cover at least all $X$ such that
\begin{align*}
    |X_0|\leq |X|\leq |X_0|^{-\frac{1}{10}}.
\end{align*}
So if we take the limit $s_0\to\infty$, then for all $X$ such that $|X_0|\leq |X|\leq |X_0|^{-\frac{1}{10}}$,
\begin{align*}
    \limsup_{s\to\infty}\widetilde{U}^\nu(X,s)\leq (C+1)M^{36}|X_0|^\frac{1}{2}.
\end{align*}
Finally, sending $X_0\to 0$, we get 
\begin{align*}
    \limsup_{s\to\infty}\widetilde{U}^\nu(X,s)=0
\end{align*}
for all $X\neq 0$.
\end{proof}

\begin{proof}[Proof of Theorem \ref{thm:main}]
Since $\|\partial_X^nU(\cdot,s)\|_{L^2}$ are uniformly bounded in $s$ for all $n=1,...,9$, by \eqref{eq:assumptionL2}, $\|\partial_x^nu(\cdot,t)\|_{L^2}$ remains finite for all $t<T_*$. And $\|u(\cdot,t)\|_{L^2}=\|u_0\|_{L^2}$. So $\|u(\cdot,t)\|_{H^9}<\infty$ for all $t\in[-\epsilon,T_*)$. By the standard well-posedness theorem, the solution $u\in\mathcal{C}\big([-\epsilon,T_*);H^9(\mathbb{R})\big)$ is unique. Moreover, we have the Sobolev embedding $H^9(\mathbb{R})\subset \mathcal{C}^8(\mathbb{R})$.

Blowup time: $|T_*|\leq 2\epsilon^\frac{39}{20}$ is proved in Section \ref{sec:modulation}.

Blowup location: by the fundamental theorem of calculus, for all $t\in[-\epsilon,T_*]$
\begin{align*}
    |\xi(t)|\leq |\xi(-\epsilon)|+\int_{-\epsilon}^{T_*}|\dot{\xi}(t')|\,dt'\leq (2\epsilon^\frac{39}{20}+\epsilon)\cdot 2M\leq \frac{5}{2}M\epsilon.
\end{align*}
Hence, $x_*=\xi(T_*)\leq 3M\epsilon$.

Solution is bounded: $\|U(\cdot,s)+e^{\frac{1}{4}s}\kappa\|_{L^\infty}\leq Me^{\frac{1}{4}s}$ is equivalent to $\|u(\cdot,t)\|_{L^\infty}\leq M$.

Shock formation and rate: since
\begin{align*}
    \partial_xu(x,t)=\big(\tau(t)-t\big)^{\frac{1}{4}-\frac{5}{4}}\partial_XU\bigg(\frac{x-\xi(t)}{\big(\tau(t)-t)^\frac{5}{4}},s\bigg)=\frac{\partial_XU(X,s)}{\tau(t)-t},
\end{align*}
we see that 
\begin{align*}
    \partial_xu\big(\xi(t),t\big)=\frac{\partial_XU(0,s)}{\tau(t)-t}=-\frac{1}{\tau(t)-t}.
\end{align*}
We claim that for all $t\in[-\epsilon,T_*)$,
\begin{align*}
    \frac{1}{2}\leq \frac{\tau(t)-t}{T_*-t}\leq 2.
\end{align*}
Indeed, this is equivalent to
\begin{align*}
    \begin{cases}
    T_*-t\leq 2\tau(t)-2t,\\
    \tau(t)-t\leq 2T_*-2t,
    \end{cases}
    \qquad\Longleftrightarrow\qquad
    \begin{cases}
    2\tau(t)-t\geq T_*,\\
    \tau(t)+t\leq 2T_*,
    \end{cases}
\end{align*}
which is true since 
\begin{align*}
    \begin{cases}
    \frac{d}{dt}\big(2\tau(t)-t\big)=2\dot{\tau}-1\leq 0,\\
    \frac{d}{dt}\big(\tau(t)+t\big)=\dot{\tau}+1\geq 0,
    \end{cases}
    \qquad\text{and}\qquad\begin{cases}
    2\tau(T_*)-T_*=T_*,\\
    \tau(T_*)+T_*=2T_*.
    \end{cases}
\end{align*}
Hence, as $t\to T_*$, $\partial_xu\big(\xi(t),t\big)\to-\infty$, so $\partial_xu$ blows up at $x_*=\xi(T_*)$. Moreover, we have the following rate:
\begin{align*}
    \frac{1}{2(T_*-t)}\leq |\partial_xu\big(\xi(t),t\big)|\leq \frac{2}{T_*-t}.
\end{align*}
For $x\neq x_*$, if $|x-x_*|>\frac{1}{2}$, then there exists $t_1\in[-\epsilon,T_*)$ such that $|x-\xi(t)|\geq \frac{1}{2}$ for all $t\in [t_1,T_*]$. In the self-similar coordinates, this means
\begin{align*}
    |X|\geq \frac{1}{2}\big(\tau(t)-t\big)^{-\frac{5}{4}}\geq \frac{1}{2}e^{\frac{5}{4}s}.
\end{align*}
Hence, by \eqref{eq:assumption1xfar},
\begin{align*}
    |\partial_xu(x,t)|=e^s\partial_XU(X,s)\leq 4\qquad\forall\,t\in[t_1,T_*].
\end{align*}
If $|x-x_*|< \frac{1}{2}$, then there exists $t_2\in[-\epsilon,T_*)$ such that for all $t\in[t_2,T_*)$, $
|x-\xi(t)|\leq \frac{1}{2}$ and $|x-\xi(t)|\geq \frac{1}{2}|x-x_*|>0$, so the corresponding $\frac{1}{2}|x-x_*|e^{\frac{5}{4}s}\leq |X|\leq \frac{1}{2}e^{\frac{5}{4}s}$.
By choosing a larger $t_2$
if necessary, we may also assume $|X|\geq 1$. Hence by \eqref{eq:assumption1xmiddle} and \eqref{eq:U2away0}, we have
\begin{align*}
    |\partial_xu(x,t)|=\frac{|\partial_XU(X,s)|}{\tau(t)-t}\leq \frac{1}{\tau(t)-t}|X|^{-\frac{4}{5}}\leq 2^\frac{4}{5}|x-x_*|^{-\frac{4}{5}}.
\end{align*}

Shock profile is a cusp: in fact, by \eqref{eq:U2away0} and $\eqref{eq:assumption1xmiddle}$, for $|X|\leq \frac{1}{2}e^{\frac{5}{4}s}$,
\begin{align*}
    \partial_XU(X,s)\sim (1+X^4)^{-\frac{1}{5}}.
\end{align*}
So for $x\neq x_*$ such that $|x-x_*|<\frac{1}{2}$, 
\begin{align*}
    |\partial_xu(x,t)|\sim |x-x_*|^{-\frac{4}{5}}\qquad\text{as }t\to T_*.
\end{align*}
This indicates that $u(\cdot,T_*)\in\mathcal{C}^\frac{1}{5}(\mathbb{R})$ and it has a cusp singularity (similar to the cusp of $|x|^\frac{1}{5}$ at $x=0$) at $(x_*,T_*)$.
\end{proof}

\begin{proof}[Proof of Corollary \ref{cor:codimension}]
First we note that we can replace $\leq$ by $<$ in the proof. Then we note that $M$, $\epsilon$ can be taken in an open set of values. For the initial data $u_0$ described by Section \ref{sec:initialdata},
\begin{gather*}
    u_0(0)=\kappa_0,\qquad u_0'(0)=-\frac{1}{\epsilon}=\min_xu_0'(x)=-\|u'\|_{L^\infty},\\
    \xi_0=0,\qquad u_0^{(4)}(0)=0,\qquad u_0^{(5)}(0)=120\epsilon^{-6},\qquad |u_0^{(6)}(0)|\leq \epsilon^{-\frac{25}{4}}.
\end{gather*}
We note that the minimum initial slope attaining at $x=0$ is not necessary (it is never used in the proof), as long as $u_0'(0)<0$ of order $\epsilon^{-1}$. For a suitable small perturbation $v_0$ of $u_0$, consider the Taylor expansion for $x\approx 0$
\begin{align*}
    v_0^{(4)}(x)=v_0^{(4)}(0)+\big(v_0^{(5)}(0)-120\epsilon^{-6}\big)x+120\epsilon^{-6}x+\frac{1}{2}v_0^{(6)}(0)x^2.
\end{align*}
If $v_0^{(6)}(0)\sim O(\epsilon^{-\frac{25}{4}})$, and $v_0^{(4)}(0)$ and $v_0^{(5)}-120\epsilon^{-6}$ are sufficiently small, we can find $|x_0|\ll 1$ such that $v_0^{(4)}(x_0)=0,\,v_0^{(5)}(0)>0$ with order $\epsilon^{-6}$ and $v_0'(x_0)<0$ with order $\epsilon^{-1}$. Hence, we can set
\begin{align*}
    \xi_0=x_0,\qquad\kappa_0=v_0(x_0),
\end{align*}
and apply the coordinate translation $x\mapsto x+x_0$. To match the 5th order derivative, apply the rescaling 
\begin{align*}
    \tilde{v}(x,t)=\bigg(\frac{120\epsilon^{-6}}{v_0^{(5)}(0)}\bigg)^{-\frac{1}{4}}v\bigg(\Big(\frac{120\epsilon^{-6}}{v_0^{(5)}(0)}\Big)^\frac{1}{4}x,t\bigg),
\end{align*}
so that we can find the corresponding $\widehat{V}_0$ according to \eqref{eq:initialdata}.
Finally, we see that $\widehat{U}_0$ satisfying the hypothesis in Section \ref{sec:initialdata}) can be taken from an open neighbourhood in $H^9(\mathbb{R})$ such that for every $\widehat{U}_0$ in this neighbourhood, there exist unique $\alpha,\beta\in\mathbb{R}$ such that the conclusion in Theorem \ref{thm:main} holds for initial data \eqref{eq:initialphysical}.
\end{proof}

\appendix
\section{Interpolation lemmas}
\begin{lemma}[Gagliardo-Nirenberg-Sobolev interpolation]
Let $f:\mathbb{R}^d\to\mathbb{R}$, and let $1\leq q,\,r\leq \infty$, $j,\,m\in \mathbb{N}$ (including 0) and $j/m\leq \alpha\leq 1$ be such that
\begin{align*}
\frac{1}{p}=\frac{j}{d}+\alpha(\frac{1}{r}-\frac{m}{d})+\frac{1-\alpha}{q},
\end{align*}
then
\begin{align*}
\|\partial^j f\|_{L^p}\lesssim \|\partial^mf\|_{L^r}^\alpha\|f\|_{L^q}^{1-\alpha},
\end{align*}
with two exceptions
\begin{enumerate}
\item If $j=0,\ mr<d$ and $q=\infty$, then we assume additionally that either $f$ tends to 0 at infinity or that $f\in L^{q'}$ for some $q'<\infty$.
\item If $1<r<\infty$ and $m-j-d/r\in\mathbb{N}$, then we also assume that $\alpha<1$. 
\end{enumerate}\label{lem:GNSinterpolation}
\end{lemma}

\begin{lemma}[Sobolev interpolation]
As a special case, when $p=q=r=2$, we can make the constant to be 1:
\begin{align*}
\|\partial^jf\|_{L^2}\leq \|\partial^mf\|_{L^2}^\alpha\|f\|_{L^2}^{1-\alpha}
\end{align*}
where $\alpha=j/m$.\label{lem:Sinterpolation}
\end{lemma}

\end{document}